\pdfoutput=1
\documentclass[11pt]{article}
\usepackage[T1,OT1]{fontenc}
\usepackage{graphicx}
\usepackage{multirow}
\usepackage{amsmath,amssymb,amsfonts}
\usepackage{amsthm}
\usepackage{mathrsfs}
\usepackage{mathtools}
\usepackage[title]{appendix}
\usepackage{xcolor}
\usepackage[margin=1in]{geometry}
\usepackage{authblk}
\usepackage[numbers,sort&compress]{natbib}
\usepackage{hyperref}
\usepackage{bookmark}
\usepackage{textcomp}
\newcommand{\Lacka}{{\fontencoding{T1}\selectfont {\L}\k{a}cka}}
\newcommand{\noopsort}[1]{}
\usepackage{manyfoot}
\usepackage{booktabs}
\usepackage{algorithm}
\usepackage{algorithmicx}
\usepackage{algpseudocode}
\usepackage{listings}

\providecommand{\Newlabel}[2]{}

\theoremstyle{plain}
\newtheorem{theorem}{Theorem}[section]

\newtheorem{corollary}[theorem]{Corollary}
\newtheorem{lemma}[theorem]{Lemma}

\newtheorem{theoremx}{Theorem}

\theoremstyle{remark}
\newtheorem{remark}[theorem]{Remark}
\newtheorem*{remark*}{Remark}
\newtheorem*{question}{Question}
\newtheorem*{claim}{Claim}

\theoremstyle{definition}
\newtheorem{definition}[theorem]{Definition}

\newcommand*{\vphi}{\varphi}

\newcommand*{\cE}{\mathcal{E}}
\newcommand*{\cK}{\mathcal{K}}
\newcommand*{\cL}{\mathcal{L}}
\newcommand*{\cF}{\mathcal{F}}
\newcommand*{\F}{\mathfrak{F}}

\newcommand*{\M}{\mathcal{M}}

\newcommand*{\Rb}{\mathbb{R}}
\newcommand*{\Zb}{\mathbb{Z}}
\newcommand*{\Nb}{\mathbb{N}}

\newcommand*{\dif}{\mathop{}\!\mathrm{d}}
\newcommand*{\Int}{\mathrm{Int}}

\newcommand*{\Merg}{\mathcal{M}_{\mathrm{erg}}}

\newcommand*{\htop}{h_{\mathrm{top}}}
\newcommand*{\ve}{\varepsilon}
\newcommand*{\veh}{\ve^{\mathrm{h}}}
\newcommand*{\veH}{\ve^{\mathrm{H}}}
\newcommand*{\ved}{\ve^{\mathrm{d}}}
\newcommand*{\Per}{\mathrm{Per}}
\newcommand*{\Orb}{\mathrm{Orb}}

\newcommand*{\cov}{\mathrm{cov}}
\newcommand*{\diam}{\mathrm{diam}}

\newcommand*{\supp}{\mathrm{supp}}
\newcommand*{\spa}{\mathrm{span}}
\newcommand*{\dima}{\dim_{\mathrm{a}}}
\newcommand*{\Aff}{\mathrm{Aff}}
\newcommand*{\Relint}{\mathrm{Relint}}

\makeatletter
\let\orgdescriptionlabel\descriptionlabel
\renewcommand*{\descriptionlabel}[1]{%
  \phantomsection
  \let\orglabel\label
  \let\label\@gobble
  \protected@edef\@currentlabel{#1}%
  \protected@edef\@currentlabelname{#1}%
  \let\label\orglabel
  \orgdescriptionlabel{#1}%
}
\makeatother

\numberwithin{equation}{section}

\begin{document}

\title{Conditional entropy realization and approximation by uniquely ergodic measures}

\author[1]{Xiaobo Hou\thanks{\texttt{xiaobohou@dlut.edu.cn}}}
\author[2]{Wanshan Lin\thanks{\texttt{21110180014@m.fudan.edu.cn}}}
\author[2]{Xueting Tian\thanks{Corresponding author: \texttt{xuetingtian@fudan.edu.cn}}}
\author[3]{Yi Yuan\thanks{\texttt{0120250109@xhu.edu.cn}}}
\author[4]{Xutong Zhao\thanks{\texttt{xutongzhao@sjtu.edu.cn}}}

\affil[1]{School of Mathematical Sciences, Dalian University of Technology, Dalian 116024, People's Republic of China}
\affil[2]{School of Mathematical Sciences, Fudan University, Shanghai 200433, People's Republic of China}
\affil[3]{School of Science, Xihua University, Chengdu 610039, People's Republic of China}
\affil[4]{School of Mathematical Sciences, CMA-Shanghai, Shanghai Jiao Tong University, Shanghai 200240, People's Republic of China}

\date{}

\maketitle

\begin{abstract}
    This paper studies conditional entropy realization and weak* approximation by uniquely ergodic measures with compact support. We prove that, after fixing an admissible potential average and an entropy strictly below the entropy supremum over the corresponding average fiber, every invariant measure satisfying these two exact constraints can be approximated weakly* by uniquely ergodic measures with compact support satisfying the same constraints. Each approximating measure is the unique invariant measure on its minimal support, whose topological entropy equals the prescribed metric entropy. This result holds for three broad classes of systems: topologically expanding maps (including topologically Anosov systems), transitive countable Markov shifts, and symbolic systems with non-uniform structure. The proof uses a nested multi-horseshoe construction, with separate arguments addressing non-invertibility, non-compactness and non-uniformity.
\end{abstract}



\section{Introduction}\label{secIntroduction}
Ergodic measures occupy a central role in describing the long-term statistical behavior of dynamical systems. A fundamental line of inquiry in ergodic theory concerns the abundance and distribution of these measures within the space of invariant measures. Researchers investigated the distribution and abundance of ergodic measures from the perspectives of density, entropy and multifractal analysis. From a topological perspective, Sigmund \cite{sigmund1974dynamical} established that ergodic measures are dense in the space of invariant measures for systems satisfying the periodic specification property, including Axiom A diffeomorphisms restricted to mixing basic sets, topologically mixing subshifts of finite type and mixing interval maps. These results were subsequently extended by Gelfert and Kwietniak \cite{gelfert2018density} to systems satisfying their closeability and linkability assumptions, including $S$-gap shifts and $\beta$-shifts. From the perspective of entropy, Eizenberg, Kifer and Weiss \cite{eizenberg1994large}, and later Pfister and Sullivan \cite{pfister2005large} demonstrated that under conditions like the specification property or approximate product property, if the entropy map is upper semi-continuous, then any $f$-invariant measure can be approximated by ergodic measures in both the weak* topology and entropy. Building on this foundation, Dong, Hou and Tian \cite{dong2024abundance} recently provided a comprehensive analysis of the abundance of ergodic measures in hyperbolic systems and certain partially hyperbolic systems with one-dimensional center bundle via multifractal analysis.

Minimal sets and minimal systems are also crucial objects in the study of dynamical systems. Suppose $(X,f)$ is a dynamical system. Given a nonempty $f$-invariant closed subset $Y$ of $X$, we call $Y$ and $(Y, f|_Y)$ a \emph{minimal set} and a \emph{minimal system} respectively if the orbit of $x$ is dense in $Y$ for any $x \in Y$.
At the same time, minimality can be characterized from the point of view of measure theory. Given an $f$-invariant measure $\mu$ on $X$, its \emph{support} is defined by
\[\supp(\mu) = \left\{x \in X: \mu(U) > 0 \mathrm{\ for\ any\ open\ neighborhood\ } U \mathrm{\ of\ } x\right\}.\]
We say $\mu$ is \emph{minimal} if $\supp(\mu)$ is minimal.

While the existence of minimal sets is guaranteed for compact systems by Zorn's lemma, their statistical flexibility may appear constrained. Periodic measures are minimal but invariably carry zero metric entropy. Moreover, Boshernitzan \cite{boshernitzan1984unique} and Cyr and Kra \cite{cyr2019counting} showed that every minimal subshift with linear block growth has only finitely many ergodic measures. Among minimal measures, uniquely ergodic measures are particularly rigid: their supports carry no other invariant probability measure. This rigidity leads to the following question.

\begin{question}
    How abundant are uniquely ergodic measures with compact support in systems exhibiting hyperbolicity?
\end{question}

More specifically, we consider the following hierarchy of questions.
\begin{itemize}
    \item Are uniquely ergodic measures with compact support dense in the space of invariant measures?

    \item Can every submaximal entropy be realized by a uniquely ergodic measure and by the dynamics on its minimal support?

    \item After prescribing both the integral of a continuous function and a compatible submaximal entropy, are uniquely ergodic measures with compact support weak*-dense in the corresponding exact constraint set of invariant measures?
\end{itemize}

Positive-entropy results reveal considerably greater flexibility beyond periodic measures. Grillenberger \cite{grillenberger1972constructions} constructed strictly ergodic systems with prescribed entropy. Li and Oprocha \cite{li2018properties} showed that, for topologically transitive systems with the shadowing property, invariant measures can be approximated weakly* by ergodic measures supported on almost one-to-one extensions of odometers, with entropy converging to any prescribed lower value. Dong and Tian \cite{dong2018differentstatisticalfuturedynamical} obtained related approximation results for topologically expanding systems. Chang, Li, and Wu \cite{chang2024orbit} studied entropy flexibility for subsystems and orbit closures in symbolic systems, while Arbieto, Oprocha and Rego \cite{arbieto2025entropy} established entropy flexibility through strictly ergodic subsystems for several discrete- and continuous-time systems. For full shifts over $\Zb$, Weiss \cite{weiss2000single} proved that
invariant measures supported on uniquely ergodic subsystems are
entropy dense among ergodic invariant measures. Very recently,
\Lacka{} and Mroczek \cite{lacka2026entropy} extended this result to
full shifts over amenable residually finite groups. These results establish entropy realization or weak*--entropy approximation by measures on minimal or uniquely ergodic subsystems, but they do not combine exact entropy and potential-integral constraints with weak* density in the corresponding exact constraint set.

In this article, we extend the aforementioned abundance results for ergodic measures to uniquely ergodic measures with compact support. More precisely, we prove that, after prescribing an admissible potential integral and a compatible submaximal entropy, such measures are weak*-dense in the corresponding exact constraint set of invariant measures. We establish this conditional entropy realization and approximation result for three classes of systems.

\section{Main results}\label{secMainResults}
In this section, we introduce the main results of this article. First, we introduce some notations and definitions. Let $\Nb, \Nb^+, \Zb, \Rb$ denote the set of natural numbers, positive integers, integers, and real numbers, respectively. Let $(X,d)$ be a (compact or non-compact) metric space. Let $C(X)$, $C_b(X)$ and $C_{b,u}(X)$ be the space of continuous functions, bounded continuous functions and bounded uniformly continuous functions on $X$, respectively. When $X$ is compact, the three spaces of functions mentioned above are identical.

Let $f: X \to X$ be a continuous map. We also call $(X,f)$ a dynamical system. For any $x \in X$, denote $\Orb(x) = \left\{f^i x: i \in \Nb\right\}$. We say $x \in X$ is a \emph{periodic point} of $f$ if there exists $p \in \Nb^+$ such that $x = f^p x$. In this case, its orbit is also called a periodic orbit. A system $(X,f)$ is said to be \emph{non-trivial} if $X$ is not a single periodic orbit. Denote by $\M(X)$ the space of probability measures, $\M(f, X)$ the space of $f$-invariant probability measures and $\Merg(f, X)$ the space of $f$-ergodic probability measures respectively.

We denote the cardinality of a set by $\left\lvert \cdot\right\rvert $. We say $\mu \in \M(f,X)$ is \emph{uniquely ergodic} if it is the unique $f$-invariant probability measure supported on $\supp(\mu)$. Every uniquely ergodic measure is ergodic; if its support is compact, then it is also minimal.

In this article, we primarily investigate three types of dynamical systems. The first case assumes $X$ is compact and $f$ is a \emph{topologically expanding} map (\emph{topologically Anosov} homeomorphism), which means that $X$ has infinitely many points, $f$ is positively expansive (expansive) and satisfies the shadowing property. The second case focuses on the dynamics of topologically transitive countable Markov shifts, which in general are not even locally compact. Readers may refer to Section \ref{sec:pre_cms} for details. Finally, the third case considers symbolic systems with non-uniform structure introduced by Climenhaga, Thompson and Yamamoto \cite{climenhaga2017large}. Readers may refer to Section \ref{sec:pre_symbolic} for details.

\begin{definition}
    We say a system satisfies (H) if it belongs to one of the following three cases.
    \begin{description}
        \item[(H1)] $f$ is a transitive topologically expanding continuous map (topologically Anosov homeomorphism) on a compact metric space $(X,d)$.
        \item[(H2)] $(\Sigma, \sigma)$ is a non-trivial topologically transitive one-sided (two-sided) countable Markov shift.
        \item[(H3)] $(X,\sigma)$ is a one-sided (two-sided) subshift on a finite alphabet with language $\mathcal{L}$ and there exists $\mathcal{G} \subset \mathcal{L}$ such that $\mathcal{G}$ satisfies $(W)$-specification and $\mathcal{L}$ is edit approachable by $\mathcal{G}$.
    \end{description}
\end{definition}

For simplicity, we uniformly use $(X, f)$ to denote all three classes of systems in this section. Throughout the paper, the phase spaces appearing in {\rm (H1)--(H3)}
are equipped with compatible metrics bounded by one; namely, we assume $\diam(X)\leq 1$. In the compact case, this causes no loss of generality, since the
original metric may be rescaled. For countable Markov shifts and
finite-alphabet subshifts, we use the bounded symbolic metrics specified
in the corresponding sections.

Katok \cite{katok1980lyapunov} established a milestone result: every $C^{1+\alpha}$ diffeomorphism $f$ in dimension two possesses horseshoes with large entropies. This implies that such systems admit ergodic measures with arbitrary intermediate metric entropies, meaning the entropy spectrum of ergodic measures includes $[0, \htop(f))$ and equals that of all invariant measures. Katok conjectured this fundamental result holds in any dimension, an open problem known as Katok's conjecture or the intermediate entropy problem.

Li and Oprocha \cite{li2018properties} proved that, for topologically transitive systems with the shadowing property, invariant measures can be approximated weakly* together with their entropy by ergodic measures supported on minimal systems. For the classes considered here, we strengthen this approximation to exact entropy realization by the unique invariant measures of compact minimal subsystems; moreover, the metric entropy of each such measure equals the topological entropy of its support. Consequently, the entropy spectrum of uniquely ergodic measures with compact support contains $[0,\htop(f))$.

\begin{theorem}\label{thm:2.2}
    Suppose $(X,f)$ satisfies (H). Then for any $h \in \left[0,\htop(f,X)\right)$, there exists a compact $f$-invariant set $Y \subseteq X$ such that $(Y,f|_Y)$ is minimal and uniquely ergodic, and
    \[h_\mu(f)=\htop(f,Y)=h,\]
    where $\mu$ is the unique $f$-invariant measure on $Y$.
\end{theorem}
\begin{remark}
    In fact, for systems satisfying (H1), we do not require the assumption of topological transitivity by utilizing the topological spectral decomposition theorem (see \cite[Theorem 3.4.4]{aoki1980topological}).
\end{remark}

In particular, Theorem \ref{thm:2.2} implies the \emph{minimal-intermediate-entropy property}, namely that every submaximal entropy is realized by a minimal measure. It further realizes the same value as the topological entropy of a compact minimal uniquely ergodic subsystem. For systems satisfying (H1) or (H3), the maximal-entropy endpoint can be included in the following entropy-flexibility statement.

\begin{corollary}\label{cor:2.3}
    Suppose $(X,f)$ satisfies (H1) or (H3). Then for any $h \in \left[0,\htop(f,X)\right]$, there exists an ergodic measure $\mu$ and a compact and invariant subset $Y \subseteq X$ such that
    \[\htop(f, Y) = h_\mu(f) = h.\]
\end{corollary}
\begin{remark}
    Here we exclude (H2), as a measure of maximal entropy may not exist for countable Markov shifts, see \cite{gurevich1970shift}.
\end{remark}

The result of the above corollary is also called the \emph{entropy flexibility}. Recently, Chang, Li, and Wu \cite{chang2024orbit} proved the minimal-intermediate-entropy property and the entropy flexibility for symbolic systems which satisfy a weak form of specification property and certain entropy gap condition. Arbieto, Oprocha and Rego \cite{arbieto2025entropy} proved similar results for suspension flows of subshifts of finite type.

Given $\varphi \in C_{b,u}(X)$, denote
\[L_\vphi(X, f) = \left[\inf_{\mu \in \M(f, X)} \int \vphi \dif \mu, \sup_{\mu \in \M(f, X)} \int \vphi \dif \mu\right].\]
When no confusion can arise, we often omit $(X,f)$ and simply write $L_\varphi$ for $L_\vphi(X, f)$ for convenience. Denote
\[\mathrm{Int} L_\vphi(X,f) = \left(\inf_{\mu \in \M(f, X)} \int \vphi \dif \mu, \sup_{\mu \in \M(f, X)} \int \vphi \dif \mu \right).\]
For any $a \in L_\varphi$, define the level set
\[R_\varphi(a) = \left\{x \in X: \lim_{n \to \infty} \frac{1}{n}\sum_{i=0}^{n-1}\varphi\left(f^i x\right) = a \right\}.\]

When $X$ is compact, given $\mu \in \M(f,X)$, we say $\mu$ is an \emph{equilibrium state} for $\varphi$ if $P(\varphi) = h_\mu(f) + \int\varphi\dif\mu$ where $P(\varphi)$ is the topological pressure of $f$ with respect to $\vphi$. We also say $\varphi$ has an equilibrium state. Denote
\[E(X,f) = \left\{\varphi \in C(X): \varphi \mathrm{\ has\ a\ unique\ equilibrium\ state}\right\}.\]

Barreira and Saussol \cite{barreira2001variational} proved that, for systems with upper semi-continuous entropy functions, if $\vphi \in C(X)$ satisfies $\mathrm{Int} L_\vphi \ne \emptyset$ and $\spa\left\{\varphi, 1\right\} \subset E(X,f)$, then for any $a \in \mathrm{Int} L_\vphi$, one has
\[\htop\left(f, R_\varphi(a)\right) = \sup\left\{h_\mu(f):\mu \in \M(f,X), \int \varphi \dif \mu = a\right\}.\]
This is called the \emph{conditional variational principle}, which gives a partial description of ergodic measures from the viewpoints of multifractal analysis and entropy simultaneously. It was later generalized to systems with the specification property by Takens \cite{takens2003variational} and to systems with the $g$-almost product property by Pfister and Sullivan \cite{pfister2007topological}. For the systems considered below, we refine this principle by showing that the supremum is unchanged when restricted to uniquely ergodic measures with compact support.

\begin{theorem}\label{thm:2.4}
    Suppose $(X,f)$ satisfies (H1) or (H3). If $\vphi \in C(X)$ satisfies $\mathrm{Int} L_\vphi \ne \emptyset$, then for any $a \in \mathrm{Int} L_\vphi$ one has
    \begin{align*}
        \htop\left(f, R_\varphi(a)\right)
        = \sup\{h_\mu(f):{} & \mu \in \M(f,X) \mathrm{\ is\ uniquely\ ergodic} \\
                             & \mathrm{\ with\ compact\ support},\ \int \varphi \dif \mu = a\}.
    \end{align*}
\end{theorem}
\begin{remark}
    The set of continuous functions $\varphi$ satisfying $\Int L_\varphi \ne \emptyset$ is open and dense in $C(X)$ for systems with $\left\lvert \M(f, X)\right\rvert > 1$, see \cite[Proposition 2.2]{dong2024abundance}.
\end{remark}

We next show that uniquely ergodic measures realize the relative interior of the Lyapunov spectrum. First, we use the following formulation of the concept of multi-average conformal diffeomorphisms introduced by Dong, Hou and Tian \cite{dong2024abundance}. Let $M$ be a compact Riemannian manifold and $f:M \to M$ be a $C^1$ diffeomorphism. Given an $f$-invariant measure $\mu$, for $\mu$-a.e. $x \in M$, denote by
\[\chi_1(x) \geq \chi_2(x) \geq \dots \geq \chi_{\dim M}(x)\]
the Lyapunov exponents at $x$. For any $\mu \in \M(f,M)$ and any $1 \leq i \leq \dim M$, denote $\chi_i(\mu) = \int\chi_i(x)\dif \mu$. We say an Anosov diffeomorphism $f$ is \emph{multi-average conformal}, if there are $t_u, t_s \in \Nb^+$, $\left\{d_i\right\}_{i=1}^{t_u+t_s} \subset \Nb^+$ and $E^1,\dots,E^{t_u+t_s} \subset TM$ such that
\begin{enumerate}
    \item $\sum_{j=1}^{t_u} d_j = \dim E^u$, $\sum_{j=t_u+1}^{t_u+t_s} d_j = \dim E^s$ and $\dim E^j = d_j$ for any $1 \le j \le t_u+t_s$;
    \item $D_x f\left(E_x^j\right) = E_{fx}^j$ for any $x \in M$ and any $1 \leq j \leq t_u+t_s$.
    \item $x \to E_x^j$ is continuous for any $1 \leq j \leq t_u+t_s$.
    \item $E_x^u = E_x^1 \oplus \dots \oplus E_x^{t_u}$, $E_x^s = E_x^{t_u + 1} \oplus \dots \oplus E_x^{t_u + t_s}$ for any $x \in M$.
    \item Let $s_0:=0$ and $s_j:=\sum_{k=1}^j d_k$ for
          $1\leq j\leq t_u+t_s$. For every $\mu\in\Merg(f,M)$ and every
          $1\leq j\leq t_u+t_s$, the restricted cocycle
          $Df|_{E^j}$ has a unique Lyapunov exponent
          $\lambda_j(\mu)$ of multiplicity $d_j$. Moreover,
          \[
              \lambda_1(\mu)\geq\cdots\geq\lambda_{t_u}(\mu)>0>
              \lambda_{t_u+1}(\mu)\geq\cdots\geq
              \lambda_{t_u+t_s}(\mu).
          \]
          Equivalently,
          \[
              \chi_{s_{j-1}+1}(\mu)=\cdots=\chi_{s_j}(\mu)
              =\lambda_j(\mu)
          \]
          for every $1\leq j\leq t_u+t_s$.
\end{enumerate}

Given $Y \subseteq \Rb^m$ for some $m \in \Nb^+$, denote the \emph{relative interior} of $Y$ by $\Relint(Y)$ (see the definition preceding Theorem \ref{thm:relint} in Section \ref{sec:cmie}).

\begin{theorem}\label{thm:2.5}
    Suppose $M$ is a compact Riemannian manifold and $f:M \to M$ is a $C^1$ topologically transitive multi-average conformal Anosov diffeomorphism. Then
    \begin{align*}
          & \mathrm{Relint} \left\{\left(\chi_1(\mu),\dots,\chi_{\dim M}(\mu)\right):\mu \in \M(f,M)\right\}                       \\
        = & \mathrm{Relint} \left\{\left(\chi_1(\mu),\dots,\chi_{\dim M}(\mu)\right):\mu \in \M(f,M) \mathrm{\ is\ uniquely\ ergodic}\right\}.
    \end{align*}
    In particular,
    \begin{align*}
          & \Int \left\{\left(\chi_1(\mu),\dots,\chi_{\dim M}(\mu)\right):\mu \in \M(f,M)\right\}                       \\
        = & \Int \left\{\left(\chi_1(\mu),\dots,\chi_{\dim M}(\mu)\right):\mu \in \M(f,M) \mathrm{\ is\ uniquely\ ergodic}\right\}.
    \end{align*}
\end{theorem}

Dong, Hou and Tian \cite{dong2024abundance} proved that, for topologically Anosov diffeomorphisms, ergodic measures are dense in the set of invariant measures with a prescribed entropy and a prescribed integral of a given continuous function. This result is called the \emph{conditional intermediate-entropy property}. For the three classes in (H), we establish the analogous exact-constraint density result and strengthen the approximating measures from general ergodic measures to uniquely ergodic measures with compact support. The results stated above are corollaries of the following theorem and its remark.

\begin{theoremx}\label{thm:A}
    Suppose $(X,f)$ satisfies (H). Then for any $\vphi \in C_{b,u}(X)$ and any \(a \in \Rb\) such that either \(a\in\Int L_\vphi\) or \(L_\vphi=\{a\}\), any $0 \leq h < \sup\left\{h_\mu(f) : \mu \in \M(f,X), \int \vphi \dif \mu = a\right\}$,
    one has
    \[\left\{\mu \in \M(f,X) : \mu \mathrm{\ is\ uniquely\ ergodic\ with\ compact\ support},\ h_\mu(f) = h, \int \vphi \dif \mu = a\right\}\]
    is nonempty and dense (in the sense of weak* topology) in
    \[\left\{\mu \in \M(f,X) : h_\mu(f) = h, \int \vphi \dif \mu = a\right\}.\]
\end{theoremx}

Theorem \ref{thm:A} gives a conditional entropy realization and weak* approximation result by uniquely ergodic measures with compact support. More precisely, even after fixing the integral of a potential and a compatible submaximal entropy, the collection of such measures is nonempty and weak*-dense in the corresponding exact constraint set of invariant measures. Since every uniquely ergodic measure with compact support is minimal, the same conclusion holds with minimal ergodic measures in place of uniquely ergodic measures with compact support. Compared with \cite{arbieto2025entropy} and \cite{chang2024orbit}, we consider the dimensions associated with both the entropy and the integral of a given function, thereby enabling us to obtain corresponding results in a finer fractal setting.

We now explain the differences and connections between the present paper and \cite{dong2024abundance}. The two works differ in two respects regarding their objects of study. In terms of measures, \cite{dong2024abundance} establishes density results for general ergodic measures in exact entropy-and-integral constraint sets, whereas the present paper constructs uniquely ergodic measures with compact support, which in particular are minimal and ergodic. In terms of dynamical systems, \cite{dong2024abundance} considers various classes of homeomorphisms, including uniformly hyperbolic and partially hyperbolic systems, whereas the present paper also treats non-invertible maps and non-compact systems such as countable Markov shifts. Although both works use the multi-horseshoe entropy-dense property as a key technical tool, the present paper requires a further construction of nested decreasing horseshoes whose invariant-measure spaces shrink to a singleton, thereby enforcing unique ergodicity while preserving the prescribed entropy and integral. The different classes of systems under consideration also require distinct treatments of non-invertibility, non-compactness and non-uniformity. Consequently, while the main results of the two works partially overlap, neither subsumes the other.

\begin{remark}\label{rmk:thmA}
    In Theorem \ref{thm:A}, replacing $\varphi$ by a finite family of continuous functions $\left\{\varphi_i\right\}_{i=1}^k \subseteq C_{b,u}(X)$ with $\Relint L_{\left\{\vphi_j\right\}_{j=1}^k} \ne \emptyset$ (see the definition preceding Theorem \ref{thm:relint} in Section \ref{sec:cmie}), we obtain a similar result. See Theorem \ref{thm:relint}.
\end{remark}

We now present several examples of the above theorem.

\begin{itemize}
    \item Systems restricted on a topologically transitive locally maximal hyperbolic set, see \cite[Corollary 6.4.10, Theorem 18.1.2]{katok1995introduction}.
    \item Topologically transitive subshifts of finite type \cite[Theorem 1]{walters1977pseudo}.
    \item $\beta$-shifts, $S$-gap shifts and their factors \cite[Section 5]{climenhaga2017large}.
    \item Expanding-Markov R\'{e}nyi maps, including Gauss maps, see \cite{iommi2015multifractal,pollicott1999multifractal}.
\end{itemize}

Next, we briefly explain the main techniques and difficulties of our proof. The case $a\in\Int L_\varphi$ of Theorem \ref{thm:A} follows from Theorem \ref{thmA}, Theorem \ref{thmB} and Theorem \ref{thm:cmed_nonuniform}. When \(L_\varphi=\{a\}\), the integral constraint is automatic, since $\int\varphi\,\dif\mu=a$ for any $\mu\in\M(f,X)$. In this degenerate case, the conclusion reduces to unconditional exact-entropy realization and weak* approximation by uniquely ergodic measures with compact support. Its proof is obtained from the proof below by omitting the measures chosen on the two sides of \(a\) and all the corresponding integral-control conditions.

The proof splits into a system-specific multi-horseshoe input and a
common nested-horseshoe reduction. For cases {\rm (H1)}, {\rm (H2)}
and {\rm (H3)}, the required inputs are supplied by Theorems
\ref{thmC}, \ref{thmD} and
\ref{prop-nonuniform-multi-horseshoe}, respectively. Once the
class-specific input provides horseshoes with the required entropy and
weak* estimates, the remaining mechanism is the same. At each stage,
we select nested subhorseshoes whose entropies and potential integrals
approach the prescribed values, while the diameters of their
invariant-measure spaces tend to zero. Their intersection therefore
supports a unique invariant measure satisfying the exact entropy and
integral constraints. We carry out this reduction in full in the proof
of Theorem \ref{thmA}; the proofs of Theorems \ref{thmB} and
\ref{thm:cmed_nonuniform} reuse it after verifying the corresponding
system-specific input.

In generalizing this property to different systems, we must overcome different technical obstacles in each instance. In what follows, we clarify the differences between these three cases in the simplest possible terms. Systems satisfying (H1) are compact and have the shadowing property but no symbolic structure. Systems satisfying (H2) have the shadowing property and the symbolic structure but are not compact. Systems satisfying (H3) are compact and have the symbolic structure but no shadowing property.

In the case of topologically expanding maps, since systems under consideration are no longer homeomorphisms, the multi-horseshoe entropy-dense property only yields horseshoes in the sense of iterates. In the inductive process, we establish a connection between the invariant space of a horseshoe and its iteration measures via a map. This map preserves the entropy, thus ensuring the feasibility of the inductive procedure.

For countable Markov shifts, the difficulty we faced stemmed from non-compactness. Katok's entropy formula, the core of a technical lemma in the construction of horseshoes, may fail for countable Markov shifts (see \cite[Theorem 3.2]{iommi2022escape}). Therefore, we need to find an alternative approach. In \cite{takahasi2020entropy}, Takahasi proved a technical lemma: given an ergodic measure, one can construct a set of words with large cardinality and almost the same Birkhoff averages with respect to given functions where each word starts and ends with the same symbol. Inspired by this lemma, we construct sets of periodic points for invariant measures satisfying similar properties. Further, we improve \cite[Main Theorem]{takahasi2020entropy} to a multi-horseshoe entropy-dense result that supplies the compact initial horseshoes required by the nested construction.

In the case of symbolic systems with non-uniform structure, the extension of our property to those systems is technically demanding due to the lack of the shadowing property or the (global) specification property. In \cite{climenhaga2017large}, Climenhaga, Thompson and Yamamoto proved the entropy-approachable property for those systems. Inspired by this, we can overcome this non-uniformity by applying the edit-approachable property and the weak specification property. It allows us to rigorously establish the existence of a sequence of horseshoes satisfying certain properties.

Thus Theorems \ref{thmA}, \ref{thmB} and
\ref{thm:cmed_nonuniform} are three implementations of the same
nested-horseshoe mechanism rather than independent constructions; the
system-dependent work lies in producing the horseshoe input and the
auxiliary estimates required by the induction.

\subsection*{Organization of this article}
The paper is divided into three parts, addressing topologically expanding maps, countable Markov shifts and symbolic systems with non-uniform structure, respectively. In Section \ref{secPreliminary}, we review the definitions needed to state and prove our results. In Section \ref{sec:mhed}, we establish the multi-horseshoe entropy-dense property for topologically expanding systems. In Section \ref{sec:cmie}, we use nested horseshoes to obtain exact entropy-and-average realization and weak* approximation by uniquely ergodic measures with compact support. In Sections \ref{sec:pre_cms} and \ref{sec:pre_symbolic}, we review the definitions and lemmas needed for countable Markov shifts and symbolic systems with non-uniform structure, respectively. In Sections \ref{sec:minimal_cms} and \ref{sec:minimal_symbolic}, we establish the corresponding realization and approximation results for these two classes.

\part{Topologically expanding maps}
In this part, we deal with the first case. Let $f:X \to X$ be a topologically expanding  continuous map on a compact metric space $(X,d)$.

\section{Preliminary}\label{secPreliminary}
\subsection{The space of probability measures}\label{subsec:M(X)}
In this subsection, we recall the weak* topology and metric on the space of probability measures and give a technical lemma.

Let $D$ be the first Wasserstein distance (also called Kantorovich-Rubinstein distance) on $\M(X)$, a metrization of the weak* topology of $\M(X)$, see \cite[Definition 6.1]{villani2009optimal}. We list some properties of this metric here.
\begin{itemize}
    \item $d(x,y) = D\left(\delta_x, \delta_y\right)$ for any $x,y \in X$.
    \item $D\left(p\mu_1+(1-p)\nu_1, p\mu_2+(1-p)\nu_2\right) \leq pD\left(\mu_1, \mu_2\right) + (1-p) D\left(\nu_1, \nu_2\right)$ for any $0 \leq p \leq 1$ and any $\mu_1, \mu_2, \nu_1, \nu_2 \in \M(X)$.
\end{itemize}
Since $\diam(X)\leq 1$, one has $\diam\bigl(\M(X)\bigr)\leq 1$.

For any $x \in X$, define the \emph{empirical measure} of $x$ (with respect to $f$) as
$$\mathcal{E}_{n}(x):=\frac{1}{n}\sum_{j=0}^{n-1}\delta_{f^{j}x}$$
where $\delta_{x}$ is the Dirac measure at $x$. Let $V_f(x)$ denote the set of all limit points of $\left\{\cE_n(x)\right\}_{n=1}^\infty$ in the weak* topology.
For any $\mu \in \M(f,X)$, define the \emph{generic points} of $\mu$ (with respect to $f$) as $G_\mu:= \left\{x \in X: V_f(x) = \left\{\mu\right\} \right\}$.

The following lemma is easily verified and shows that if any two orbits in finite steps are close enough, then the two empirical measures induced by them are also close.

\begin{lemma}\label{lemMeasuresDistance}
    Let $\ve > 0$, $0 < \delta < 1$, $n \in \Nb^+$ and $\left\{x_i\right\}_{i=0}^{n-1}$, $\left\{y_i\right\}_{i=0}^{n-1}$ be two sequences in $X$. For any nonempty $J \subseteq \left\{0,1,\dots,n-1\right\}$ with $\left(n-\left\lvert J\right\rvert \right) / n < \delta $,
    \[D\left(\frac{1}{n}\sum_{i=0}^{n-1}\delta_{x_i}, \frac{1}{\left\lvert J\right\rvert }\sum_{i \in J}\delta_{x_i}\right) < \delta.\]
    If further $d(x_i, y_i) < \ve$ holds for any $i \in J$, then
    \[D\left(\frac{1}{n}\sum_{i=0}^{n-1}\delta_{x_i}, \frac{1}{n}\sum_{i=0}^{n-1}\delta_{y_i}\right) < \ve + \delta\]
    and hence
    \[D\left(\frac{1}{n}\sum_{i=0}^{n-1}\delta_{x_i}, \frac{1}{\left\lvert J\right\rvert }\sum_{i \in J}\delta_{y_i}\right) < \ve + 2\delta.\]
    In particular, if $d(x_i, y_i) \leq \ve$ holds for any $i \in \left\{0,1,\dots,n-1\right\}$, then
    \[D\left(\frac{1}{n}\sum_{i=0}^{n-1}\delta_{x_i}, \frac{1}{n}\sum_{i=0}^{n-1}\delta_{y_i}\right) \leq \ve.\]
\end{lemma}

\subsection{Entropy and pressure}
In this subsection, we recall the definition of entropy and pressure, consisting of the topological entropy for non-compact sets introduced by Bowen in \cite{bowen1973topological}, the metric entropy and the topological pressure.

\subsubsection{Topological entropy}\label{subsec:htop}
In this subsection, we recall the definition of topological entropy. Let $f:X \to X$ be a continuous map on a (compact or non-compact) metric space $(X,d)$.

For any $x \in X$, $n \in \Nb$ and any $\ve > 0$, define
\[B_n(x,\ve) = \left\{y \in X: d\left(f^ix, f^iy\right) < \ve \mathrm{\ for\ any\ } i \in \left\{0,1,\dots,n-1\right\} \right\}.\]
\begin{definition}\label{def:htop}
    Let $E \subseteq X$, $\ve > 0$, and $\F_n(E,\ve)$ be the collection of all finite or countable covers of $E$ by sets of the form $B_m(x,\ve)$ with $m \geq n$. Let
    $$C(E;t,n,\ve,f) = \inf_{\mathcal{C} \in \F_n(E,\ve)}\left\{\sum_{B_m(x,\ve) \in \mathcal{C}} e^{-tm}\right\}$$
    and
    $$C(E;t,\ve,f) = \lim_{n\to\infty}C(E;t,n,\ve,f).$$
    Let
    $$\htop(f,E,\ve)=\inf\left\{t:C(E;t,\ve,f) = 0\right\} = \sup\left\{t:C(E;t,\ve,f) = \infty\right\}.$$
    The \emph{(Bowen) topological entropy} of $E$ with respect to $f$ is defined by
    $$\htop(f,E) = \lim_{\ve\to 0}\htop(f,E,\ve).$$
\end{definition}

Next, we recall some basic facts of topological entropy.

\begin{lemma}[{{\cite[Proposition 2]{bowen1973topological}}}]\label{lem:htop_sum}
    Given a sequence of subsets $\left\{Y_i\right\}_{i=1}^{\infty} \subseteq X$, one has
    \[\htop\left(f, \bigcup_{i=1}^{\infty} Y_i\right) = \sup_{1 \leq i < \infty} \htop\left(f, Y_i\right).\]
\end{lemma}

A point $x \in X$ is called \emph{wandering} for $f$ if there exists an open neighborhood $U$ of $x$ such that elements of $\left\{f^{-n} U\right\}_{n \geq 0}$ are mutually disjoint. The \emph{non-wandering set} for $f$, denoted by $\Omega(f)$, consists of all points that are not wandering for $f$.

\begin{lemma}[{{\cite[Theorem 2.4]{bowen1970topological}}}]\label{lem:htop_nw}
    If $X$ is compact, then $\htop(f,X) = \htop\left(f, \Omega(f)\right)$.
\end{lemma}

\subsubsection{Metric entropy}
In this subsection, we recall the definition of metric entropy. We call $(X, \mathcal{B}, \mu)$ a probability space if $\mathcal{B}$ is a Borel $\sigma$-algebra on $X$ and $\mu$ is a probability measure on $X$. We say $\xi$ is a partition of $X$ if $\bigcup_{A \in \xi} A = X$ and elements in $\xi$ are pairwise disjoint. We say a partition $\xi$ is measurable if any element of $\xi$ is measurable. For any two partitions $\xi, \eta$ of $X$, let
\[\xi \vee \eta = \left\{A \cap B: A \in \xi, B \in \eta\right\}.\]
Let $\xi = \left\{A_i \subseteq X:1 \leq i \leq k\right\}$ be a finite measurable partition of $X$. The \emph{entropy of $\xi$ with respect to $\mu \in \M(X)$} is defined by
$$H_\mu(\xi) = -\sum_{i=1}^k \mu\left(A_i\right) \log\mu\left(A_i\right) $$
where $0 \log 0:= 0$. If further $\mu \in \M(f,X)$, the following limit exists:
$$h_\mu(f,\xi) = \lim_{n\to\infty}\frac{1}{n}H_\mu\left(\bigvee_{i=0}^{n-1}f^{-i}\xi\right).$$
The \emph{metric entropy} of $\mu$ is defined by
$$h_\mu(f) = \sup\left\{h_\mu(f,\xi):\xi\mathrm{\ is\ a\ finite\ measurable\ partition\ of\ }X\right\}.$$
More information about metric entropy can be checked in \cite[Chapter 4]{walters1982introduction}.

\subsubsection{Topological pressure}
In this subsection, we recall the definition of topological pressure. Let $\varphi \in C(X)$. Given an open cover $\alpha$ of $X$, let
\[P_n(\varphi,\alpha) = \inf\left\{\sum_{U \in \gamma} \sup_{x \in U} e^{\sum_{i=0}^{n-1}\varphi\left(f^i x\right)}: \gamma \mathrm{\ is\ a\ finite\ subcover\ of\ } \bigvee_{i=0}^{n-1} f^{-i}\alpha\right\}.\]
This sequence $\log P_n(\varphi,\alpha)$ is subadditive and so the limit
\[P(\varphi,\alpha) = \lim_{n \to \infty} \frac{1}{n}\log P_n(\varphi,\alpha)\]
exists. Define the \emph{topological pressure} of the potential $\varphi$ with respect to $f$ to be
\[P(\varphi) = \lim_{\diam \alpha \to 0} P(\varphi,\alpha)\]
where $\diam \alpha := \sup\left\{\diam U: U \in \alpha\right\}$. More information about topological pressure can be checked in \cite[Section 10.3]{viana2016foundations}.

\subsection{Transitivity, mixing and the shadowing property}
In this subsection, we recall some definitions about transitivity, mixing and the shadowing property. First, we recall the definition of topological transitivity.
\begin{definition}
    We say $(X,f)$ is \emph{topologically transitive} if for any two nonempty open sets $U,V \subseteq X$ there exists $n \in \Nb^+$ such that $f^{-n}(U) \cap V \ne \emptyset$.
\end{definition}

Internal chain transitivity can be viewed as a weaker form of transitivity.

\begin{definition}
    We say $\left\langle x_i \right\rangle_{i=1}^l$ in $X$ is a \emph{$\delta$-chain} for some $\delta > 0$ and $l \in \Nb^+$ if $d\left(f(x_i), x_{i+1}\right) < \delta $ for any $1 \leq i \leq l-1$. Let $A \subseteq X$ be a nonempty compact $f$-invariant set. We say a chain $\left\langle x_i \right\rangle_{i=1}^l$ is \emph{in $A$} if $x_i \in A$ for any $1 \leq i \leq l$. We say a $\delta$-chain $\left\langle x_i \right\rangle_{i=1}^l$ \emph{connects $a$ and $b$} if $a = x_1$ and $d\left(f (x_l), b\right) < \delta$. We call $A$ \emph{internally chain transitive} if for any $x,y \in A$ and any $\ve > 0$, there exists an $\ve$-chain in $A$ connecting $x$ and $y$.
\end{definition}
\begin{remark}
    If $(X,f)$ is topologically transitive, then $X$ is internally chain transitive.
\end{remark}

On the other hand, mixing can be viewed as a stronger form of transitivity.

\begin{definition}
    We say $(X,f)$ is topologically mixing if for any two nonempty open sets $U,V \subseteq X$ there exists $N = N(U,V) \in \Nb^+$ such that for any $n \geq N$, one has $f^{-n}(U) \cap V \ne \emptyset$.
\end{definition}

At the end of this subsection, we recall the definition of the shadowing property.

\begin{definition}
    We say $\left\{x_n\right\}_{n = 0}^{+\infty}$ in $X$ is a $\delta$-pseudo-orbit for some $\delta > 0$ if $d\left(f(x_n), x_{n+1}\right) < \delta $ for any $n \in \Nb$. We say $\left\{x_n\right\}_{n = 0}^{+\infty}$ in $X$ is $\ve$-shadowed by $y$ in $X$ for some $\ve > 0$ if $d\left(x_n,f^n y\right) < \ve$ for any $n \in \Nb$. We say $(X,f)$ satisfies the shadowing property if for any $\ve > 0$, there exists $\delta > 0$ such that any $\delta$-pseudo-orbit is $\ve$-shadowed by some point in $X$.
\end{definition}

Note in some articles, $\delta$-chains are also called $\delta$-pseudo-orbits. In this article, chains are used in the case of finitely many points, while pseudo-orbits are used in the case of infinitely many points.

Full shifts are examples satisfying the above property, see \cite[Example 9.2.15]{viana2016foundations}, \cite[Proposition 7.5]{denker1976ergodic} and \cite{walters1977pseudo} for a proof.

\begin{lemma}\label{lemFullShiftsProperties}
    Every two-sided (one-sided) full shift is expansive (positively expansive), topologically mixing and satisfies the shadowing property.
\end{lemma}

\subsection{Expansivity, the uniform separation property and the entropy-dense property}
In this subsection, we recall the definitions of expansivity, the uniform separation property and the entropy-dense property. First, we recall the definition of expansivity.

\begin{definition}
    When $f:X \to X$ is a homeomorphism, we say $(X,f)$ is expansive if there exists a constant $c > 0$ such that for any $x \ne y \in X$, one has $d\left(f^i x, f^i y\right) > c $ for some $i \in \Zb$. We call $c$ an expansive constant.
\end{definition}
\begin{definition}
    When $f:X \to X$ is a continuous map, we say $(X,f)$ is positively expansive if there exists a constant $c > 0$ such that for any $x \ne y \in X$, one has $d\left(f^i x, f^i y\right) > c $ for some $i \in \Nb$. We call $c$ a positively expansive constant.
\end{definition}

Note positively expansive maps and expansive homeomorphisms are distinguished in this article. In some articles, positively expansive maps are called expansive maps for consistency. The following lemma, always stated for expansive homeomorphisms, also holds for positively expansive maps.

\begin{lemma}[{{\cite[Corollary 9.2.17]{viana2016foundations}}}]\label{lemUpperSemiContinuity}
    If $(X,f)$ is positively expansive, then the entropy function $h_{\left(\cdot\right)}(f) : \M(f,X) \to \Rb$ is upper semi-continuous.
\end{lemma}

In this article, we deal with the iteration of the systems in some cases. To make our proof more comprehensible, we recall the following lemma.

\begin{lemma}[{{\cite[Corollary 5.22.1]{walters1982introduction}}}]\label{lemIterationExpansive}
    If $(X,f)$ is positively expansive, then $(X,f^n)$ is positively expansive for any $n \in \Nb^+$.
\end{lemma}

Next, we recall the definition of uniform separation property introduced by Pfister and Sullivan in \cite{pfister2007topological}. First, we recall the definition of $(\delta,n,\ve)$-separated set.
\begin{definition}
    For any $\ve > 0$ and $n \in \Nb^+$, we say two points $x$, $y \in X$ are \emph{$(n,\ve)$-separated} if there exists $0 \leq j \leq n-1$ such that $d \left(f^j x ,f^j y\right) > \ve$. We say a subset $E \subseteq X$ is \emph{$(n,\ve)$-separated} if any pair of different points in $E$ is $(n,\ve)$-separated.

    For any $\delta > 0$, $\ve > 0$ and $n \in \Nb^+$, we say two points $x$, $y \in X$ are \emph{$(\delta,n,\ve)$-separated} if
    $$\frac{1}{n}\left\lvert \left\{ 0 \leq j \leq n-1:d \left(f^j x ,f^j y\right) > \ve\right\} \right\rvert  \geq \delta.$$
    We say a subset $E \subseteq X$ is \emph{$(\delta,n,\ve)$-separated} if any pair of different points in $E$ is $(\delta,n,\ve)$-separated.
\end{definition}

Note all $(\delta,n,\ve)$-separated sets are also $(n,\ve)$-separated. Next, we recall the definition of the uniform separation property.

\begin{definition}
    We say that $(X,f)$ has \emph{uniform separation property} if for any $\eta > 0$, there exist $\delta^* > 0$ and $\ve^* > 0$ such that for any ergodic measure $\mu$ and any neighborhood $F \subseteq \M(X)$ of $\mu$, there exists $n_{F,\mu,\eta}^*$ such that for any integer $n \geq n_{F,\mu,\eta}^*$, there exists a $(\delta^*,n,\ve^*)$-separated set $\Gamma_n \subseteq X_{n,F}$ with
    $$\left\lvert \Gamma_n \right\rvert  \geq e^{n(h_\mu(f)-\eta)}$$
    where $X_{n,F}:=\left\{x \in X: \mathcal{E}_{n}(x) \in F\right\} $.
\end{definition}

The following lemma gives the relationship between expansivity and the uniform separation property.

\begin{lemma}[{{\cite[Theorem 3.1]{pfister2007topological}}}]\label{lemExpansiveImpliesUSP}
    If $(X,f)$ is positively expansive, then $(X,f)$ has the uniform separation property.
\end{lemma}

Next, we recall the definition of entropy-dense property.

\begin{definition}
    We say $(X,f)$ satisfies the \emph{entropy-dense property} if for any $\mu \in \M(f,X)$, any neighborhood $B(\mu)$ of $\mu$ in $\M(X)$ and any $h < h_\mu(f)$, there exists $\nu \in \Merg(f,X)$ such that $h_\nu(f) > h$ and $\M(f, \supp\left(\nu\right) ) \subseteq B(\mu)$.
\end{definition}
\begin{remark}
    By the variational principle, this definition is equivalent to the following condition: for any $\mu \in \M(f,X)$, any neighborhood $B(\mu)$ of $\mu$ in $\M(X)$ and any $h < h_\mu(f)$, there exists a nonempty closed $f$-invariant set $\Lambda = \Lambda(\mu, B(\mu), h) \subseteq X$ such that $\M(f, \Lambda) \subseteq B(\mu)$ and $\htop(f, \Lambda) > h$.
\end{remark}

By \cite[Theorem 30]{kwietniak2016panorama}, topologically transitive systems with the shadowing property satisfy the approximate product property. By \cite[Proposition 2.3(1)]{pfister2005large}, systems with the approximate product property satisfy the entropy-dense property, which leads to the following lemma.

\begin{lemma}\label{lemTransShaImpliesEDP}
    Suppose $(X,f)$ is topologically transitive and has the shadowing property. Then $(X,f)$ satisfies the entropy-dense property.
\end{lemma}

\section{Multi-horseshoe entropy-dense property}\label{sec:mhed}
In this section, we prove the ``multi-horseshoe'' entropy-dense property for topologically expanding maps. Dong, Hou and Tian proved the ``multi-horseshoe'' entropy-dense property for topologically Anosov diffeomorphisms.

\begin{theorem}\cite[Theorem 3.1]{dong2024abundance}
    Suppose $(X,f)$ is topologically Anosov and topologically transitive. Then for any positive integer $m$, any $f$-invariant measures $\left\{\mu_i\right\}_{i=1}^m \subseteq \M(f,X)$, any $x \in X$ and any $\veh, \veH, \ved > 0$, there exists $n \in \Nb^+$ and compact $f$-invariant subsets $\Lambda_i \subseteq \Lambda \subsetneq X$ such that for each $1 \leq i \leq m$, the following statements hold.
    \begin{enumerate}
        \item $\left(\Lambda_i, f\right)$ and $\left(\Lambda, f\right)$ are conjugate to a two-sided topologically transitive subshift of finite type respectively.
        \item $\htop\left(f,\Lambda_i\right) > h_{\mu_i}(f) - \veh$.
        \item $D_{\mathrm{H}}\left(K, \M(f, \Lambda)\right) < \veH$, $D_{\mathrm{H}}\left(\left\{\mu_i\right\} , \M(f,\Lambda_i)\right) < \veH$ where $K = \cov\left\{\mu_i\right\}_{i=1}^m$ and $D_{\mathrm{H}}$ is the Hausdorff distance.
        \item For any $z \in \Lambda_i$ or $\Lambda$ one has $f^{j+tn}z \in B(x,\ved)$ for some $1 \leq j \leq n$ for any $t \in \Nb$.
    \end{enumerate}
\end{theorem}

The following theorem is a generalization of the ``multi-horseshoe'' entropy-dense property for topologically Anosov diffeomorphisms \cite[Theorem 3.1]{dong2024abundance}.

\begin{theoremx}\label{thmC}
    Suppose $(X,f)$ is topologically expanding and topologically transitive. Then for any positive integer $m$, any $m$ distinct $f$-invariant measures $\left\{\mu_i\right\}_{i=1}^m \subseteq \M(f,X)$, any $x \in X$ and any $\veh, \veH, \ved > 0$, there exists $n \in \Nb^+$ and compact $f^n$-invariant subsets $\Delta_i \subseteq \Delta \subsetneq X$ such that for each $1 \leq i \leq m$, the following statements hold.
    \begin{enumerate}
        \item $\left(\Delta_i, f^n|_{\Delta_i}\right)$ and $\left(\Delta, f^n|_{\Delta}\right)$ are topologically conjugate to a one-sided full shift respectively.
        \item $\htop\left(f,\Lambda_i\right) = \htop \left(f^n,\Delta_i\right)  / n > h_{\mu_i}(f) - \veh$ where $\Lambda_i = \bigcup_{j=0}^{n-1}f^j \left(\Delta_i\right) $.
        \item $D_{\mathrm{H}}\left(K, \M(f, \Lambda)\right) < \veH$ and $D_{\mathrm{H}}\left(\left\{\mu_i\right\} , \M(f,\Lambda_i)\right) < \veH$ where $K = \cov\left\{\mu_i\right\}_{i=1}^m$ and $\Lambda = \bigcup_{j=0}^{n-1}f^j \left(\Delta\right) $.
        \item For any $z \in \Lambda$, there exists $1 \leq j \leq n$ such that $f^{j+tn}z \in B(x,\ved)$ for any $t \in \Nb$.
    \end{enumerate}
\end{theoremx}

Here we briefly explain the differences between Theorem \ref{thmC} and \cite[Theorem 3.1]{dong2024abundance}. Influenced by the restriction that $f$ is a map, one can only obtain that $\left(\Delta_i, f^n\right)$ is topologically conjugate to a one-sided full shift. In contrast, \cite[Theorem 3.1]{dong2024abundance} further established that $\left(\Lambda_i, f\right)$ is topologically conjugate to a topologically transitive subshift of finite type on two-sided sequences.

To prove Theorem \ref{thmC}, we introduce the following lemma first, which generalizes \cite[Corollary 4.13]{dong2022different} to the case of maps.

\begin{lemma}\label{lemSeparatedSet}
    Suppose $(X,f)$ is topologically transitive and topologically expanding. Then for any $\eta > 0$, there exists $\ve_1^* = \ve_1^*(\eta) > 0$ such that for any $\mu \in \M(f,X)$ and its neighborhood $F_\mu$, there exists $0 < \ve_2^* = \ve_2^*\left(\eta, \mu, F_\mu\right) < \ve_1^*$ such that for any $x \in X$, any $0 < \ve \leq \ve_2^*$, any $N \in \Nb^+$, there exists an $n = n\left(\eta, \mu, F_\mu, \ve, x\right) \geq N$ such that for any $p \in \Nb^+$, there exists a $\left(pn, \ve_1^* / 3\right) $-separated set $\Gamma_{pn}$ such that
    \begin{enumerate}
        \item $\Gamma_{pn} \subseteq X_{pn,F_\mu} \cap B(x,\ve) \cap P_{pn}(f)$;
        \item $\left(\log\left\lvert \Gamma_{pn}\right\rvert \right) / \left(pn\right)  > h_\mu(f) - \eta$.
    \end{enumerate}
\end{lemma}
\begin{proof}
    Since $(X, f)$ is topologically expanding, it satisfies the shadowing property and is positively expansive. Let $c > 0$ a positive expansivity constant for the map $f$. By Lemma \ref{lemExpansiveImpliesUSP}, $(X,f)$ has the uniform separation property. Then for the given $\eta > 0$, there exist $\ve^* > 0$ such that for any $\mu \in \Merg(f,X)$ and any neighborhood $F \subseteq \M(X)$ of $\mu$, there exists $n^*_{F,\mu,\eta}$ such that for any $n \geq n^*_{F,\mu,\eta}$, there exists a $(n, \ve^*)$-separated set $\Gamma_n \subseteq X_{n,F}$ with
    $\left\lvert \Gamma_n \right\rvert  \geq e^{n(h_\mu(f)-\eta/4)}$. Let $\ve_1^* = \ve^*$.

    Fix $\mu \in \mathcal{M}(f, X)$ and its neighborhood $F_\mu$ in the weak* topology. Fix $a > 0$ such that $B(\mu, a) \subseteq F_\mu$. Let $\ve^*_2 = \min\left\{\ve_1^*/9, a/4, c/2\right\}$. Fix $x \in X$, $0 < \ve \le \ve^*_2$ and $N \in \Nb^+$. Since $(X,f)$ has the shadowing property, there is $0 < \delta < \ve$ such that any $\delta$-pseudo-orbit can be $\ve$-shadowed by some point in $X$.

    By the Ergodic Decomposition Theorem, there exists a convex combination of ergodic measures $\frac{1}{m}\sum_{i=1}^m \nu_i$ such that
    \begin{equation}\label{eq:lem4.2_D}
        D\left(\frac{1}{m}\sum_{i=1}^m \nu_i,\mu\right) < a/4
    \end{equation}
    and
    \begin{equation}\label{eq:4.1}
        \frac{1}{m}\sum_{i=1}^m h_{\nu_i}(f) > h_\mu(f) - \eta/4.
    \end{equation}
    For any $1 \leq i \leq m$, by the uniform separation property of $(X,f)$, there exists $N_i \in \Nb^+$ such that for any $n \geq N_i$, there exists a $(n, \ve^*)$-separated set $\Gamma_n^{i} \subseteq X_{n,B\left(\nu_i,\ve/4\right)}$ with
    $\left\lvert \Gamma_n^{i} \right\rvert  \geq e^{n(h_{\nu_i}(f)-\eta/4)}$.

    Note $\left\{B(x,\delta/2):x\in X\right\}$ is an open cover of $X$. By the compactness of $X$, there exists a finite subcover $\left\{B(x_i,\delta/2)\right\}_{i=1}^l$. Since $(X,f)$ is topologically transitive, for any $1 \leq i \leq l$, there exists a $\delta$-chain $\mathbf{C}_{x x_i}$ connecting $x$ and $x_i$, and a $\delta$-chain $\mathbf{C}_{x_i x}$ connecting $x_i$ and $x$. Let $L$ be the maximum of the lengths of these pseudo-orbits. Take $k \in \Nb^+$ large enough such that
    \begin{equation}\label{eq:Lem4.2_L}
        4L / k \leq \ve
    \end{equation}
    \begin{equation}\label{eq:Lem4.2_l}
        k \geq 4\left(\log l + L\max\left\{h_{\mu}(f)-\eta,0\right\}  \right) / \eta
    \end{equation}
    \begin{equation*}
        mk > N
    \end{equation*}
    \begin{equation*}
        k \geq \max\left\{N_i:1 \leq i \leq m\right\}.
    \end{equation*}

    For any $1 \leq i \leq m$, there is a $\left(k,\ve^*\right)$-separated set $\Gamma_{k}^{i}\subseteq X_{k,B(\nu_i,a/4)}$ with
    \begin{equation}\label{eq:Lem4.2_hk}
        \frac{\log\left\lvert \Gamma_{k}^{i} \right\rvert }{k}\geq h_{\nu_i}(f)-\eta/4
    \end{equation}
    By the pigeonhole principle, there exist indices $1\leq i_1,i_2\leq l$ and subsets $\widetilde{\Gamma_{k}^{i}} \subseteq \Gamma_{k}^{i}$ such that
    \begin{equation}\label{eq:Lem4.2_pigeonhole}
        \left\lvert \widetilde{\Gamma_{k}^{i}}\right\rvert \geq \left\lvert\Gamma_{k}^{i} \right\rvert / l^2
    \end{equation}
    and $\widetilde{\Gamma_{k}^{i}} \subseteq B\left(x_{i_1},\delta/2\right) \cap f^{-k}B\left(x_{i_2},\delta/2\right)$.

    Let $\Gamma = \widetilde{\Gamma_{k}^{1}} \times \dots \times \widetilde{\Gamma_{k}^{m}}$. Any $\left(y_1,\dots,y_m\right) \in \Gamma$ corresponds to a $\delta$-chain
    \[\mathbf{C}_{\overline{y}} = \mathbf{C}_{xx_{1_1}}\left\langle y_1,fy_1,\dots,f^{k-1} y_1\right\rangle \mathbf{C}_{x_{1_2}x}\dots\mathbf{C}_{xx_{m_1}}\left\langle y_m,fy_m,\dots,f^{k-1} y_m\right\rangle \mathbf{C}_{x_{m_2}x}\]
    with length in $\left[mk,mk+2mL\right]$. Let $n$ be the length of this chain. Fix $p \in \Nb^+$. Chains $\mathbf{C}_{\overline{y}}$ generated by points in $\Gamma$ can be concatenated freely. By the shadowing property of $(X,f)$, there exists $z\left(\overline{y}\right) \in X$ $\ve$-shadows $\mathbf{C}_{\overline{y}}$ chain generated by $\overline{y} \in \Gamma^p := \prod_{i=1}^{p} \Gamma$. Let $\Gamma_{pn} = \left\{z\left(\overline{y}\right): \overline{y} \in \Gamma^p\right\}$.

    Since $0 < \ve \leq \ve^* / 9$, $\Gamma_{pn}$ is $\left(pn,\ve^*/3\right)$-separated. By the construction of those chains, $\Gamma_{pn} \subseteq B(x,\ve) \cap f^{-pn}B(x,\ve)$. Since $f^{pn}z\left(\overline{y}\right)$ also shadows the chain generated by $\overline{y}$ for any $\overline{y} \in \Gamma^p$, one has $d\left(f^{pn+i}z,f^{i}z\right) < 2\ve \leq 2\ve^*_2 \leq c$ for any $z \in \Gamma_{pn}$. Since $c$ is a positively expansive constant of $(X,f)$, one has $f^{pn}z = z$ for any $z \in \Gamma_{pn}$, which implies $\Gamma_{pn} \subseteq P_{pn}(f)$. For any $z = z(\overline{y}) \in \Gamma_{pn}$, assume $\overline{y} = (y_{1,1},\dots,y_{p,m})$ where $\left(y_{s,1},\dots,y_{s,m}\right)  \in \Gamma$ for any $1 \leq s \leq p$.
    \begin{align*}
        D\left(\cE_{pn}(z),\mu\right)                                       & \leq D\left(\cE_{pn}(z),\frac{1}{m}\sum_{i=1}^m \nu_i\right) + D\left(\frac{1}{m}\sum_{i=1}^m \nu_i,\mu\right) \\
        (\text{by Lemma \ref{lemMeasuresDistance} and (\ref{eq:lem4.2_D})}) & < \frac{1}{pm}\sum_{s=1}^{p}\sum_{i=1}^{m}D\left(\cE_{k}(y_{s,i}),\nu_i\right)  + \ve + 4L/k + a/4             \\
        (\text{by (\ref{eq:Lem4.2_L})})                                     & \leq a/2 + 2\ve^*_2 \leq a,
    \end{align*}
    which implies that $\Gamma_{pn} \subseteq X_{pn,B(\mu,a)} \subseteq X_{pn,F_\mu}$. Finally,
    \begin{align*}
        \left\lvert \Gamma_{pn}\right\rvert      & = \left(\prod_{i=1}^m \left\lvert \widetilde{\Gamma_{k}^i} \right\rvert \right)^p \\
        (\text{by (\ref{eq:Lem4.2_pigeonhole})}) & \geq \left(\prod_{i=1}^m \left\lvert \Gamma_{k}^i\right\rvert  / l^2\right)^p     \\
        (\text{by (\ref{eq:Lem4.2_hk})})         & > e^{kp\sum_{i=1}^m \left(h_{\nu_i}(f)-\eta/4\right)} / l^{2mp}                   \\
        (\text{by (\ref{eq:4.1})})               & > e^{kmp\left(h_\mu(f)-\eta/2\right) } / l^{2mp}                                  \\
        (\text{by (\ref{eq:Lem4.2_l})})          & > e^{pn\left(h_{\mu}(f)-\eta\right)},
    \end{align*}
    which completes the proof.
\end{proof}

Next, we prove Theorem \ref{thmC}.

\begin{proof}[Proof of Theorem \ref{thmC}]
    Fix $m \in \Nb$, $\left\{\mu_i\right\}_{i=1}^m \subseteq \M(f,X)$, $x \in X$ and $\veh, \veH, \ved > 0$. Without loss of generality, assume $m \geq 2$. Let
    \[D_0 = \min\left\{D\left(\mu_i, \mu_j\right) : 1 \leq i < j \leq m \right\}.\]
    By Lemma \ref{lemTransShaImpliesEDP}, $(X,f)$ has the entropy-dense property and hence $\left\lvert \Merg(f,X)\right\rvert = +\infty$. By \cite[Theorem 6.10 (iii)]{walters1982introduction}, ergodic measures are extreme points in $\M(f,X)$. Hence, $\left\lvert K \cap \Merg(f,X) \right\rvert \leq m < \left\lvert \Merg(f,X)\right\rvert$, which implies that $\Merg(f,X) \not \subseteq K$ and hence $D_\mathrm{H}\left(K,\M(f,X)\right) > 0$. Since $\left\lvert X\right\rvert = \infty$, $f$ is topologically transitive and $f$ satisfies the shadowing property, one has $\htop(f,X) > 0$, see \cite[Section 2.3]{dong2024abundance}. Let $0 < \veh_0 \leq \min\left\{\htop(f), \veh\right\} $, $0 < \veH_0 < \min\left\{D_\mathrm{H}\left(K,\M(f,X)\right), D_0, \veH\right\}$.

    By Lemma \ref{lemSeparatedSet}, there exists $\ve^*_1 > 0$ and $0 < \ve_2^* < \ve^*_1$ such that for any $0 < \delta < \ve_2^*$ and any $1 \leq i \leq m$, there exists $n_i = n_i(\delta) \in \Nb^+$ such that for any $p \in \Nb^+$, there exists a $\left(pn_i, \ve_1^* / 3\right) $-separated set $\Gamma_{pn_i}^{\mu_i}$ such that
    \begin{enumerate}
        \item $\Gamma_{pn_i}^{\mu_i} \subseteq X_{pn_i,B\left(\mu_i, \veH_0 / 4\right) } \cap B(x,\delta / 2) \cap P_{pn_i}(f)$;
        \item $\left(\log\left\lvert \Gamma_{pn_i}^{\mu_i} \right\rvert \right) / \left(pn_i\right)  > h_{\mu_i} - \veh_0$.
    \end{enumerate}

    Let $n = n(\delta) = \prod_{i=1}^m n_i$.
    Fix
    \[0 < \ve < \min\left\{ \veH_0 / 6, \ve_1^* / 9, \ved / 2, c / 3 \right\}\]
    where $c$ is a positively expansive constant of $(X, f)$. Since $(X,f)$ satisfies the shadowing property, there exists $0 < \delta < \ve$ such that any $\delta$-pseudo-orbit can be $\ve$-shadowed by some point in $X$. Then for any $1 \leq i \leq m$, there exists $(n, \ve_1^* / 3) $-separated set $\Gamma_{n}^{i}$ such that
    \begin{enumerate}
        \item $\Gamma_{n}^{i} \subseteq X_{n,B\left(\mu_i, \veH_0 / 4\right) } \cap B(x,\delta / 2) \cap P_{n}(f)$;
        \item $\left(\log\left\lvert \Gamma_{n}^{i} \right\rvert \right) / n  > h_{\mu_i} - \veh_0$.
    \end{enumerate}

    Denote $\mathbf{C}_y^n = \left\langle y, fy, \ldots, f^{n-1}y \right\rangle $ for any $y \in X$. Let $\Gamma_n = \bigsqcup_{i=1}^m \Gamma_n^{i}$. Then $\Gamma_n$ is $\left(n,3\ve\right)$-separated. Otherwise, there are $y_1,y_2 \in \Gamma_n$ with $y_1 \ne y_2$ and $d_n\left(y_1, y_2\right) \leq 3\ve < \ve_1^* / 3$ and hence $D\left(\cE_n\left(y_1\right), \cE_n\left(y_2\right)\right) \leq 3\ve < \veH_0 / 2$ by Lemma \ref{lemMeasuresDistance}. If $y_1, y_2 \in \Gamma_n^i$ for some $1 \leq i \leq m$, then $d_n\left(y_1,y_2\right) > \ve_1^* / 3$ since $\Gamma_n^i$ is $\left(n,\ve_1^*/3\right)$-separated, which leads to a contradiction. If $y_1 \in \Gamma_n^i$ and $y_2 \in \Gamma_n^j$ for some $1 \leq i \ne j \leq m$, then $y_1 \in X_{n,B\left(\mu_i, \veH_0 / 4\right) }$ and $y_2 \in X_{n,B\left(\mu_j, \veH_0 / 4\right) }$. Hence, $D\left(\cE_n\left(y_1\right), \cE_n\left(y_2\right)\right) > D_0 - \veH_0 / 2 > \veH_0 / 2$, which leads to a contradiction.

    For any $y \in \Gamma_n$, $\mathbf{C}_y^n$ is a $\delta$-chain and those $\delta$-chains can be concatenated with each other to form a new $\delta$-chain or a $\delta$-pseudo-orbit. Denote $r_i = \left\lvert \Gamma_n^{i} \right\rvert $ and $r = \left\lvert \Gamma_n\right\rvert $. Let
    \[\Sigma_{r_i}^i = \left\{ \overline{y}  = \left(y_0,y_1,\ldots\right) : y_j \in \Gamma_n^{i} \mathrm{\ for\ any\ }j \in \Nb  \right\}, \]
    \[\Sigma_{r} = \left\{ \overline{y}  = \left(y_0,y_1,\ldots\right) : y_j \in \Gamma_n \mathrm{\ for\ any\ }j \in \Nb  \right\}. \]
    Then $\Sigma_{r_i}^i \subseteq \Sigma_r$ and $\left(\Sigma_{r_i}^i,\sigma|_{\Sigma_{r_i}^i}\right)$ and $\left(\Sigma_{r},\sigma|_{\Sigma_r}\right)$ are one-sided full shifts for any $1 \leq i \leq m$.

    Denote $\mathbf{C}_{\overline{y} } = \left\langle \mathbf{C}_{y_0}^n \mathbf{C}_{y_1}^n \ldots\right\rangle = \left\langle y_0^*, y_1^*, \ldots\right\rangle $ as the $\delta$-pseudo-orbit generated by $\overline{y} $ for any $\overline{y}  \in \Sigma_{r}$. Define
    \[Y_{\overline{y} } = \left\{z \in X : d\left(f^j z, y_j^*\right) \leq \ve \mathrm{\ for\ any\ } j \in \Nb \right\}.\]
    Since $(X,f)$ satisfies the shadowing property and $\ve < c / 3$, one has $\left\lvert Y_{\overline{y} }\right\rvert = 1$ for any $\overline{y}  \in \Sigma_{r}$. Since $\Gamma_n$ is $(n, 3\ve)$-separated, $Y_{\overline{y}} \cap Y_{\overline{w}} = \emptyset$ for any $\overline{y}, \overline{w} \in \Sigma_{r}$ with $\overline{y} \ne \overline{w}$. Define $\pi: \Delta \to \Sigma_r$ by $\pi(z) = \overline{y} $ for any $z \in Y_{\overline{y}}$ where $\Delta = \bigsqcup_{\overline{y} \in \Sigma_{r}} Y_{\overline{y}}$. Since each $\mathbf{C}_{\overline{y}}$ consists of $\delta$-chains of length $n$, we have $\Delta$ is $f^n$-invariant. Similarly, define $\pi_i: \Delta_i \to \Sigma_{r_i}^i$ by $\pi_i(z) = \overline{y} $ for any $z \in Y_{\overline{y}}$ where $\Delta_i = \bigsqcup_{\overline{y} \in \Sigma_{r_i}^i} Y_{\overline{y}}$, and we have $\Delta_i$ is $f^n$-invariant for any $1 \leq i \leq m$.

    First, we prove that $\pi$ is a topological conjugation between $\left(\Delta, f^n\right)$ and $\left(\Sigma_r, \sigma\right)$. Since $\left\lvert Y_{\overline{y} }\right\rvert = 1$ for any $\overline{y}  \in \Sigma_{r}$, one has $\pi$ is bijective. To prove $\pi$ is a homeomorphism, it suffices to prove that $\pi^{-1} = Y_{\left(\cdot\right)}: \Sigma_r \to \Delta$ is continuous. Assume $\left\{\overline{y^k}\right\}_{k=1}^\infty \subseteq \Sigma_r$ satisfies $\lim_{k \to \infty} \overline{y^k} = \overline{y}$. Then $\lim_{k \to \infty} y^k_i = y_i$ for any $i \in \Nb$. Since $\Sigma_r$ is compact, $\overline{y} \in \Sigma_r$. Since $\ve < c / 3$, one has  $\lim_{k \to \infty} \pi^{-1}\left(\overline{y^k}\right) = \pi^{-1}\left(\overline{y}\right)$.
    Similarly, $\pi_i$ is a conjugation between $\left(\Delta_i, f^n\right)$ and $\left(\Sigma_{r_i}^i, \sigma\right)$ for any $1 \leq i \leq m$, which proves item (1) in the theorem.

    Second, we prove that $\htop(f,\Lambda_i) > h_{\mu_i}(f) - \veh$.
    \begin{align*}
        \htop(f,\Lambda_i) & = \htop(f^n, \Lambda_i) / n  \geq \htop(f^n, \Delta_i) / n                                           \\
                           & = \htop(\sigma, \Sigma^i_{r_i}) / n  = \left(\log\left\lvert \Gamma_{n}^{i} \right\rvert \right) / n \\
                           & > h_{\mu_i} - \veh_0  \ge h_{\mu_i} - \veh,
    \end{align*}
    which proves item (2) in the theorem.

    Next, we prove that $D_{\mathrm{H}}\left(K, \M(f, \Lambda)\right) < \veH$ and $D_{\mathrm{H}}\left(\left\{\mu_i\right\}, \M(f,\Lambda_i)\right) < \veH$ for any $1 \leq i \leq m$. On one hand, given $\nu_i \in \Merg(f,\Lambda_i)$, one has $\nu_i\left(G_{\nu_i}\right) = 1$ and $\nu_i\left(\Delta_i\right) \geq 1/n > 0$. Pick $z_i \in G_{\nu_i} \cap \Delta_i$ and hence $z_i \in Y_{\overline{y}}$ for some $\overline{y} = \left(y_0, y_1, \dots\right) \in \Sigma_{r_i}^i$. Then for any $t \in \Nb$,
    \begin{align*}
        D\left(\cE_n\left(f^{tn}z_i\right), \mu_i\right) & \leq D\left(\cE_n\left(f^{tn}z_i\right), \cE_n\left(y_t\right)\right) + D\left(\cE_n\left(y_t\right), \mu_i\right) \\
                                                         & <  \ve + \veH_0 / 4 < \veH_0 / 2.
    \end{align*}
    Besides, $\nu_i = \lim_{t \to \infty} \cE_{tn}\left(z_i\right) $. Therefore, $D\left(\mu_i, \nu_i\right) \leq \veH_0 / 2$. By the ergodic decomposition theorem, $\M\left(f,\Lambda_i\right) \subseteq  \overline{B\left(\mu_i, \veH_0 / 2\right)}$ and hence $D_{\mathrm{H}}\left(\left\{\mu_i\right\} , \M(f,\Lambda_i)\right) \leq \veH_0 / 2$. Since $K$ is convex and $\Lambda_i \subseteq \Lambda$, one has $K \subseteq B\left(\M(f,\Lambda), \veH_0\right)$ and hence $K \subseteq B\left(\M(f,\Lambda), \veH\right)$. On the other hand, given $\mu \in \Merg(f,\Lambda)$, pick $z \in G_\mu \cap \Delta$. There is $\overline{y} \in \Sigma_r$ such that $\mathbf{C}_{\overline{y}}$ is $\ve$-shadowed by $z$. By a process similar to the above, there exists $\nu \in K$ such that $D\left(\mu, \nu\right) < \veH_0 / 2$.
    By the ergodic decomposition theorem, $\M(f,\Lambda) \subseteq B\left(K, \veH_0\right) \subseteq B\left(K, \veH\right)$. Therefore, $D_{\mathrm{H}}\left(K, \M(f,\Lambda)\right) < \veH$, which proves item (3) in the theorem.

    Finally, we prove item (4). For any $z \in \Lambda$, there exist
    $0 \leq q \leq n-1$ and $w \in \Delta$ such that $z=f^q w$.
    Let $j=n-q$. If
    \[
        \pi(w)=\left(y_0,y_1,\dots\right)\in\Sigma_r,
    \]
    then, for any $t\in\Nb$,
    \[
        \begin{aligned}
            d\left(f^{j+tn}z,x\right)
             & =d\left(f^{(t+1)n}w,x\right)           \\
             & \leq d\left(f^{(t+1)n}w,y_{t+1}\right)
            +d\left(y_{t+1},x\right)                  \\
             & \leq \ve+\delta/2<\ved.
        \end{aligned}
    \]
\end{proof}

\section{Conditional entropy realization for topologically expanding systems}\label{sec:cmie}
In this section, we establish exact entropy-and-average realization and weak* approximation by uniquely ergodic measures with compact support for topologically expanding maps and topologically Anosov homeomorphisms. Thus Theorems \ref{thmA}, \ref{thmB} and
\ref{thm:cmed_nonuniform} are three implementations of the same
nested-horseshoe mechanism rather than independent constructions; the
system-dependent work lies in producing the horseshoe input and the
auxiliary estimates required by the induction.

\begin{theorem}\label{thmA}
    Suppose $f$ is a topologically transitive and topologically expanding (topologically Anosov) continuous map (homeomorphism) on a compact metric space $(X,d)$. Let $\vphi \in C(X)$ such that $\mathrm{Int} L_\vphi \ne \emptyset$. For any $\ve > 0$, any $a \in \mathrm{Int} L_\vphi$, any $\mu \in \M(f,X)$ satisfying $\int\vphi\dif\mu = a$ and $h_\mu(f) < \sup\left\{h_\nu(f): \nu \in \M(f,X), \int \varphi \dif \nu = a\right\} $, and any $0 \leq h \leq h_\mu(f)$, there exists a compact $f$-invariant set $Y \subseteq X$ such that
    \begin{enumerate}
        \item $\left(Y,f\right) $ is minimal and uniquely ergodic;
        \item $h_{\nu}(f) = \htop(f,Y) = h$ where $\nu$ is the unique $f$-invariant measure on $Y$;
        \item $\int \vphi \dif \nu = a$;
        \item $D(\nu,\mu) < \ve$.
    \end{enumerate}
\end{theorem}

Before proving this theorem, we need the following lemma.

\begin{lemma}\label{lemTau}
    Suppose $f:X \to X$ is a continuous map on a compact metric space $(X,d)$, $k \in \Nb^+$ and $\Delta \subseteq X$ is compact and $f^k$-invariant. The map $\tau: \M\left(\Delta\right)  \to \M(\Lambda)$ is defined by
    \[\tau\mu = \frac{1}{k}\left(\mu + \mu\circ f^{-1} + \ldots + \mu\circ f^{-(k-1)}\right)  \]
    where $\Lambda = \bigcup_{i=0}^{k-1}f^i\left(\Delta\right) $. Then the following statements hold.
    \begin{enumerate}
        \item $\tau$ is continuous and affine.
        \item $\tau\left(\M\left(f^k, \Delta\right)\right) = \M(f,\Lambda)$.
        \item $h_{\tau\mu}(f^k) = h_{\mu}(f^k)$ for any $\mu \in \M\left(f^k,\Delta\right)$.
    \end{enumerate}
\end{lemma}
\begin{proof}
    The continuity of $\tau$ is a direct corollary of \cite[Theorem 6.7]{walters1982introduction}. The affinity of $\tau$ can be easily checked. Then item (1) is proved.

    Given $\mu \in \Merg\left(f^k, \Delta\right) \subseteq \Merg\left(f^k, \Lambda\right)$ and $A \subseteq \Lambda$ with $f^{-1}A = A$ and hence $f^{-k}A = A$. Then $\mu(A)$ is $0$ or $1$. We have
    \[\tau\mu\left(A\right) = \frac{1}{k}\left(\mu(A) + \mu\left(f^{-1} A\right)  + \ldots + \mu \left(f^{-(k-1)} A\right) \right) = \mu(A).\]
    Therefore, $\tau\mu(A)$ is $0$ or $1$, which proves that $\tau\left(\Merg\left(f^k, \Delta\right)\right) \subseteq \Merg(f,\Lambda)$. On the other hand, given $\mu \in \Merg(f,\Lambda)$, we have $\mu(\Delta) > 0$ and $\mu\left(G_\mu\right) = 1$. Take $x \in G_\mu \cap \Delta$ where $G_\mu$ are defined in Section \ref{subsec:M(X)}. Let
    \[\nu = \lim_{l \to \infty} \frac{1}{n_l} \sum_{i=0}^{n_l -1} \delta_{f^{ki}x}\]
    and we have $\nu \in V_{f^k}(x) \subseteq \M\left(f^k, \Delta\right)$.  Then
    \[\tau \nu = \lim_{l \to \infty} \frac{1}{n_l} \sum_{i=0}^{n_l -1} \tau\delta_{f^{ki}x} = \lim_{l \to \infty} \frac{1}{kn_l} \sum_{i=0}^{kn_l -1} \delta_{f^{i}x} = \mu,\]
    which implies that $\Merg(f,\Lambda) \subseteq \tau\left(\M\left(f^k, \Delta\right)\right)$. Since $\tau$ is continuous and affine, we have $\tau\left(\M\left(f^k, \Delta\right)\right) = \M(f,\Lambda)$ by ergodic decomposition. Then item (2) is proved.

    Item (3) comes from the affinity of entropy function and \cite[Lemma 18.5]{denker1976ergodic} (although the original result was stated for homeomorphisms, it is easy to see that it remains valid for continuous maps).
\end{proof}
\begin{remark}\label{rmk:tau-barycenter}
    For every \(\mu\in\M(\Delta)\), the map \(\tau\) in
    Lemma \ref{lemTau} admits the barycentric representation
    \[
        \tau\mu
        =\frac{1}{k}\sum_{i=0}^{k-1}\mu\circ f^{-i}
        =\int_\Delta \cE_k(x)\,\dif\mu(x).
    \]
    Here the last integral is understood in the weak sense. Indeed,
    for every \(\psi\in C(X)\),
    \[
        \begin{split}
            \int\psi\,\dif\left(\int_\Delta
            \cE_k(x)\,\dif\mu(x)\right)
             & =\int_\Delta\frac{1}{k}
            \sum_{i=0}^{k-1}\psi(f^ix)\,\dif\mu(x) \\
             & =\int\psi\,\dif(\tau\mu).
        \end{split}
    \]
    Consequently, if \(\nu\in\M(f,\Lambda)\) and
    \(\mu\in\M(f^k,\Delta)\) satisfies \(\tau\mu=\nu\), then
    \[
        \nu=\int_\Delta\cE_k(x)\,\dif\mu(x).
    \]
    In particular, the convexity of the first Wasserstein distance
    gives
    \[
        D(\nu,\alpha)
        \leq\int_\Delta
        D\bigl(\cE_k(x),\alpha\bigr)\,\dif\mu(x)
    \]
    for every \(\alpha\in\M(X)\).
\end{remark}

For any $\mu, \nu \in \M(X)$, define $\left[\mu,\nu\right] := \left\{\theta\mu + (1-\theta)\nu:\theta \in [0,1]\right\} $ be the convex hull of $\mu$ and $\nu$. Given $x \in X$, we say $x$ is a periodic point if $f^n(x) = x$ for some $n \in \Nb^+$. We call $\mu \in \M(f,X)$ a periodic measure if $\supp(\mu) = \Orb(x)$ for some periodic point $x$. Now, we prove Theorem \ref{thmA}.

\begin{proof}[Proof of Theorem \ref{thmA}]
    We only deal with the case of maps. The case of homeomorphisms is similar.

    Fix $\ve > 0$, $a \in \Int L_\varphi$, $\mu \in \M(f,X)$ satisfying $\int\varphi\dif\mu = a$ and
    \[
        h_\mu(f) < \sup\left\{h_\omega(f): \omega \in \M(f,X), \int \varphi \dif \omega = a\right\}
    \]
    Fix $0 \leq h \leq h_\mu(f)$. We prove the theorem by countable steps. Note for measures and subsets, we use superscript $k$ to denote sets in Step $k$ and subscripts $\left\{1,2\right\} $ to denote different horseshoes in the same step with respect to different measures.

    \textbf{Step (1).} Since $a \in \Int L_\vphi$, there exist $\mu_{l}^1$ and $\mu_{r}^1$ such that $\int \vphi \dif \mu_{l}^1 < a < \int \vphi \dif \mu_{r}^1$.
    Since
    \[h_\mu(f) < \sup\left\{h_\omega(f):\omega \in \M(f,X), \int\vphi\dif\omega=a\right\},\]
    there exists $\mu_{a,s}^1 \in \M(f,X)$ such that
    \[\int \vphi \dif \mu_{a,s}^1 = a\]
    and
    \[h_\mu(f) < h_{\mu_{a,s}^1}(f) \leq \sup\left\{h_\omega(f):\omega \in \M(f,X), \int\vphi\dif\omega=a\right\}.\]
    Take $\mu_{a}^1 \in \left(\mu,\mu_{a,s}^1\right] \cap B\left(\mu,\ve / 6\right)$. Then $h_{\mu_{a}^1}(f) > h$ and $\int\vphi\dif\mu_{a}^1 = a$.
    Fix $\delta = h_{\mu_a^1}(f) - h > 0$. By Lemma \ref{lemSeparatedSet}, periodic measures are dense in $\M(f,X)$. Then there exist two periodic measures whose integrals with respect to $\varphi$ are greater than and less than $a$, respectively. By convex combination and the density of periodic measures, there exists $\mu_0^1 \in \M(f,X)$ satisfying $h_{\mu_0^1}(f) = 0$, $\int\vphi\dif\mu_0^1 = a$ and $\mu_0^1 \in B(\mu,\ve/6)$. Therefore, there exists $\mu^1 \in \left[\mu_0^1, \mu_{a}^1\right] \subseteq B(\mu, \ve/6)$ such that $h_{\mu^1}(f) = h + 5\delta/8$. There exist $\mu_{1}^1 \in \left[\mu_{l}^1, \mu^1\right]$, $\mu_{2}^1 \in \left[\mu^1, \mu_{r}^1\right]$ such that
    \begin{description}
        \item[(a)\label{1.a}] $a - \delta < \int \vphi \dif \mu_{1}^1 < a < \int \vphi \dif \mu_{2}^1 < a + \delta$;
        \item[(b)\label{1.b}] $h + \delta / 2 < h_{\mu_{j}^1}(f) < h + 3\delta / 4$ for any $j \in \left\{1,2\right\}$;
        \item[(c)\label{1.c}] $D\left(\mu_{1}^1, \mu_{2}^1\right) < \ve / 6$, $D\left(\mu_1^1,\mu^1\right) < \ve / 6$, $D\left(\mu_2^1,\mu^1\right) < \ve / 6$.
    \end{description}
    Then there exists $0 < \eta_1 < \ve / 6$ such that
    \begin{description}
        \item[(d)\label{1.e}] $\int \vphi \dif \nu < a \mathrm{\ for\ any\ } \nu \in B\left(\mu_{1}^1,\eta_1\right),\quad \int \vphi \dif \nu > a \mathrm{\ for\ any\ } \nu \in B\left(\mu_{2}^1,\eta_1\right)$.
    \end{description}
    By Theorem \ref{thmC}, there exists $n_1 \in \Nb^+$ and three $f^{N_1}$-invariant compact subsets $\Delta^1_1, \Delta^1_2, \Delta^{1}$ with $\Delta^1_1, \Delta^1_2 \subseteq \Delta^{1}$ where $N_1 := n_1$ such that for any $j \in \left\{1,2\right\} $,
    \begin{description}
        \item[(1)\label{1.1}] $\left(\Delta^1, f^{N_1}\right)$ and $\left(\Delta^1_j, f^{N_1}\right)$ are topologically conjugate to a one-sided full shift respectively;
        \item[(2)\label{1.2}]
              \[D_{\mathrm{H}}\left(\left[\mu_{1}^1,\mu_{2}^1\right], \M \left(f, \bigcup_{i=0}^{n_1 - 1} f^i \left(\Delta^1\right) \right)   \right) < \eta_1,\]
              \[D_{\mathrm{H}}\left(\left\{\mu_{j}^1\right\}, \M \left(f, \bigcup_{i=0}^{n_1 - 1} f^i \left(\Delta^1_j\right) \right) \right) < \eta_1;\]
        \item[(3)\label{1.3}]
              \begin{align*}
                  h + \delta/4 & < \htop\left(f^{N_1}, \Delta^1_j\right) / N_1 = \htop\left(f, \bigcup_{i=0}^{n_1 - 1} f^i \left(\Delta^1_j\right) \right) \\ & \leq \htop\left(f, \bigcup_{i=0}^{n_1 - 1} f^i \left(\Delta^1\right) \right).
              \end{align*}
    \end{description}
    Let $\Lambda^1 = \bigcup_{i=0}^{N_1 - 1} f^i \left(\Delta^1\right) $. By item \ref{1.c} and item \ref{1.2}, $\diam \M \left(f, \Lambda^1\right) < 2\eta_1 + \ve/ 6  < \ve/2$.

    \textbf{Step (2).} Let $\tau_1 : \M\left(\Delta^1\right)  \to \M(\Lambda^1)$ be the map defined by
    \[\tau_1\nu = \frac{1}{N_1}\left(\nu + \nu\circ f^{-1} + \ldots + \nu\circ f^{-(N_1-1)}\right).\]
    By Lemma \ref{lemTau}, $\tau_1$ is continuous. Then there exists $0 < \gamma_2 < \ve/6 $ such that
    \[D\left(\tau_1\nu,\tau_1\omega\right) < \ve/6 \mathrm{\ for\ any\ } \nu, \omega \in \M\left(\Delta^1\right) \mathrm{\ with\ } D\left(\nu, \omega\right) < \gamma_2.\]

    \begin{claim}
        There exists $\mu_{l}^2 \in \M\left(f^{N_1}, \Delta^1_1\right)$ such that
        \[h + \delta / 4 < h_{\mu_{l}^2}\left(f^{N_1}\right) / N_1  < h + 3\delta/8 \]
        and
        \[\int \vphi \dif \left(\tau_1\mu_{l}^2\right)  < a.\]
    \end{claim}
    \begin{proof}[Proof of the claim]
        By the classical variational principle and item \ref{1.3} in Step (1), there exists $\xi \in \M(f,\bigcup_{i=0}^{n_1 - 1} f^i \left(\Delta^1_1\right))$ such that $h_\xi(f) > h + \delta/4$. By item \ref{1.e} and item \ref{1.2} in Step (1), $\int \vphi \dif \xi  < a$. By Lemma \ref{lemTau}, there exists $\xi_1 \in \M(f^{N_1},\Delta^1_1)$ such that $\tau_1\xi_1 = \xi$. Since $\left(\Delta^1_1, f^{N_1}\right)$ is topologically conjugate to a one-sided full shift, periodic measures are dense in $\M(f^{N_1},\Delta_1^1)$. Then there exists $\xi_0$ such that $\int\vphi\dif\left(\tau_1\xi_0\right) < a$ and $h_{\xi_0}\left(f^{N_1}\right)  = 0$. By convex combination of $\xi_0$ and $\xi_1$, we complete the proof of the claim.
    \end{proof}

    Similarly, there exists $\mu_{r}^2 \in \M\left(f^{N_1}, \Delta^1_2\right)$ such that
    \[h + \delta / 4 < h_{\mu_{r}^2}\left(f^{N_1}\right) / N_1  < h + 3\delta/8,\quad \int \vphi \dif \left(\tau_1\mu_{r}^2\right)  > a.\]
    By Lemma \ref{lemTau}, $\tau_1$ is affine. Similar to Step (1), there exist $\mu_{1}^2, \mu_{2}^2 \in \left[\mu_{l}^2, \mu_{r}^2\right] \subseteq \M\left(f^{N_1}, \Delta^1\right)$ such that
    \begin{description}
        \item[(a)\label{2.a}] $a - \delta / 2 < \int \vphi \dif \left(\tau_1\mu_{1}^2\right)  < a < \int \vphi \dif \left(\tau_1\mu_{2}^2\right)  < a + \delta / 2$;
        \item[(b)] $h + \delta / 4 < h_{\mu_{j}^2}\left(f^{N_1}\right) / N_1 < h + 3 \delta / 8$ for any $j \in \left\{1,2\right\}$;
        \item[(c)\label{2.c}] $D\left(\mu_{1}^2, \mu_{2}^2\right) < \gamma_2 / 3$.
    \end{description}
    By Lemma \ref{lemFullShiftsProperties}, $\left(\Delta^1, f^{N_1}\right) $ is positively expansive and hence the entropy function with respect to $f^{N_1}$ is upper semi-continuous by Lemma \ref{lemUpperSemiContinuity}. By Lemma \ref{lemTau}, $h_{\tau_1\mu_{j}^2}\left(f^{N_1}\right) = h_{\mu_{j}^2}\left(f^{N_1}\right)$ for any $j \in \left\{1,2\right\} $. Then there exists $0 < \eta_2 < \gamma_2 / 3$ such that
    \begin{description}
        \item[(d)] $h_{\nu}\left(f^{N_1}\right) / N_1 < h + \delta / 2$ for any $f^{N_1}$-invariant measure $\nu \in B\left( \left[\mu_{1}^2, \mu_{2}^2\right]  , \eta_2\right)$;
        \item[(e)]
              \[\int \vphi \dif \left(\tau_1\nu\right)  < a \mathrm{\ for\ any\ } \nu \in B\left(\mu_{1}^2,\eta_2,\right)\]
              and
              \[\int \vphi \dif \left(\tau_1\nu\right)  > a \mathrm{\ for\ any\ } \nu \in B\left(\mu_{2}^2,\eta_2\right).\]
    \end{description}
    By Lemma \ref{lemFullShiftsProperties} and item (1) in Step (1), $\left(\Delta^1, f^{N_1}\right)$ is topologically transitive and satisfies the shadowing property. By Theorem \ref{thmC}, there exists $n_2 \in \Nb^+$ and three $f^{N_2}$-invariant compact subsets $\Delta^2_1, \Delta^2_2, \Delta^{2}$ with $\Delta^2_1, \Delta^2_2 \subseteq \Delta^{2}$ where $N_2 := n_1n_2$ such that for any $j \in \left\{1,2\right\} $,
    \begin{description}
        \item[(1)] $\left(\Delta^2, f^{N_2}\right)$ and $\left(\Delta^2_j, f^{N_2}\right)$ are topologically conjugate to a one-sided full shift respectively;
        \item[(2)\label{2.2}]
              \[D_{\mathrm{H}}\left(\left[\mu_{1}^2,\mu_{2}^2\right], \M \left(f^{N_1}, \bigcup_{i=0}^{n_2 - 1} f^{iN_1} \left(\Delta^2\right) \right)   \right) < \eta_2\]
              and
              \[D_{\mathrm{H}}\left(\left\{\mu_{j}^2\right\} , \M \left(f^{N_1}, \bigcup_{i=0}^{n_2 - 1} f^{iN_1}\left(\Delta^2_j\right) \right) \right) < \eta_2;\]
        \item[(3)\label{2.3}]
              \begin{align*}
                  h + \delta / 8 & < \htop\left(f^{N_2}, \Delta^2_j\right) / N_2  \leq \htop\left(f^{N_2}, \Delta^2\right) / N_2 \\ & = \htop\left(f^{N_1}, \bigcup_{i=0}^{n_2 - 1} f^{iN_1}\left(\Delta^2\right)   \right) / N_1 < h + \delta/2.
              \end{align*}
    \end{description}
    Let $\Lambda^2 = \bigcup_{i=0}^{N_2 - 1} f^i \Delta^2$. By Lemma \ref{lemTau} and item \ref{2.3} in Step (2), $h + \delta / 8 < \htop\left(f,\Lambda^2\right) < h+\delta/2$. By item \ref{2.c} and item \ref{2.2} in Step (2),
    \[\diam \M \left(f^{N_1}, \bigcup_{i=0}^{n_2 - 1} f^{iN_1} \left(\Delta^2\right) \right) < 2\eta_2 + \gamma_2 / 3 < \gamma_2\]
    and hence
    \[\diam \M \left(f, \Lambda^2\right) = \diam\left(\tau_1\M \left(f^{N_1}, \bigcup_{i=0}^{n_2 - 1} f^{iN_1} \left(\Delta^2\right) \right)\right)  < \ve / 6.\]

    \textbf{Step (k).} Let $\tau_{k-1} : \M\left(f^{N_{k-1}}, \Delta^{k-1}\right)  \to \M(f,X)$ be the map defined by
    \[\tau_{k-1}\nu = \frac{1}{N_{k-1}}\left(\nu + \nu\circ f^{-1} + \ldots + \nu\circ f^{-(N_{k-1}-1)}\right).\]
    By Lemma \ref{lemTau}, $\tau_{k-1}$ is continuous. Then there exists $0 < \gamma_k < \ve / \left(3\cdot2^{k-1}\right)  $ such that
    \[D\left(\tau_{k-1}\nu,\tau_{k-1}\omega\right) < \ve / \left(3\cdot2^{k-1}\right)\]
    for any $\nu, \omega \in \M\left(f^{N_{k-1}}, \Delta^{k-1}\right)$ with $D\left(\nu, \omega\right) < \gamma_k$.
    Similar to the previous steps, there exists $\mu_{l}^k \in \M\left(f^{N_{k-1}}, \Delta^{k-1}_1\right)$ such that
    \[h + \delta / 2^k < h_{\mu_{l}^k}\left(f^{N_{k-1}}\right) / N_{k-1} < h + 3\delta / 2^{k+1},\quad \int \vphi \dif \left(\tau_{k-1}\mu_{l}^k\right) < a.\]
    Similarly, there exists $\mu_{r}^k \in \M\left(f^{N_{k-1}}, \Delta^{k-1}_2\right)$ such that
    \[h + \delta / 2^k < h_{\mu_{r}^k}\left(f^{N_{k-1}}\right) / N_{k-1} < h + 3\delta / 2^{k+1}, \quad \int \vphi \dif \left(\tau_{k-1}\mu_{r}^k\right)  > a.\]
    By Lemma \ref{lemTau}, $\tau_{k-1}$ is affine. Similar to Step (1), there are $\mu_{1}^k, \mu_{2}^k \in \left[\mu_{l}^k, \mu_{r}^k\right] \subseteq \M\left(f^{N_{k-1}}, \Delta^{k-1}\right)$ such that
    \begin{description}
        \item[(a)] $a - \delta / 2^{k-1} < \int \vphi \dif \left(\tau_{k-1}\mu_{1}^k\right)  < a < \int \vphi \dif \left(\tau_{k-1}\mu_{2}^k\right)  < a + \delta / 2^{k-1}$;
        \item[(b)] $h + \delta / 2^k < h_{\mu_{j}^k}\left(f^{N_{k-1}}\right) / N_{k-1} < h + 3 \delta / 2^{k+1}$ for any $j \in \left\{1,2\right\}$;
        \item[(c)\label{k.c}] $D\left(\mu_{1}^k, \mu_{2}^k\right) < \gamma_k / 3$.
    \end{description}
    By Lemma \ref{lemFullShiftsProperties}, $\left(\Delta^{k-1}, f^{N_{k-1}}\right)$ is positively expansive and hence the entropy function with respect to $f^{N_{k-1}}$ is upper semi-continuous by Lemma \ref{lemUpperSemiContinuity}. Then there exists $0 < \eta_k < \gamma_k / 3$ such that
    \begin{description}
        \item[(d)] $h_{\nu}\left(f^{N_{k-1}}\right) / N_{k-1} < h + \delta / 2^{k-1}$ for any $f^{N_{k-1}}$-invariant measure $\nu \in B\left( \left[\mu_{1}^{k}, \mu_{2}^k\right]  , \eta_k\right)$;
        \item[(e)] \[\int \vphi \dif \left(\tau_{k-1}\nu\right)  < a \mathrm{\ for\ any\ } \nu \in B\left(\mu_{1}^k,\eta_k\right),\]
              \[\int \vphi \dif \left(\tau_{k-1}\nu\right)  > a \mathrm{\ for\ any\ } \nu \in B\left(\mu_{2}^k,\eta_k\right).\]
    \end{description}
    By Lemma \ref{lemFullShiftsProperties} and item (1) in Step $(k-1)$, $\left(\Delta^{k-1}, f^{N_{k-1}}\right)$ is topologically transitive and satisfies the shadowing property. By Theorem \ref{thmC}, there exists $n_k \in \Nb^+$ and three $f^{N_k}$-invariant compact subsets $\Delta^k_1, \Delta^k_2, \Delta^{k}$ with $\Delta^k_1, \Delta^k_2 \subseteq \Delta^{k}$ where $N_k := N_{k-1}n_k$ such that for any $j \in \left\{1,2\right\} $,
    \begin{description}
        \item[(1)] $\left(\Delta^k, f^{N_k}\right)$ and $\left(\Delta^k_j, f^{N_k}\right)$ are topologically conjugate to a one-sided full shift respectively;
        \item[(2)\label{k.2}] \[D_{\mathrm{H}}\left(\left[\mu_{1}^k,\mu_{2}^k\right], \M \left(f^{N_{k-1}}, \bigcup_{i=0}^{n_k - 1} f^{iN_{k-1}} \left(\Delta^k\right) \right)   \right) < \eta_k\]
              and
              \[D_{\mathrm{H}}\left(\left\{\mu_{j}^k\right\} , \M \left(f^{N_{k-1}}, \bigcup_{i=0}^{n_k - 1} f^{iN_{k-1}}\left(\Delta^k_j\right) \right) \right) < \eta_k;\]
        \item[(3)\label{k.3}] \begin{align*}
                  h + \delta / 2^{k+1} & < \htop\left(f^{N_k}, \Delta^k_j\right) / N_k \leq \htop\left(f^{N_k}, \Delta^k\right) / N_k \\ & = \htop\left(f^{N_{k-1}}, \bigcup_{i=0}^{n_k - 1} f^{iN_{k-1}}\left(\Delta^k\right) \right) / N_{k-1} < h + \delta / 2^{k-1}.
              \end{align*}
    \end{description}
    Let $\Lambda^k = \bigcup_{i=0}^{N_k - 1} f^{i}\Delta^k$. By Lemma \ref{lemTau} and item \ref{k.3} in Step $(k)$, similar to Step (2),
    \[h + \delta / 2^{k+1} < \htop\left(f, \Lambda^k\right) < h + \delta / 2^{k-1}.\]
    By item \ref{k.c} and item \ref{k.2} in Step $(k)$,
    \[\diam \M \left(f^{N_{k-1}}, \bigcup_{i=0}^{n_k - 1} f^{iN_{k-1}} \left(\Delta^k\right) \right) < 2\eta_k + \gamma_k/3 <  \gamma_k\]
    and hence
    \[\diam \M \left(f, \Lambda^k\right) = \diam \left(\tau_{k-1} \M \left(f^{N_{k-1}}, \bigcup_{i=0}^{n_k - 1} f^{iN_{k-1}} \left(\Delta^k\right) \right)\right)  < \ve / \left(3\cdot2^{k-1}\right).\]

    \textbf{After countable steps.} We obtain a sequence of subsets $\left\{\Lambda^k\right\}_{k=1}^{\infty}$ satisfying $\Lambda^{k+1} \subseteq \Lambda^k$ for any $k \in \Nb^+$. Let $\Lambda = \bigcap_{k=1}^{\infty} \Lambda^k$ and hence $\M\left(f,\Lambda\right) \subseteq \M\left(f, \Lambda^k\right)$ for any $k \in \Nb^+$. Moreover, $\lim_{k \to \infty} \diam \M\left(f, \Lambda^k\right) = 0$. Therefore, $\M\left(f,\Lambda\right)$ is a single $f$-invariant measure $\left\{\nu\right\}$. Let $Y = \supp\left(\nu\right)$. Since $Y \subseteq \Lambda$, one has $\M(f,Y)=\{\nu\}$, so $\left(Y,f|_Y\right)$ is uniquely ergodic. By \cite[Theorem 6.17]{walters1982introduction}, it is also minimal, and hence item (1) in Theorem \ref{thmA} is satisfied.

    Since $\lim_{k \to \infty} \diam \M\left(f, \Lambda^k\right) = 0$, one has $\nu = \lim_{k \to \infty} \tau_{k}\mu^{k+1}_1$. Since the entropy map is upper semi-continuous, by item (b) in each step,
    \[h_\nu\left(f\right) \geq \limsup_{k \to \infty} h_{\tau_{k}\mu_1^{k+1}}(f) = \limsup_{k \to \infty} h_{\mu_1^{k+1}}\left(f^{N_{k}}\right) / N_{k}  \geq h.\]
    On the other hand, by item (3) in each step, $\htop\left(f,\Lambda\right) \leq \htop\left(f,\Lambda^{k+1}\right) \leq h + \delta / 2^{k}$ for any $k \in \Nb^+$ and hence $\htop\left(f,\Lambda\right) \leq h$. Note $\nu$ is the unique measure with maximal entropy and hence $\htop(f, \Lambda) =  h_\nu(f) = h$. Therefore, item (2) in Theorem \ref{thmA} is satisfied.

    Since $\lim_{k \to \infty} \diam \M\left(f, \Lambda^k\right) = 0$, one has $\nu = \lim_{k \to \infty} \tau_{k}\mu^{k+1}_1 = \lim_{k \to \infty} \tau_{k}\mu^{k+1}_2$. By the item (a) in each step, item (3) in Theorem \ref{thmA} is satisfied.

    Finally,
    \begin{align*}
        D\left(\mu,\nu\right) & \leq D\left(\mu,\mu^1\right) + D\left(\mu^1,\mu^1_1\right) + D\left(\mu^1_1,\nu\right) \\
                              & < \ve/6 + \ve/6 + \eta_1 + \diam \M \left(f, \Lambda^1\right)                          \\
                              & < \ve/6 + 5\ve/6 = \ve,
    \end{align*}
    which means that item (4) in Theorem \ref{thmA} is satisfied.
\end{proof}

Next, we prove the corollaries of Theorem \ref{thm:A} for topologically expanding maps. First, we recall the topological spectral decomposition theorem, see \cite[Theorem 3.4.4]{aoki1980topological}.

\begin{theorem}[Topological Spectral Decomposition Theorem]\label{thm:top_spec_decomp}
    Let $(X,f)$ be a positively expansive dynamical system with the shadowing property and assume that $f$ is surjective. Then the following properties hold:
    \begin{enumerate}
        \item $\Omega(f)$ contains a finite sequence of pairwise disjoint $f$-invariant closed subsets $\left\{B_i\right\}_{i=1}^n$ such that $\Omega(f) = \bigcup_{i=1}^n B_i$ and each subsystem $\left(B_i, f\right)$ is topologically transitive (each $B_i$ is called a basic set).
        \item For every basic set $B$, there exists an integer $m > 0$ and a finite sequence of pairwise disjoint $f^m$ invariant closed sets $\left\{C_i\right\}_{i=0}^{m-1}$ such that $B = \bigcup_{i=0}^{m-1} C_i$, $f\left(C_i\right) = C_{i+1\ (\mathrm{ mod\ } m)}$ and each subsystem $\left(C_i, f^m\right)$ is topologically mixing (each $C_i$ is called an elementary set).
    \end{enumerate}
\end{theorem}

\begin{proof}[Proof of Theorem \ref{thm:2.2}]
    First suppose that \((X,f)\) satisfies (H). Take
    \(\varphi\equiv0\). Then $L_\varphi=\{0\}$, so the integral constraint is automatic. The degenerate case of Theorem \ref{thm:A} therefore gives, for every $0\leq h<\htop(f,X)$, a uniquely ergodic measure $\mu$ with compact support such that $h_\mu(f)=h$. Let $Y=\supp(\mu)$. Then $(Y,f|_Y)$ is minimal and uniquely ergodic, and the variational principle gives $\htop(f,Y)=h_\mu(f)=h$.

    It remains to prove the stronger assertion in the preceding
    remark for a not necessarily transitive system satisfying (H1). By \cite[Corollary 1(i)]{moothathu2011implications}, we have $f|_{\Omega(f)}:\Omega(f) \to \Omega(f)$ is surjective. By \cite[Lemma 1]{moothathu2011implications}, we have $\left(\Omega(f), f\right)$ has the shadowing property. By Theorem \ref{thm:top_spec_decomp}, there exists a finite sequence of pairwise disjoint $f$-invariant closed subsets $\left\{B_i\right\}_{i=1}^l$ of $\Omega\left(f|_{\Omega(f)}\right)$ such that $\Omega\left(f|_{\Omega(f)}\right) = \bigcup_{i=1}^l B_i$ and each subsystem $\left(B_i, f\right)$ (i.e., $\left(B_i, f|_{\Omega(f)}\right)$) is topologically transitive. By Lemma \ref{lem:htop_sum} and Lemma \ref{lem:htop_nw}, $\htop(f,X) = \htop\left(f, \Omega(f)\right) = \htop\left(f, \Omega\left(f|_{\Omega(f)}\right) \right) = \max_{1 \leq i \leq l}\left\{\htop\left(f, B_i\right) \right\}$. Then $\htop(f,X) = \htop\left(f, B_i\right)$ for some $1 \leq i \leq l$. Fix such an $i$. Applying the degenerate case $\varphi\equiv0$ of Theorem \ref{thm:A} to $(B_i,f)$, for every $0\leq h<\htop(f,B_i)$ there exists a uniquely ergodic measure $\mu$ with compact support contained in $B_i$ such that $h_\mu(f)=h$. Let $Y=\supp(\mu)$. Then $(Y,f|_Y)$ is minimal and uniquely ergodic, and the variational principle gives $\htop(f,Y)=h_\mu(f)=h$. Since $\htop(f,B_i)=\htop(f,X)$, the same conclusion holds for $(X,f)$.
\end{proof}

\begin{proof}[Proof of corollary \ref{cor:2.3}]
    In case (H1) and (H3), the entropy function is upper semi-continuous and there exists a measure of maximal entropy. Combined with Theorem \ref{thm:2.2}, we prove this corollary.
\end{proof}

\begin{proof}[Proof of Theorem \ref{thm:2.4}]
    Fix $a \in \Int L_\varphi$. For any $\mu \in \Merg(f,X)$ satisfying $\int\vphi\dif\mu = a$, one has $G_\mu \subseteq R_\varphi(a)$ where $G_\mu$ is the set of generic points of $\mu$ defined in Section \ref{subsec:M(X)}. By \cite[Theorem 3]{bowen1973topological},
    \[\htop\left(f, R_\varphi(a)\right) \geq \htop\left(f, G_\mu\right) = h_\mu(f).\]
    Since $\mu$ is arbitrary, one has
    \[\htop\left(f, R_\varphi(a)\right) \geq \sup\left\{h_\mu(f):\mu \in \Merg(f,X),\int\varphi\dif\mu = a\right\}.\]
    On the other hand, let $\M_a = \left\{\mu \in \M(f,X):\int \varphi \dif \mu = a\right\}$. Then $R_\varphi(a) \subseteq G^{\M_a}$ where
    \[G^{\M_a}:= \left\{x \in X:\left\{\cE_n(x)\right\}_{n=1}^\infty \mathrm{\ has\ all\ its\ limit\ points\ in\ } \M_a \right\}.\]
    Note $G^{\M_a} \subseteq \prescript{\M_a}{}{G}$ where
    \[\prescript{\M_a}{}{G} := \left\{x \in X:\left\{\cE_n(x)\right\}_{n=1}^\infty \mathrm{\ has\ a\ limit\ point\ in\ } \M_a \right\}.\]
    Since $\M_a$ is closed, by \cite[Theorem 4.1 (1)]{pfister2007topological}, one has
    \[\htop\left(f, R_\varphi(a)\right) \leq \htop\left(f,\prescript{\M_a}{}{G}\right) \leq \sup\left\{h_\mu(f): \mu \in \M_a\right\}.\]
    By Theorem \ref{thm:A}, one has
    \begin{align*}
          & \sup\left\{h_\mu(f) :\mu \in \M(f,X), \int \varphi \dif \mu = a\right\} \\
        ={} & \sup\left\{h_\mu(f):\mu \in \M(f,X) \mathrm{\ is\ uniquely\ ergodic\ with\ compact\ support},\int \varphi \dif \mu = a\right\},
    \end{align*}
    which completes the proof.
\end{proof}

As mentioned in Remark \ref{rmk:thmA}, we present the following generalized theorem with a sketch of proof. Let $m \in \Nb^+$ and $Y$ be a nonempty convex subset of $\Rb^m$. The set of all affine combinations of points in $Y$ is defined as
\[\mathrm{Aff}(Y) = \left\{\sum_{i=1}^{k}\alpha_i y_i: k \in \Nb^+, y_i \in Y, \alpha_i \in \Rb \mathrm{\ for\ any\ } 1 \leq i \leq k, \sum_{i=1}^{k}\alpha_i = 1\right\},\]
called the \emph{affine hull} of $Y$. We define the \emph{affine dimension} of $Y$ as the dimension of its affine hull, denoted by $\dim_{\mathrm{a}}(Y)$. The \emph{relative interior} of $Y$ is defined as its interior within the affine hull of $Y$. In other words,
\[\mathrm{Relint}(Y) = \left\{y \in Y:\mathrm{there\ exists\ } \ve > 0 \mathrm{\ such\ that\ } B(y,\ve) \cap \mathrm{Aff}(Y) \subseteq Y\right\}.\]
The following properties of relative interiors can be easily checked.
\begin{enumerate}
    \item If $Y$ is a singleton ($\dima Y = 0$), then $\Aff(Y) = \Relint(Y) = Y$.
    \item If $0 \leq \dima(Y) < m$, then $\Int(Y) = \emptyset$.
    \item If $\dima(Y) = m$, then $\Int(Y) = \Relint(Y)$.
\end{enumerate}
Given $\left\{\vphi_j\right\}_{j=1}^m \subset C(X)$ and $m \in \Nb^+$, define
\[L_{\left\{\vphi_j\right\}_{j=1}^m} = \left\{\left(\int\vphi_1\dif\mu,\dots,\int\vphi_m\dif\mu\right):\mu \in \M(f,X)\right\}.\]
Since $\M(f,X)$ is compact and convex, we have $L_{\left\{\vphi_j\right\}_{j=1}^m}$ is a compact and convex subset of $\Rb^m$.

\begin{theorem}\label{thm:relint}
    Suppose $f$ is a topologically transitive and topologically expanding (topologically Anosov) continuous map (homeomorphism) on a compact metric space $(X,d)$. Let $m \in \Nb^+$ and $\left\{\vphi_j\right\}_{j=1}^m \subset C(X)$ such that $\Relint L_{\left\{\vphi_j\right\}_{j=1}^m} \ne \emptyset$. For any $x \in X$, any $\ve > 0$, any $\left(a_1,\dots,a_m\right) \in \Relint L_{\left\{\vphi_j\right\}_{j=1}^m}$, any $\mu \in \M(f,X)$ satisfying
    \[\left(\int\vphi_1\dif\mu, \dots, \int\vphi_m\dif\mu\right)   = \left(a_1,\dots,a_m\right)\]
    and
    \[h_\mu(f) < \sup\left\{h_\omega(f): \omega \in \M(f,X), \left(\int\vphi_1\dif\omega, \dots, \int\vphi_m\dif\omega\right)   = \left(a_1,\dots,a_m\right)\right\},\]
    and any $0 \leq h \leq h_\mu(f)$, there exists a compact $f$-invariant set $Y \subseteq X$ such that
    \begin{enumerate}
        \item $\left(Y,f\right) $ is minimal and uniquely ergodic;
        \item $h_{\nu}(f) = \htop(f,Y) = h$ where $\nu$ is the unique $f$-invariant measure on $Y$;
        \item \[\left(\int\vphi_1\dif\nu, \dots, \int\vphi_m\dif\nu\right)   = \left(a_1,\dots,a_m\right);\]
        \item $D(\nu,\mu) < \ve$.
    \end{enumerate}
\end{theorem}
\begin{proof}[Sketch of proof]
    Define the continuous affine map
    \[
        \mathbf{I}:\M(f,X)\longrightarrow\Rb^m,\qquad
        \mathbf{I}(\omega)
        =
        \left(
        \int\vphi_1\dif\omega,\dots,\int\vphi_m\dif\omega
        \right),
    \]
    and denote $\left(a_1,\dots,a_m\right)$ by
    $\overrightarrow{a}$. We follow the proof of Theorem
    \ref{thmA}, replacing the scalar map
    $\omega\mapsto\int\vphi\dif\omega$ by $\mathbf{I}$.

    If $
        \dima\left(
        L_{\left\{\vphi_j\right\}_{j=1}^m}
        \right)=0
    $ then
    $L_{\left\{\vphi_j\right\}_{j=1}^m}
        =\{\overrightarrow{a}\}$, and hence
    $\mathbf{I}(\omega)=\overrightarrow{a}$ for every
    $\omega\in\M(f,X)$. The result follows from the entropy part of the
    proof of Theorem \ref{thmA}, with the integral conditions omitted.

    Now suppose that
    $
        0<
        \dima\left(
        L_{\left\{\vphi_j\right\}_{j=1}^m}
        \right)
        =d\leq m.
    $
    After translating $\overrightarrow{a}$ to the origin and choosing
    affine coordinates, we identify
    $
        \Aff\left(
        L_{\left\{\vphi_j\right\}_{j=1}^m}
        \right)
    $
    with $\Rb^d$. Since
    $\overrightarrow{a}\in
        \Relint L_{\left\{\vphi_j\right\}_{j=1}^m}$,
    every sufficiently small neighborhood of the origin in $\Rb^d$ is
    contained in
    $L_{\left\{\vphi_j\right\}_{j=1}^m}
        -\overrightarrow{a}$.

    In Step (1) of the proof of Theorem \ref{thmA}, replace the two
    measures whose integrals lie on the two sides of $a$ by a family
    \[
        \left\{\mu_w^1:w\in\{l,r\}^d\right\}
        \subseteq\M(f,X)
    \]
    such that
    $\mathbf{I}(\mu_w^1)-\overrightarrow{a}$ belongs to the orthant
    indexed by $w$. Using convex combinations, these measures can be
    chosen sufficiently close to $\mu$, with their entropies in the
    same interval as in Step (1) of that proof. Moreover,
    \[
        \overrightarrow{a}
        \in
        \Relint\cov\left\{
        \mathbf{I}(\mu_w^1):w\in\{l,r\}^d
        \right\}.
    \]
    Here we use the elementary fact that if a finite set contains a
    point in every open orthant of $\Rb^d$, then the origin belongs to
    the interior of its convex hull.

    Apply Theorem \ref{thmC} to the family
    $\{\mu_w^1:w\in\{l,r\}^d\}$. In each subsequent step, replace the
    scalar affine map
    \[
        \nu\longmapsto\int\vphi\dif(\tau_k\nu)
    \]
    in the proof of Theorem \ref{thmA} by the vector-valued affine map
    $\mathbf{I}\circ\tau_k$. The measures can again be chosen in the
    $2^d$ orthants, arbitrarily close to
    $\overrightarrow{a}$, while preserving the entropy estimates.
    Choosing the Hausdorff errors sufficiently small preserves the
    condition that $\overrightarrow{a}$ lies in the relative interior
    of the convex hull of the corresponding integral vectors.

    The remainder of the construction is the same as in the proof of
    Theorem \ref{thmA}. We obtain a decreasing sequence of compact
    invariant sets $\{\Lambda^k\}_{k\geq1}$ such that
    \[
        \diam\M(f,\Lambda^k)\longrightarrow0,
        \qquad
        \htop(f,\Lambda^k)\longrightarrow h,
    \]
    and the corresponding integral vectors converge to
    $\overrightarrow{a}$. Hence
    $\bigcap_{k\geq1}\Lambda^k$ supports a unique invariant measure
    $\nu$ satisfying
    \[
        h_\nu(f)=h,\qquad
        \mathbf{I}(\nu)=\overrightarrow{a},\qquad
        D(\nu,\mu)<\ve.
    \]
    Let $Y=\supp(\nu)$. Then $\M(f,Y)=\{\nu\}$, so $(Y,f|_Y)$ is uniquely ergodic; by \cite[Theorem 6.17]{walters1982introduction}, it is also minimal. The variational principle gives $\htop(f,Y)=h_\nu(f)=h$. Thus all the required conclusions hold.
\end{proof}

\begin{proof}[Proof of Theorem \ref{thm:2.5}]
    By Theorem \ref{thm:relint},
    \[\Relint \left\{\left(\int \vphi_1 \dif \mu,\dots,\int \vphi_m \dif \mu\right) : \mu \mathrm{\ uniquely\ ergodic}, \mu \in \M(f,M)\right\} = \Relint L_{\left\{\vphi_j\right\}_{1}^m}\]
    for any $m \in \Nb^+$ and any $\left\{\vphi_j\right\}_{1}^m \subseteq C(M)$. For any $1 \leq i \leq t_u+t_s$, define $\psi_i(x) = \log \left\lvert \det Df|_{E_x^i}\right\rvert $. By the definition of multi-average conformality and the ergodic decomposition theorem, one has
    \[
        \int\psi_i\,\dif\mu
        =d_i\chi_{s_{i-1}+1}(\mu)
    \]
    for every $\mu\in\M(f,M)$. Denote $\phi_i = \psi_i / d_i$. Since $E_x^i$ is continuous with respect to $x$, one has that $\phi_i(x)$ is a continuous function. Therefore,
    \begin{align*}
          & \mathrm{Relint} \left\{\left(\chi_1(\mu),\dots,\chi_{\dim M}(\mu)\right):\mu \in \M(f,M)\right\}                       \\
        = & \mathrm{Relint} \left\{\left(\chi_1(\mu),\dots,\chi_{\dim M}(\mu)\right):\mu \in \M(f,M) \mathrm{\ is\ uniquely\ ergodic}\right\}.
    \end{align*}
    When $\dima\left(L_{\left\{\vphi_j\right\}_{j=1}^m}\right) = m$, we have $\Relint \left(L_{\left\{\vphi_j\right\}_{j=1}^m}\right) = \Int L_{\left\{\vphi_j\right\}_{j=1}^m}$. When $\dima\left(L_{\left\{\vphi_j\right\}_{j=1}^m}\right) < m$, we have $\Int L_{\left\{\vphi_j\right\}_{j=1}^m} = \emptyset$. Therefore,
    \begin{align*}
          & \Int \left\{\left(\chi_1(\mu),\dots,\chi_{\dim M}(\mu)\right):\mu \in \M(f,M)\right\}                       \\
        = & \Int \left\{\left(\chi_1(\mu),\dots,\chi_{\dim M}(\mu)\right):\mu \in \M(f,M) \mathrm{\ is\ uniquely\ ergodic}\right\}.
    \end{align*}
\end{proof}

\part{Countable Markov shifts}
In this part, we study the countable Markov shifts.

\section{Preliminary}\label{sec:pre_cms}
\subsection{Basic definitions}
We begin this section by reviewing the fundamental definitions, properties, and necessary lemmas associated with countable Markov shifts. For further detail and broader context, readers are referred to \cite{kitchens1998symbolic}.

Let $M$ be a $\Nb \times \Nb$ matrix with entries $0$ or $1$. The symbolic space associated to $M$ with alphabet $\Nb$ is defined by
\[\Sigma := \left\{(x_0x_1\ldots) \in \Nb^{\Nb}: M(x_i, x_{i+1}) = 1 \mathrm{\ for\ any\ i} \in \Nb\right\}.\]
We endow $\Nb$ with the discrete topology and ${\Nb}^{\Nb}$ with the product topology. We consider the induced topology on $\Sigma$ given by the natural inclusion $\Sigma \subseteq {\Nb}^{\Nb}$. Note in general, $\Sigma$ is non-compact. When $M$ is irreducible, $\Sigma = \Sigma(M)$ is locally compact if and only if for every $i \in \Nb$ we have $\sum_{j \in \Nb} M(i,j) < \infty$ (see \cite[Observation 7.2.3]{kitchens1998symbolic}).

The \emph{shift map} $\sigma: \Sigma \to \Sigma$ is defined by $\left(\sigma(x)\right)_i = x_{i+1}$ where $x = \left(x_0x_1\ldots\right) \in \Sigma$. Note $\sigma$ is a continuous map. We call $\left(\Sigma, \sigma\right)$ a one-sided \emph{countable Markov shift}. The matrix $M$ can be identified with a directed graph $G$ with no multiple edges (but allowing edges connecting a vertex to itself).

A \emph{word} of length $N$ is a string $w = w_0 w_1 \ldots w_{N-1} $ of letters in the alphabet. We call $w$ an \emph{admissible word} if $M\left(w_i, w_{i+1}\right) = 1$ for any $i \in \left\{0, \ldots, N-2\right\}$. In particular, we say that $w$ \emph{connects} $w_0$ and $w_{N-1}$. We also use $l(w)$ to denote the length of a word $w$. In this part, parenthesized sequences denote points in the space, whereas unparenthesized sequences represent words.

Given a word $w = w_0 w_1 \ldots w_{N-1}$, a \emph{cylinder} of length $N$ is the set
\[\left[w_0 \ldots w_{N-1}\right] : = \left\{x = \left(x_0x_1\ldots\right) \in \Sigma: x_i = w_i \mathrm{\ for\ } 0 \leq i \leq N-1 \right\}.\]
We use the notation $[x]_N$ to denote the cylinder of length $N$ containing $x$. In general, we use bold big letters to denote cylinders. We also use $l(\mathbf{C})$ to denote the length of a cylinder $\mathbf{C}$. Note the topology generated by the cylinders coincides with that induced by the product topology.

Denote by $\cL_n$ the set of all admissible words with length $n$ and denote $\cL = \bigcup_{n=1}^{\infty} \cL_n$ the set of all admissible words. For convenience, put $\cL_0 = \left\{e\right\}$ and $ve = v = ev$ for any $v \in \cL$. Given $\cK \subseteq \cL$, denote $\left[\cK\right] = \bigcup_{w \in \cK} \left[w\right]$. Given $\Lambda \subseteq \Sigma$ and $n \in \Nb$, denote $[\Lambda]_n = \bigcup_{x \in \Lambda}[x]_n$.

The space $\Sigma$ is metrisable. Let $d: \Sigma \times \Sigma \to \Rb$ be the function defined by
\begin{align*}
    d(x,y):=
    \left\{
    \begin{array}{ll}
        1      & \mathrm{if\ } x_0 \ne y_0;                                                                              \\
        2^{-k} & \mathrm{if\ } x_i = y_i \mathrm{\ for\ } i \in \left\{0,\ldots,k-1\right\} \mathrm{\ and\ }x_k \ne y_k; \\
        0      & \mathrm{if\ } x = y.                                                                                    \\
    \end{array}
    \right.
\end{align*}
The function $d$ is a metric, and it generates the same topology as that of the cylinders. With this metric, $\Sigma$ is bounded.

Next, we recall some basic definitions and properties of the space of probability measures $\M(\Sigma)$ on countable Markov shifts. The weak* topology in the space of invariant measures of countable Markov shifts is similar to the case of compact metric spaces. A sequence of measures $\left\{\mu_n\right\}_n^\infty$ in $\M(\Sigma)$ converges to a measure $\mu$ in the weak* topology if for any $\vphi \in C_b(\Sigma)$, one has $\lim_{n \to \infty} \int \vphi \dif \mu_n = \int \vphi \dif \mu$. Note that in this notion of convergence, we can replace the space of bounded continuous functions $C_b(\Sigma)$ by the space of bounded uniformly continuous functions $C_{b,u}(\Sigma)$ (see \cite[Remark 8.3.1]{bogachev2007measure}). A basis for this topology is given by the collection of sets of the form
\[V\left(\varphi_1,\dots,\varphi_k,\mu,\ve\right):=\left\{\nu \in \M(\Sigma):\left\lvert \int \varphi_i \dif \mu - \int \varphi_i \dif \nu\right\rvert < \ve, \mathrm{\ for\ } i \in \left\{1,\dots,k\right\} \right\}.\]
In general, the space of $\sigma$-invariant probability measures $\M(\sigma, \Sigma)$ is closed but not compact with respect to the weak* topology. And the space $\M(\sigma, \Sigma)$ is convex whose extreme points are ergodic measures, see \cite[Theorem 6.10]{walters1982introduction}.

Since $\Sigma$ is a Polish space, we can equip it with the first Wasserstein metric (also called the Kantorovich-Rubinstein metric) $D$, see \cite[Definition 6.1]{villani2009optimal}. It metrizes the weak* topology (see \cite[Corollary 6.13]{villani2009optimal}). Note the first Wasserstein metric is convex and $D\left(\delta_x, \delta_y\right) = d(x,y)$ for any $x,y\in\Sigma$. Lemma \ref{lemMeasuresDistance} still holds for countable Markov shifts.

The definition of the one-sided Markov shift provided above can be naturally extended to the two-sided case. Unless otherwise specified, we shall primarily focus on one-sided Markov shifts. The results of this paper remain analogous for the two-sided setting.

\subsection{Dynamical properties}
The definition of transitivity, mixing, expansivity and the shadowing property of countable Markov shifts are as same as the case of continuous maps on compact space. Here, we list some basic properties. For transitivity and mixing, we have the following equivalent definitions.

\begin{definition}
    We say that $(\Sigma, \sigma)$ is \emph{topologically transitive} if for each pair $a,b \in \Nb$, there exists an admissible word connecting $a$ and $b$. We say that $(\Sigma, \sigma)$ is \emph{topologically mixing} if for each pair $a,b \in \Nb$, there exists $N \in \Nb$ such that for any $n > N$, there exists an admissible word connecting $a$ and $b$ with length $n$.
\end{definition}

Moreover, we have the following properties for countable Markov shifts.

\begin{lemma}\label{lem:propertiesCMS}
    Let $(\Sigma,\sigma)$ be a one-sided (two-sided) countable Markov shift.
    \begin{enumerate}
        \item $(\Sigma,\sigma)$ is positively expansive (expansive).
        \item $(\Sigma,\sigma)$ satisfies the shadowing property.
    \end{enumerate}
\end{lemma}

The first item of the above lemma can be easily checked by the metric on $\Sigma$. The second item can be checked in a similar way of \cite[Theorem 1]{walters1977pseudo}.

Ergodic decomposition theorem can also be applied to countable Markov shifts.

\begin{lemma}[{{see \cite[Lemma 2.5]{takahasi2020entropy}}}]\label{lem:ErgodicDecompositionCMS}
    For any $\mu \in \M(\sigma, \Sigma)$, there exists a Borel probability measure $\tau$ on $\M(\sigma,\Sigma)$ such that $\tau\left(\Merg(\sigma,\Sigma)\right) = 1$ and for any $\vphi \in C_b(\Sigma)$,
    \[\int\vphi\dif\mu = \int_{\Merg(\sigma,\Sigma)}\left(\int \vphi \dif \nu\right) \dif \tau(\nu).\]
\end{lemma}

\subsection{Entropy}
In this subsection, we recall the definitions of metric entropy and topological entropy in countable Markov shifts.
\subsubsection{Metric entropy}
\begin{definition}
    Let $\mu \in \M(\Sigma)$ and $\xi$ be a countable (finite or infinite) measurable partition of  $\Sigma$. The entropy of $\xi$ with respect to $\mu$ is defined by
    \[H_\mu(\xi) = - \sum_{A \in \xi}\mu(A)\log\mu(A).\]
    If further $\mu \in \M(\sigma, \Sigma)$, the following limit exists:
    $$h_\mu(\sigma,\xi) = \lim_{n\to\infty}\frac{1}{n}H_\mu\left(\bigvee_{i=0}^{n-1}\sigma^{-i}\xi\right).$$
    The \emph{metric entropy} of $\mu$ is defined by
    $$h_\mu(\sigma) = \sup\left\{h_\mu(\sigma,\xi):\xi\mathrm{\ is\ a\ countable\ measurable\ partition\ of\ }\Sigma \mathrm{\ with\ } H_\mu(\xi) < \infty\right\}.$$
\end{definition}

Here, partitions can be infinite and hence it is possible that $H_\mu(\xi) = \infty$ for some $\xi$. Therefore, the definition of metric entropy in countable Markov shifts is a little different to the case of maps on compact metric space.

By the affinity of entropy function and the ergodic decomposition theorem, we have the following lemma.

\begin{lemma}[{{\cite[Lemma 2.6]{takahasi2020entropy}}}]\label{lem:EntropyDecompositionCMS}
    Let $\mu \in \M(\sigma,\Sigma)$. Let $\tau$ be a Borel probability measure on $\M(\sigma, \Sigma)$ as in Lemma \ref{lem:ErgodicDecompositionCMS}. Then
    \[h_\mu(\sigma) = \int_{\Merg(\sigma,\Sigma)} h_{\nu}(\sigma)\dif \tau(\nu).\]
\end{lemma}

Takahasi \cite{takahasi2020entropy} proved that if $(\Sigma,\sigma)$ is topologically transitive, then any $\sigma$-invariant probability measure $\mu$ is entropy-approachable by ergodic measures supported on compact sets.

\begin{lemma}[{{\cite[Main Theorem]{takahasi2020entropy}}}]\label{lem:EntropyApproachableCMS}
    If $(\Sigma,\sigma)$ is topologically transitive, then for any $\mu \in \M(\sigma,\Sigma)$ there exists a sequence of ergodic measures $\left\{\mu_n:n \in \Nb^+\right\}$ supported on compact sets such that $\lim_{n \to \infty} \mu_n = \mu$ and $\lim_{n \to \infty} h_{\mu_n}(\sigma) = h_{\mu}(\sigma)$.
\end{lemma}

\subsubsection{Topological entropy}
In this section, we recall the definitions of topological entropy. In Section \ref{subsec:htop}, we recall the topological entropy for non-compact sets introduced by Bowen \cite{bowen1973topological}. However, for non-compact spaces the Bowen entropy depends on the metric (see \cite{handel1995metrics}). For countable Markov shifts, we use the topological entropy defined by Gurevich \cite{gurevich1969topological}.

\begin{definition}
    Suppose $(\Sigma, \sigma)$ is topologically transitive. Given $a \in \Nb$, the (Gurevich) topological entropy of $(\Sigma, \sigma)$ is defined by
    \[h_{G}(\sigma, \Sigma) = \limsup_{n \to \infty} \frac{1}{n}\log\left\lvert \left\{x \in \Sigma: \sigma^n x = x,\ x \in [a]\right\} \right\rvert.\]
\end{definition}

The limsup is independent of the choice of $a$. When $(\Sigma, \sigma)$ is topologically mixing, then the upper limit can be replaced by limit (see \cite[Theorem 1]{sarig1999thermodynamic}). It was proved by Gurevich \cite{gurevich1969topological} that the Gurevich topological entropy satisfies the variational principle.
\[h_G(\sigma, \Sigma) = \sup\left\{h_{\mu}(\sigma):\mu \in \M(\sigma,\Sigma)\right\}.\]

\begin{remark}
    We let $\htop(\sigma, \Sigma)$ represent the Gurevich entropy $h_G(\sigma, \Sigma)$. In addition, throughout the proof, we only consider the topological entropy $\htop(\sigma, B)$ of compact invariant sets, in which case it represents the standard topological entropy and coincides with the Bowen topological entropy.
\end{remark}

\section{Conditional entropy realization for countable Markov shifts}\label{sec:minimal_cms}
In this section, we establish a multi-horseshoe entropy-dense result for countable Markov shifts and use it to obtain exact entropy-and-average realization and weak* approximation by uniquely ergodic measures with compact support. In \cite{takahasi2020entropy}, Takahasi proved that for topologically transitive one-sided (two-sided) countable Markov shifts, any shift-invariant Borel probability measure is entropy-approachable by ergodic measures supported on compact sets. The following theorem generalizes the above result to the case of multiple measures.

\begin{theoremx}\label{thmD}
    Suppose $(\Sigma,\sigma)$ is a non-trivial topologically transitive one-sided (two-sided) countable Markov shift. Then for any $m, k \in \Nb^+$, any $\left\{\mu_i\right\}_{i=1}^m \subset \M(\sigma, \Sigma)$ with finite entropy, any $x \in \Sigma$, any $\left\{\vphi_j\right\}_{j=1}^k \subset C_{b,u}(\Sigma)$ and any $\ve > 0$, there exists $n \in \Nb^+$ and compact $\sigma^n$-invariant subsets $\Delta_i \subseteq \Delta \subseteq \Sigma$ such that for each $1 \leq i \leq m$, the following statements hold.
    \begin{enumerate}
        \item $\left(\Delta_i, \sigma^n|_{\Delta_i}\right)$ and $\left(\Delta, \sigma^n|_{\Delta}\right)$ are topologically conjugate to a one-sided (two-sided) full shift respectively.
        \item $\htop\left(\sigma,\Lambda_i\right) \geq \htop \left(\sigma^n,\Delta_i\right)  / n > h_{\mu_i}(\sigma) - \ve$ where $\Lambda_i = \bigcup_{j=0}^{n-1}\sigma^j \left(\Delta_i\right) $.
        \item For any $\nu_i \in \M\left(\sigma,\Lambda_i\right)$ and any $1 \leq j \leq k$, one has
              \[\left\lvert \int\vphi_j\dif\mu_i - \int\vphi_j\dif\nu_i\right\rvert  < \ve.\]
              For any $\nu \in \M\left(\sigma,\Lambda\right)$ where $\Lambda = \bigcup_{j=0}^{n-1}\sigma^j \left(\Delta\right) $, there exists $\mu \in \cov\left(\left\{\mu_i\right\}_{i=1}^m\right)$ such that
              \[\left\lvert \int\vphi_j\dif\mu - \int\vphi_j\dif\nu\right\rvert  < \ve\]
              for any $1 \leq j \leq k$.
        \item For any $z \in \Lambda_i$ or $\Lambda$ one has $\sigma^{j+tn}z \in B(x,\ve)$ for some $1 \leq j \leq n$ for any $t \in \Nb$.
    \end{enumerate}
\end{theoremx}
\begin{remark}
    It follows from item (3) that for any $\mu \in \cov\left(\left\{\mu_i\right\}_{i=1}^m\right)$, there exists $\nu \in \M\left(\sigma, \bigcup_{i=1}^m \Lambda_i\right) \subseteq \M(\sigma,\Lambda)$ such that
    \[\left\lvert \int\vphi_j\dif\mu - \int\vphi_j\dif\nu\right\rvert  < \ve\]
    for any $1 \leq j \leq k$, which leads to a ``Hausdorff'' result similar to Theorem \ref{thmC}.
\end{remark}

We follow a proof similar to that for topologically expanding maps on compact metric space. For any function $\psi$ on $\Sigma$, define
\[D_n(\psi) = \sup_{x \in \Sigma, y \in \left[x\right]_n} \left\lvert \sum_{i=0}^{n-1} \psi\left(\sigma^i x\right) - \sum_{i=0}^{n-1} \psi\left(\sigma^i y\right) \right\rvert.\]
Define $\cL(a,b)$ to be the set of admissible words with $a$ and $b$ as the first and last letter, respectively. $\cL_n(a,b)$ is defined similarly. The following lemma shows that any ergodic measure can be approximated by finite collections of cylinders or separated sets.

\begin{lemma}[{{\cite[Lemma 2.3]{takahasi2020entropy}}}]\label{lem:takahasi_lem2.3}
    Let $m \geq 1$ be an integer and $\psi_1, \dots, \psi_m:\Sigma \to \Rb$ be measurable functions satisfying $D_n\left(\psi_j\right) = o(n) $ as $n \to \infty$ for any $1 \leq j \leq m$. Let $\mu \in \Merg(\sigma, \Sigma)$ satisfy $h_\mu(\sigma) < \infty$, $\psi_j \in L^1(\mu)$ for any $1 \leq j \leq m$, and let $M \geq 1$ be an integer with $\sum_{k=0}^{M-1} \mu\left([k]\right) > 0$. For any $ \ve > 0$, there exists $N \geq 1$ such that for any $n \geq N$, there exists an integer $l \geq n$, $(a,b) \in \left\{0,\dots,M-1\right\}^2 $, and a finite set $F^l \subseteq \cL_l(a,b)$ such that for any $1 \leq j \leq m$, the following holds:
    \[\left\lvert \frac{1}{l}\log \left\lvert F^l\right\rvert - h_\mu(\sigma)\right\rvert \leq \ve,\]
    \[\sup_{x \in \left[F^l\right]}\left\lvert \frac{1}{l}\sum_{i=0}^{l-1} \psi_j\left(\sigma^i x\right)  - \int\psi_j\dif\mu\right\rvert \leq \ve.\]
\end{lemma}

Note every bounded uniformly continuous function $\vphi \in C_{b,u}(\Sigma)$ satisfies $D_n\left(\vphi\right) = o(n) $ as $n \to \infty$.

\begin{lemma}\label{lem:SeparatedSetCMS}
    Suppose $(\Sigma,\sigma)$ is a topologically transitive one-sided
    (respectively, two-sided) countable Markov shift. Then for any
    $\mu\in\M(\sigma,\Sigma)$ satisfying
    $0\leq h_\mu(\sigma)<\infty$, any $k\in\Nb^+$, any
    $\{\varphi_j\}_{j=1}^k\subset C_{b,u}(\Sigma)$, any $x\in\Sigma$,
    any $\ve>0$, and any $N\in\Nb^+$, there exist $n\geq N$ and
    $M\in\Nb$ such that, for every $p\in\Nb^+$, there exists a
    $(pn,1/3)$-separated set $\Gamma_{pn}$ satisfying
    \begin{enumerate}
        \item
              $\Gamma_{pn}\subseteq B(x,\ve)\cap\Per_{pn}(\sigma)
                  \cap\{0,\ldots,M\}^{\Nb}$ in the one-sided case, and
              $\Gamma_{pn}\subseteq B(x,\ve)\cap\Per_{pn}(\sigma)
                  \cap\{0,\ldots,M\}^{\Zb}$ in the two-sided case;

        \item
              \[
                  \frac{1}{pn}\log|\Gamma_{pn}|
                  >h_\mu(\sigma)-\ve;
              \]

        \item for every $1\leq j\leq k$ and every $z\in\Gamma_{pn}$,
              \[
                  \left|
                  \int\varphi_j\dif\mu
                  -\frac{1}{pn}\sum_{t=0}^{pn-1}
                  \varphi_j(\sigma^t z)
                  \right|<\ve.
              \]
    \end{enumerate}
\end{lemma}
\begin{proof}
    Without loss of generality, we assume $\left\lvert \varphi_j\right\rvert \leq 1$ for any $1 \leq j \leq k$. Fix $\mu \in \M(\sigma, \Sigma)$ with $0 \leq h_{\mu}(\sigma) < \infty$, $\left\{\varphi_j\right\}_{j=1}^k \subset C_{b,u}(\Sigma)$, $x \in \Sigma$, $\ve > 0$ and $N \in \Nb^+$. Assume $0 < \ve < 1$ and $\ve < h_{\mu}(\sigma)$ if $h_{\mu}(\sigma) > 0$. By Lemma \ref{lem:ErgodicDecompositionCMS} and Lemma \ref{lem:EntropyDecompositionCMS}, there exists $r \in \Nb^+$ and $\left\{\mu_i\right\}_{i=1}^r \subset \Merg\left(\sigma, \Sigma\right)$ such that
    \[ \left\lvert h_{\mu'} -  h_\mu\right\rvert  < \ve/2,\]
    \[\left\lvert \int\varphi_j\dif\mu-\int\varphi_j\dif\mu'\right\rvert  <  \ve / 2,\]
    for any $1 \leq j \leq k$ where
    $\mu' = \frac{1}{r} \sum_{i=1}^r \mu_i$.

    Choose a cylinder $\mathbf U$ with $x \in \mathbf U \subseteq B(x,\ve)$. In the one-sided case, we write $\mathbf U = [\mathbf{u}]$, where $\mathbf{u}=u_0\cdots u_{s-1}$ is an admissible word. In the two-sided case we write \( \mathbf U=[x_{-s}\cdots x_s]_{[-s,s]} \) and put $\mathbf{u}=x_{-s}\cdots x_s$. In both cases, denote the first and last letters of $\mathbf{u}$ by $u_{0}$ and $u_{-1}$, respectively.

    Fix $q \in \Nb$ such that $q > \max\left\{u_j: 0 \leq j \leq l\left(\mathbf{u}\right) -1\right\} $ and
    \[
        \sum_{j=0}^{q-1}\mu_i\left([j]\right) > 0
    \]
    for any $1 \leq i \leq r$. Note $\left(\Sigma,\sigma\right)$ is topologically transitive. For any $a, b \in \left\{0,1,\dots,q-1\right\}$, fix $\mathbf{v_{ab}}$ such that $a\mathbf{v_{ab}}b$ is admissible. Take $l^* \in \Nb$ such that
    \[
        l^* \geq \max\left\{ l\left(\mathbf{v_{ab}}\right):a,b \in \left\{0,1,\dots,q-1\right\}\right\}, \qquad l^* \geq l(\mathbf{u}).
    \]

    By Lemma \ref{lem:takahasi_lem2.3}, for any $1 \leq i \leq r$, there exists $m^* \in \Nb$ such that for any $m \geq m^*$, there exists $l_i \geq m$, $(a_i,b_i) \in \left\{0,\dots,q-1\right\}^2 $ and a finite set $F_i \subseteq \cL_{l_i}(a_i,b_i)$ such that
    \begin{equation}\label{eq:7.1}
        \frac{1}{l_i}\log \left\lvert F_i\right\rvert \geq h_{\mu_i}(\sigma) - \ve / 4
    \end{equation}
    and
    \begin{equation}\label{eq:7.2}
        \left\lvert \frac{1}{l_i}\sum_{t=0}^{l_i-1}\varphi_j\left(\sigma^t y\right)  -\int\varphi_j\dif\mu_i\right\rvert   \leq \ve / 4
    \end{equation}
    for any $1 \leq j \leq k$ and any $y \in \left[F_i\right]$. Take $m \in \Nb^+$ such that
    \begin{equation}\label{eq:7.3}
        m \geq \max\left\{m^*, N, 32l^* / \ve, 16l^*h_{\mu'}(\sigma)/\ve - 8l^*\right\}.
    \end{equation}
    Take $M \in \Nb$ large enough such that $M \geq q$ and every letter of words in
    \[
        \left\{\mathbf{u}\right\} \cup\left\{\mathbf{v_{ab}}: 0 \leq a,b < q\right\} \cup\bigcup_{i=1}^r F_i
    \]
    is smaller than $M$.

    Let $L = \prod_{i=1}^r l_i$ and $L_i = L / l_i \leq L/m$ for any $1 \leq i \leq r$. Let
    \[n_{\mathrm{b}} = l(\mathbf{u})+ l(\mathbf{v_{u_{-1}a_1}}) + l(\mathbf{v_{b_ru_0}}) + \sum_{i=1}^{r-1}l\left(\mathbf{v_{b_ia_{i+1}}}\right) + \sum_{i=1}^{r}\left(L_i-1\right)l(\mathbf{v_{b_ia_i}})\]
    and $n = n_{\mathrm{b}} + rL$. Then
    \begin{equation}\label{eq:7.4}
        n_{\mathrm{b}} \leq \left(r + 2 + \sum_{i=1}^{r}L_i\right)l^* \leq \left(r + 2 + \frac{rL}{m}\right)l^*.
    \end{equation}
    Therefore,
    \begin{equation}\label{eq:7.5}
        \begin{aligned}
                                          & \phantom{{}\leq{}} \frac{n_{\mathrm{b}}}{n}                                                                              \\
            (\text{by (\ref{eq:7.4})})    & \leq \frac{\left(r+2+\frac{rL}{m}\right)l^*}{\left(r+2+\frac{rL}{m}\right)l^* + rL} \quad                                \\
                                          & = \frac{\left(1+\frac{2}{r}\right)l^* + \frac{l^*}{m}L }{\left(1+\frac{2}{r}\right)l^* + \left(1+\frac{l^*}{m}\right) L} \\
                                          & \leq \frac{4l^*}{4l^*+ m} \quad (\mathrm{by\ } m \leq L \mathrm{\ and\ } r \geq 1)                                       \\
            (\mathrm{by\ (\ref{eq:7.3})}) & \leq \frac{1}{1+8/\ve}.
        \end{aligned}
    \end{equation}
    If $h_\mu(\sigma) > 0$, we also have
    \begin{equation}\label{eq:7.51}
        \frac{n_{\mathrm{b}}}{n} \leq \frac{\ve / 4}{h_{\mu'}(\sigma)- \ve/4}
    \end{equation}
    by (\ref{eq:7.3}) and the above inequality.
    Here, the justification for the inequality in each line is provided at the end of the line. Fix $p \in \Nb^+$. Let $\Gamma_{pn}$ be the set of all points generated by words of the following form (when $r > 1$, similar for $r = 1$) for the one-sided case.
    \begin{align*}
        \mathbf{uv_{u_{-1}a_1}} & \mathbf{w_{1,1,1}v_{b_1a_1}w_{1,1,2}v_{b_1a_1}}\dots\mathbf{v_{b_1a_1}w_{1,1,L_1}v_{b_1a_2}}     \\
                                & \mathbf{w_{1,2,1}v_{b_2a_2}w_{1,2,2}v_{b_2a_2}}\dots\mathbf{v_{b_2a_2}w_{1,2,L_2}v_{b_2a_3}}     \\
                                & \dots                                                                                            \\
                                & \mathbf{w_{1,r,1}v_{b_r a_r}w_{1,r,2}v_{b_r a_r}}\dots\mathbf{v_{b_r a_r}w_{1,r,L_r}v_{b_r u_0}} \\
        \mathbf{uv_{u_{-1}a_1}} & \mathbf{w_{2,1,1}v_{b_1a_1}w_{2,1,2}v_{b_1a_1}}\dots\mathbf{v_{b_1a_1}w_{2,1,L_1}v_{b_1a_2}}     \\
                                & \mathbf{w_{2,2,1}v_{b_2a_2}w_{2,2,2}v_{b_2a_2}}\dots\mathbf{v_{b_2a_2}w_{2,2,L_2}v_{b_2a_3}}     \\
                                & \dots                                                                                            \\
                                & \mathbf{w_{2,r,1}v_{b_r a_r}w_{2,r,2}v_{b_r a_r}}\dots\mathbf{v_{b_r a_r}w_{2,r,L_r}v_{b_r u_0}} \\
                                & \dots                                                                                            \\
        \mathbf{uv_{u_{-1}a_1}} & \mathbf{w_{p,1,1}v_{b_1a_1}w_{p,1,2}v_{b_1a_1}}\dots\mathbf{v_{b_1a_1}w_{p,1,L_1}v_{b_1a_2}}     \\
                                & \mathbf{w_{p,2,1}v_{b_2a_2}w_{p,2,2}v_{b_2a_2}}\dots\mathbf{v_{b_2a_2}w_{p,2,L_2}v_{b_2a_3}}     \\
                                & \dots                                                                                            \\
                                & \mathbf{w_{p,r,1}v_{b_r a_r}w_{p,r,2}v_{b_r a_r}}\dots\mathbf{v_{b_r a_r}w_{p,r,L_r}v_{b_r u_0}} \\
    \end{align*}
    where $\mathbf{w_{t,i,j}} \in F_i$ for any $1 \leq t \leq p$, any $1 \leq i \leq r$ and any $1 \leq j \leq L_i$. In the two-sided case, first construct the corresponding bi-infinite periodic points and then apply $\sigma^s$ to them. Thus, $\Gamma_{pn}$ is $(pn,1/3)$-separated in both cases.

    It can be easily checked that
    \[
        \Gamma_{pn} \subseteq B(x,\ve) \cap \Per_{pn}(\sigma) \cap \left\{0,\dots,M\right\}^{\Nb}
    \]
    for one-sided case and
    \[
        \Gamma_{pn} \subseteq B(x,\ve) \cap \Per_{pn}(\sigma) \cap \left\{0,\dots,M\right\}^{\Zb}
    \]
    for two-sided case. It remains to prove that
    \[
        \left(\log\left\lvert \Gamma_{pn}\right\rvert \right) / \left(pn\right)  > h_\mu - \ve
    \]
    and
    \[
        \left\lvert \int\varphi_j\dif\mu - \frac{1}{pn}\sum_{i=0}^{pn-1}\varphi_j\left(\sigma^iz\right) \right\rvert < \ve
    \]
    for any $1 \leq j \leq k$ and any $z \in \Gamma_{pn}$.

    For any $1 \leq j \leq k$ and any $z \in \Gamma_{pn}$,
    \begin{align*}
                                                                 & \quad \left\lvert \int\varphi_j\dif\mu' - \frac{1}{pn}\sum_{i=0}^{pn-1}\varphi_j\left(\sigma^i z\right) \right\rvert                                                                                                             \\
        (\mathrm{similar\ to\ Lemma\ \ref{lemMeasuresDistance}}) & \leq \frac{2n_{\mathrm{b}}}{n} + \frac{rL}{n}\max_{1 \leq i \leq r}\max\left\{\left\lvert \int\varphi_j\dif\mu_i - \frac{1}{l_i}\sum_{t=0}^{l_i-1}\varphi_j\left(\sigma^t y\right) \right\rvert: y \in \left[F_i\right] \right\} \\
        (\mathrm{by\ (\ref{eq:7.5})})                            & \leq \frac{2}{1+8/\ve} + \ve/4                                                                                                                                                                                                   \\
                                                                 & \leq  \ve / 2.
    \end{align*}
    Therefore, $\left\lvert \int\varphi_j\dif\mu - \frac{1}{pn}\sum_{i=0}^{pn-1}\varphi_j\left(\sigma^i z\right) \right\rvert < \ve$ for any $1 \leq j \leq k$ and any $z \in \Gamma_{pn}$.

    On the other hand, if $h_\mu(\sigma) = 0$, note $F_i \ne \emptyset$ for any $1 \leq i \leq r$ and hence $\Gamma_{pn} \ne \emptyset$. If $h_\mu(\sigma) > 0$, then
    \begin{align*}
                                                                    & \phantom{{}\leq{}} \frac{1}{pn}\log \left\lvert\Gamma_{pn}\right\rvert               \\
        (\mathrm{by\ the\ construction\ of\ } \Gamma_{pn})          & \geq \frac{1}{pn}\log\left(\prod_{i=1}^r \left\lvert F_i\right\rvert^{L_i} \right)^p \\
                                                                    & = \frac{1}{n}\sum_{i=1}^r L_i\log\left\lvert F_i\right\rvert                         \\
        (\mathrm{by\ }L_i = L / l_i \mathrm{\ and\ (\ref{eq:7.1})}) & \geq \frac{L}{n}\sum_{i=1}^{r}\left(h_{\mu_i}(\sigma)-\ve/4\right)                   \\
                                                                    & = \frac{rL}{n}\left(h_{\mu'}(\sigma)-\ve/4\right)                                    \\
                                                                    & \geq \left(1-\frac{n_{\mathrm{b}}}{n}\right) \left(h_{\mu'}(\sigma)-\ve/4\right)     \\
        (\mathrm{by\ }(\ref{eq:7.51}))                              & \geq h_{\mu'}(\sigma)-\ve/2 > h_{\mu}(\sigma) - \ve,
    \end{align*}
    which completes the proof.
\end{proof}

\begin{proof}[Proof of Theorem \ref{thmD}]
    Fix $m \in \Nb^+$, $k \in \Nb$, $\left\{\mu_i\right\}_{i=1}^m \subseteq \M(\sigma,\Sigma)$ with finite entropy, $x \in \Sigma$, $\left\{\vphi_j\right\}_{j=1}^k \subseteq C_{b,u}(\Sigma)$ and $\ve>0$. For any $z\in\Sigma$ and $n\in\Nb^+$, write
    \[S_n\vphi_j(z)=\frac{1}{n}\sum_{t=0}^{n-1}\vphi_j(\sigma^t z).\]

    Choose $q\in\Nb^+$ sufficiently large such that the cylinder \(\mathbf U=[x_0\ldots x_q]\) in the one-sided case, and the cylinder $\mathbf U=[x_{-q}\ldots x_q]$ in the two-sided case, is contained in $B(x,\ve)$. Fix $0<\eta<\ve/3$ sufficiently small such that $B(x,\eta)\subseteq \mathbf U$. Since $D_n(\vphi_j)=o(n)$ for any $1\leq j\leq k$, there exists $N>2q$ such that
    \begin{equation}\label{eq:thmD_variation}
        \frac{D_l(\vphi_j)}{l}<\eta
    \end{equation}
    for any $l\geq N$ and any $1\leq j\leq k$.

    Applying Lemma \ref{lem:SeparatedSetCMS} to $\mu_i$, $x$, $\eta$ and $N$, there exist $n_i\geq N$ and $M_i\in\Nb$ such that for any $p\in\Nb^+$, there exists a $(pn_i,1/3)$-separated set $\Gamma_{pn_i}^i$ satisfying
    \begin{enumerate}
        \item $\Gamma_{pn_i}^i\subseteq B(x,\eta)\cap\Per_{pn_i}(\sigma)\cap\left\{0,\ldots,M_i\right\}^{\Nb}$ in the one-sided case, with the analogous inclusion in $\left\{0,\ldots,M_i\right\}^{\Zb}$ in the two-sided case;
        \item
              \begin{equation}\label{eq:thmD_cardinality}
                  \frac{1}{pn_i}\log\left\lvert\Gamma_{pn_i}^i\right\rvert>h_{\mu_i}(\sigma)-\eta;
              \end{equation}
        \item for any $1\leq j\leq k$ and any $y\in\Gamma_{pn_i}^i$,
              \begin{equation}\label{eq:thmD_average}
                  \left\lvert S_{pn_i}\vphi_j(y)-\int\vphi_j\dif\mu_i\right\rvert<\eta.
              \end{equation}
    \end{enumerate}
    Let
    \[n=\prod_{i=1}^m n_i,\qquad p_i=\frac{n}{n_i},\qquad \Gamma_n^i=\Gamma_{p_i n_i}^i\]
    and $M=\max\left\{M_i:1\leq i\leq m\right\}$. Thus every $\Gamma_n^i$ consists of $n$-periodic points using only the symbols in $\left\{0,\ldots,M\right\}$, and every point in $\Gamma_n^i$ belongs to the common cylinder $\mathbf U$.

    For any $y\in\bigcup_{i=1}^m\Gamma_n^i$, denote by
    \[\mathbf w(y)=y_0y_1\ldots y_{n-1}\]
    its period word, and put
    \[\mathcal W_i=\left\{\mathbf w(y):y\in\Gamma_n^i\right\},\qquad \mathcal W=\bigcup_{i=1}^m\mathcal W_i.\]
    Since every $y\in\bigcup_{i=1}^m\Gamma_n^i$ is $n$-periodic, its period word $\mathbf w(y)$ determines $y$. Hence,
    \begin{equation}\label{eq:thmD_words}
        \left\lvert\mathcal W_i\right\rvert=\left\lvert\Gamma_n^i\right\rvert.
    \end{equation}
    Moreover, all words in $\mathcal W$ have the same first symbol $x_0$. For any $\mathbf w(y)\in\mathcal W$, the transition from its last symbol $y_{n-1}$ to its first symbol $y_0=x_0$ is admissible since $y$ is $n$-periodic. It follows that words in $\mathcal W$ can be concatenated freely.

    Let $\mathbb I=\Nb$ in the one-sided case and $\mathbb I=\Zb$ in the two-sided case. For $\overline{\mathbf w}=(\mathbf w_t)_{t\in\mathbb I}\in\mathcal W^{\mathbb I}$, denote by $\Psi(\overline{\mathbf w})$ the point obtained by concatenating the words $\mathbf w_t$ in their natural order, aligned so that the word $\mathbf w_0$ starts at coordinate zero. Define
    \[\Delta=\Psi\left(\mathcal W^{\mathbb I}\right),\qquad \Delta_i=\Psi\left(\mathcal W_i^{\mathbb I}\right).\]
    Clearly, $\Delta_i\subseteq\Delta$ and $\sigma^n(\Delta_i)=\Delta_i$, $\sigma^n(\Delta)=\Delta$. Since $\mathcal W$ is finite and uses only symbols in $\left\{0,\ldots,M\right\}$, the sets $\Delta_i$ and $\Delta$ are compact. Moreover, $\Psi$ is a homeomorphism and
    \[\Psi\circ\sigma=\sigma^n\circ\Psi.\]
    Therefore, $\left(\Delta_i,\sigma^n|_{\Delta_i}\right)$ and $\left(\Delta,\sigma^n|_{\Delta}\right)$ are topologically conjugate to the one-sided (two-sided) full shifts on $\mathcal W_i$ and $\mathcal W$, respectively. This proves item (1).

    Put
    \[\Lambda_i=\bigcup_{s=0}^{n-1}\sigma^s(\Delta_i),\qquad \Lambda=\bigcup_{s=0}^{n-1}\sigma^s(\Delta).\]
    These are compact $\sigma$-invariant subsets of $\Sigma$. By item (1), (\ref{eq:thmD_cardinality}) and (\ref{eq:thmD_words}),
    \begin{align*}
        \htop(\sigma,\Lambda_i)
         & =\frac{1}{n}\htop(\sigma^n,\Lambda_i)               \\
         & \geq\frac{1}{n}\htop(\sigma^n,\Delta_i)             \\
         & =\frac{1}{n}\log\left\lvert\mathcal W_i\right\rvert \\
         & =\frac{1}{n}\log\left\lvert\Gamma_n^i\right\rvert   \\
         & >h_{\mu_i}(\sigma)-\eta
        >h_{\mu_i}(\sigma)-\ve.
    \end{align*}
    This proves item (2).

    We next prove item (3). For every $z\in\Delta_i$, there is a unique $y\in\Gamma_n^i$ such that the coordinates of $z$ from $0$ to $n-1$ form the word $\mathbf w(y)$. It follows from the definition of $D_n(\vphi_j)$, (\ref{eq:thmD_variation}) and (\ref{eq:thmD_average}) that
    \begin{equation}\label{eq:thmD_block_average}
        \left\lvert S_n\vphi_j(z)-\int\vphi_j\dif\mu_i\right\rvert
        \leq \frac{D_n(\vphi_j)}{n}
        +\left\lvert S_n\vphi_j(y)-\int\vphi_j\dif\mu_i\right\rvert
        <2\eta
    \end{equation}
    for any $1\leq j\leq k$.

    Define the map $\tau_n:\M(\Delta)\to\M(\Lambda)$ by
    \[\tau_n\lambda=\frac{1}{n}\sum_{s=0}^{n-1}\lambda\circ\sigma^{-s}.\]
    By Lemma \ref{lemTau},
    \begin{equation}\label{eq:thmD_tau}
        \tau_n\left(\M(\sigma^n,\Delta)\right)=\M(\sigma,\Lambda).
    \end{equation}
    The same statement holds with $\Delta$ and $\Lambda$ replaced by $\Delta_i$ and $\Lambda_i$, respectively. Given $\nu_i\in\M(\sigma,\Lambda_i)$, by (\ref{eq:thmD_tau}) there exists $\lambda_i\in\M(\sigma^n,\Delta_i)$ such that $\nu_i=\tau_n\lambda_i$. Hence, by (\ref{eq:thmD_block_average}),
    \begin{align*}
        \left\lvert\int\vphi_j\dif\nu_i-\int\vphi_j\dif\mu_i\right\rvert
         & =\left\lvert\int S_n\vphi_j\dif\lambda_i-\int\vphi_j\dif\mu_i\right\rvert \\
         & <2\eta<\ve
    \end{align*}
    for any $1\leq j\leq k$.

    Now fix $\nu\in\M(\sigma,\Lambda)$. By (\ref{eq:thmD_tau}), there exists $\lambda\in\M(\sigma^n,\Delta)$ such that $\nu=\tau_n\lambda$. If a word in $\mathcal W$ belongs to more than one $\mathcal W_i$, assign it to one such index and fix a map
    \[\iota:\mathcal W\to\left\{1,\ldots,m\right\}\quad\text{such that}\quad \mathbf w\in\mathcal W_{\iota(\mathbf w)}.\]
    For $1\leq i\leq m$, let
    \[C_i=\left\{z\in\Delta:\iota(z_0\ldots z_{n-1})=i\right\},\qquad q_i=\lambda(C_i).\]
    Then $\left\{C_i\right\}_{i=1}^m$ is a partition of $\Delta$, $q_i\geq0$ and $\sum_{i=1}^m q_i=1$. Define
    \[\mu=\sum_{i=1}^m q_i\mu_i\in\cov\left(\left\{\mu_i\right\}_{i=1}^m\right).\]
    The same argument as in (\ref{eq:thmD_block_average}) shows that for any $z\in C_i$,
    \[\left\lvert S_n\vphi_j(z)-\int\vphi_j\dif\mu_i\right\rvert<2\eta.\]
    Therefore,
    \begin{align*}
        \left\lvert\int\vphi_j\dif\nu-\int\vphi_j\dif\mu\right\rvert
         & =\left\lvert\int S_n\vphi_j\dif\lambda-
        \sum_{i=1}^m q_i\int\vphi_j\dif\mu_i\right\rvert                         \\
         & \leq\sum_{i=1}^m\int_{C_i}
        \left\lvert S_n\vphi_j(z)-\int\vphi_j\dif\mu_i\right\rvert\dif\lambda(z) \\
         & <2\eta<\ve
    \end{align*}
    for any $1\leq j\leq k$. This proves item (3).

    Finally, we prove item (4). We first claim that
    \begin{equation}\label{eq:thmD_return}
        \sigma^{tn}z\in\mathbf U\subseteq B(x,\ve)
    \end{equation}
    for any $z\in\Delta$ and any $t\in\mathbb I$ for which $tn\geq0$. In the one-sided case, this follows since the word starting at every coordinate $tn$ belongs to $\mathcal W$ and hence its first $q+1$ symbols agree with $x_0\ldots x_q$. In the two-sided case, its coordinates from $0$ to $q$ agree with $x_0\ldots x_q$. Its coordinates from $-q$ to $-1$ are the last $q$ symbols of the preceding period word $\mathbf w(y)$. Since $y$ is $n$-periodic and belongs to $\mathbf U$, these symbols are
    \[y_{n-q}\ldots y_{n-1}=y_{-q}\ldots y_{-1}=x_{-q}\ldots x_{-1}.
    \]
    Thus (\ref{eq:thmD_return}) also holds in the two-sided case.

    Given $z\in\Lambda$, there exist $w\in\Delta$ and $0\leq s\leq n-1$ such that $z=\sigma^s w$. Let $j=n-s\in\left\{1,\ldots,n\right\}$. It follows from (\ref{eq:thmD_return}) that
    \[\sigma^{j+tn}z=\sigma^{(t+1)n}w\in B(x,\ve)\]
    for any $t\in\Nb$. The same argument applies to every $z\in\Lambda_i$. This proves item (4) and completes the proof.
\end{proof}

Theorem \ref{thmD}, together with Lemma
\ref{lem:EntropyApproachableCMS}, provides the countable-state input
needed to run the nested construction from the proof of Theorem
\ref{thmA}.

\begin{theorem}\label{thmB}
    Suppose $(\Sigma,\sigma)$ is a non-trivial topologically transitive one-sided (two-sided) countable Markov shift. Let $\vphi \in C_{b,u}(\Sigma)$ such that $\mathrm{Int} L_\vphi \ne \emptyset$. Then for any $\ve > 0$, any $a \in \mathrm{Int} L_\vphi$, any $\mu \in \M(\sigma,\Sigma)$ satisfying $\int\vphi\dif\mu = a$ and
    \[h_\mu(\sigma) < \sup\left\{h_\omega(\sigma):\omega \in \M(\sigma,\Sigma), \int\varphi\dif\omega = a\right\},\]
    and any $0 \leq h \leq h_{\mu}(\sigma)$, there exists a compact invariant subset $\Gamma \subset \Sigma$ such that
    \begin{enumerate}
        \item $\left(\Gamma,\sigma|_{\Gamma}\right)$ is minimal and uniquely ergodic;
        \item $h_{\nu}(\sigma) = \htop\left(\sigma, \Gamma\right) = h$ where $\nu$ is the unique invariant measure on $\Gamma$;
        \item $\int \vphi \dif \nu = a$;
        \item $D(\nu,\mu) < \ve$.
    \end{enumerate}
\end{theorem}
\begin{proof}
    Fix $a \in \Int L_\varphi$, $\mu \in \M(\sigma,\Sigma)$ and
    $0 \leq h \leq h_\mu(\sigma)$ as in the statement. Since
    \[
        h_\mu(\sigma)
        <
        \sup\left\{h_\omega(\sigma):
        \omega\in\M(\sigma,\Sigma),\
        \int\varphi\dif\omega=a\right\},
    \]
    one has $h_\mu(\sigma)<\infty$.

    We first show that the measures used in Step (1) of the proof of
    Theorem \ref{thmA} can be chosen to have finite entropy. Take
    \[
        h_\mu(\sigma)<c <
        \sup\left\{h_\omega(\sigma):
        \omega\in\M(\sigma,\Sigma),\
        \int\varphi\dif\omega=a\right\}.
    \]
    There exists $\rho\in\M(\sigma,\Sigma)$ such that $\int\varphi\dif\rho=a$ and $h_\rho(\sigma)>c$.
    By Lemma \ref{lem:EntropyApproachableCMS}, there exists a sequence
    $\{\rho_n\}_{n\geq1}$ of ergodic measures supported on compact sets
    such that
    \[
        \rho_n\to\rho,
        \qquad h_{\rho_n}(\sigma)\to h_\rho(\sigma).
    \]
    In particular, every $\rho_n$ has finite entropy and
    $\int\varphi\dif\rho_n\to a$.

    Since $a\in\Int L_\varphi$ and periodic measures are dense in
    $\M(\sigma,\Sigma)$, there exist periodic measures
    $\omega_l,\omega_r$ such that $\int\varphi\dif\omega_l<a< \int\varphi\dif\omega_r$. By taking a convex combination of $\rho_n$ with $\omega_l$ or $\omega_r$, according as $\int\varphi\dif\rho_n>a$ or $\int\varphi\dif\rho_n<a$, we obtain $\widetilde\rho_n\in\M(\sigma,\Sigma)$ satisfying $\int\varphi\dif\widetilde\rho_n=a$. Since $\int\varphi\dif\rho_n\to a$, the coefficient of $\rho_n$ in
    this convex combination tends to one. Hence, for all sufficiently
    large $n$, we have $c<h_{\widetilde\rho_n}(\sigma)<\infty$. Fix such a measure and denote it by $\rho_a$.

    Take $\mu_a^1\in(\mu,\rho_a]$ sufficiently close to $\mu$. Then
    \[
        \int\varphi\dif\mu_a^1=a,\qquad
        h_{\mu_a^1}(\sigma)>h_\mu(\sigma)\geq h,
        \qquad D(\mu_a^1,\mu)<\ve/6.
    \]
    Moreover, using again the density of periodic measures and taking
    suitable convex combinations with $\omega_l$ or $\omega_r$, there
    exists $\mu_0^1\in\M(\sigma,\Sigma)$ such that
    \[
        \int\varphi\dif\mu_0^1=a,\qquad
        h_{\mu_0^1}(\sigma)=0,\qquad
        D(\mu_0^1,\mu)<\ve/6.
    \]

    Now take $\mu_l^1=\omega_l$ and $\mu_r^1=\omega_r$ in Step (1) of
    the proof of Theorem \ref{thmA}. Thus, all measures to which
    Theorem \ref{thmD} is applied in Step (1) have finite entropy.
    After Step (1), the construction is carried out inside compact
    subsystems conjugate to finite full shifts, so the same is true in
    every subsequent step.

    The remainder of the proof is the same as that of Theorem
    \ref{thmA}, with Theorem \ref{thmC} replaced by Theorem
    \ref{thmD}. Since \(D\) metrizes the weak* topology, the compactness of \(K=\cov\{\mu_i\}_{i=1}^m\) and a finite-subcover argument allow the test functions in Theorem \ref{thmD} to be chosen so that, for any prescribed \(\gamma>0\), item {\rm (3)} and the subsequent remark give
    \[
        D_{\mathrm H}\bigl(K,\M(\sigma,\Lambda)\bigr)<\gamma,
        \qquad
        D_{\mathrm H}\bigl(\{\mu_i\},\M(\sigma,\Lambda_i)\bigr)<\gamma.
    \]
    Thus Theorem \ref{thmD} provides all the Hausdorff estimates
    required in the proof of Theorem \ref{thmA}.
\end{proof}

\begin{remark}
    The only place the upper semi-continuity of the system's entropy function is used in the proof is in the final estimate of the lower bound for the entropy of $\nu$. In general, the entropy function is not upper semi-continuous \cite[Remark 8.2 (a)]{iommi2022escape}. However, the upper semi-continuity of the entropy function in the first horseshoe $\left(\Delta, \sigma^{N_1}\right)$ leads to the same conclusion.
\end{remark}

Finally, we apply our results to the Expanding-Markov-Renyi (EMR) interval maps. EMR maps were considered by Pollicott and Weiss in \cite{pollicott1999multifractal} when studying multifractal analysis of pointwise dimension.

\begin{definition}
    A map $T:I \to I$ (where $I = [0,1]$) is an \emph{EMR map}  if there exists a countable family $\left\{I_i\right\}_{i=0}^\infty$ of closed intervals (with disjoint interiors $\Int I_i$) with $I_i \subset I$ for any $i \in \Nb$ and $I = \bigcup_{i=0}^\infty I_i$, satisfying:
    \begin{enumerate}
        \item the map is $C^2$ on $\bigcup_{i=0}^\infty \Int I_i$;
        \item there exists $\xi > 1$ and $N \in \Nb$ such that for any $n \geq N$, one has $\left\lvert \left(T^n\right)'(x) \right\rvert > \xi^n$ for any $x \in \bigcup_{i=0}^\infty \Int I_i$;
        \item the map $T$ is Markov, and it can be coded by a full shift on a countable alphabet;
        \item the map satisfies the Renyi condition, that is, there exists a positive number $K > 0$ such that
              \[\sup_{n \in \Nb} \sup_{x,y,z\in I_n} \frac{\left\lvert T''(x)\right\rvert }{\left\lvert T'(y)\right\rvert \left\lvert T'(z)\right\rvert} \leq K\]
    \end{enumerate}
\end{definition}

Let $E = \bigcup_{i=0}^\infty \partial I_i$ and $O = \bigcup_{n=0}^\infty T^{-n} E$. Then \(O\) is countable and \(I\setminus O\) is forward
invariant. We assume throughout the paper that zero is the unique accumulation point of $E$. By the Markov property, every point
\(x\in I\setminus O\) has a unique symbolic coding; see
\cite[Section 2]{pollicott1999multifractal}. Let \((\Sigma,\sigma)\) denote the corresponding full shift on a countable alphabet. The natural
coding map $\pi:(\Sigma,\sigma)\longrightarrow(I\setminus O,T)$ is a topological conjugacy and continuous. In particular, it induces a homeomorphism $\pi_*: \M(\sigma, \Sigma) \to \M(T, I \setminus O)$.

\begin{lemma}
    $\pi:\Sigma \to I \setminus O$ is uniformly continuous.
\end{lemma}
\begin{proof}
    According to condition (2) of the EMR definition, there exists a positive integer $N$ and $\xi > 1$ such that $\left\lvert (T^n)'(x)\right\rvert > \xi^n$ for any $x \in I \setminus O$ and any $n \geq N$. Let $\omega = (\omega_0\omega_1\dots) \in \Sigma$. The point $x = \pi(\omega)$ is the unique point contained in the intersection of nested closed cylinders:
    $$x = \bigcap_{k=1}^{\infty} I_{\omega_0\omega_1\dots\omega_{k-1}}$$
    where $I_{\omega_0 \omega_1 \dots \omega_{k-1}} := \left\{x \in I : T^i(x) \in I_{\omega_i} \mathrm{\ for\ any\ } 0 \leq i \leq k - 1\right\} $. For any $k \geq N+1$, by the Mean Value Theorem,
    \[\diam\left(I_{\omega_0 \dots \omega_{k-1}}\right) < \diam I_{\omega_{k-1}} \cdot \xi^{-(k-1)} \leq \xi^{-(k-1)}\]
    For any $\ve > 0$, we can choose $k \ge N+1$ large enough such that $\xi^{-(k-1)} < \ve$. Now, let $\delta = 2^{-k}$. For any two points $\omega, \eta \in \Sigma$, if $d(\omega, \eta) < \delta$, then $\omega_i = \eta_i$ for $0 \leq i < k$. Consequently, both $\pi(\omega)$ and $\pi(\eta)$ belong to the same $k$-th order cylinder set $I_{\omega_0 \dots \omega_{k-1}}$. Therefore,
    $$\left\lvert \pi(\omega) - \pi(\eta)\right\rvert  \leq \diam\left(I_{\omega_0 \dots \omega_{k-1}}\right)  < \ve.$$
    Since $k$ (and thus $\delta$) depends only on $\ve$ and $\xi$, the map $\pi$ is uniformly continuous.
\end{proof}

\begin{theorem}
    Suppose $(I,f)$ is an EMR map. Let $\vphi \in C_{b,u}(I \setminus O)$ such that $\mathrm{Int} L_\vphi \ne \emptyset$. Then for any $\ve > 0$, any $a \in \mathrm{Int} L_\vphi$, any $\mu \in \M(f,I \setminus O)$ satisfying $\int\vphi\dif\mu = a$ and
    \[h_\mu(f) < \sup\left\{h_\omega(f):\omega \in \M(f,I \setminus O), \int\varphi\dif\omega = a\right\},\]
    and any $0 \leq h \leq h_{\mu}(f)$, there exists a compact invariant subset $\Gamma \subset I \setminus O$ such that
    \begin{enumerate}
        \item $\left(\Gamma,f|_{\Gamma}\right)$ is minimal and uniquely ergodic;
        \item $h_{\nu}(f) = \htop\left(f, \Gamma\right) = h$ where $\nu$ is the unique invariant measure on $\Gamma$;
        \item $\int \vphi \dif \nu = a$;
        \item $D(\nu,\mu) < \ve$.
    \end{enumerate}
\end{theorem}
\begin{proof}
    Fix $\ve > 0$, $a \in \mathrm{Int} L_\vphi$, $\mu \in \M(f,I \setminus O)$ and $0 \leq h \leq h_{\mu}(f)$ satisfying $\int\vphi\dif\mu = a$ and
    \[h_\mu(f) < \sup\left\{h_\omega(f):\omega \in \M(f,I \setminus O), \int\varphi\dif\omega = a\right\}.\]

    Let $\widetilde\mu=\pi_*^{-1}\mu$. By the continuity of $\pi_*$ at
    $\widetilde\mu$, there exists $\delta>0$ such that
    \[
        D_\Sigma(\widetilde\nu,\widetilde\mu)<\delta
        \quad\Longrightarrow\quad
        D_I(\pi_*\widetilde\nu,\mu)<\ve
    \]
    for any $\widetilde\nu\in\M(\sigma,\Sigma)$. Let $\psi = \vphi \circ \pi$. Then $\psi \in C_{b,u}(\Sigma)$ and $\mathrm{Int} L_{\psi} \ne \emptyset$. Since $(\Sigma,\sigma)$ is topologically conjugate to $(f,I \setminus O)$, the conclusion follows from Theorem \ref{thmB}.
\end{proof}

\part{Symbolic systems with non-uniform structure}
In this part, we study the subshifts with non-uniform structure, i.e., subshifts on a finite alphabet with language $\mathcal{L}$ such that $\mathcal{G} \subset \mathcal{L}$ has $(W)$-specification and $\mathcal{L}$ is edit approachable by $\mathcal{G}$. For convenience, system $(X, \sigma)$ is denoted by $X$ in this part.

\section{Preliminary}\label{sec:pre_symbolic}
\subsection{Symbolic systems}
For any finite alphabet $A$, the \emph{full symbolic space} is the set
$$A^{\Zb} = \left\{(\ldots x_{-1}x_0x_1\ldots) : x_i \in A\right\}$$
which is viewed as a compact topological space with the discrete product topology. The shift action on full symbolic space is defined by
\[ \sigma:A^{\Zb} \to A^{\Zb}, \quad (\ldots x_{-1}x_{0}x_1\ldots) \mapsto (\ldots x_0x_1x_2\ldots). \]
$\left(A^{\Zb}, \sigma\right)$ forms a dynamical system under the discrete product topology and is called a two-sided shift. We equip $A^{\Zb}$ with the following metric. Let $q = \left\lvert A\right\rvert $. For $x = \left(\ldots x_{-1}x_0x_1\ldots\right) $ and $y = \left(\ldots y_{-1}y_0y_1\ldots\right)$ with $x \ne y$, let
\[ d(x,y) =  q^{-\min\left\{\left\lvert i\right\rvert : i \in \Zb, x_i \ne y_i \right\} }.\]

A closed subset $X \subseteq A^{\Zb}$ is called \emph{two-sided subshift} if it is invariant under the shift action $\sigma$. Define $A^n = \left\{x_0\ldots x_{n-1}: x_i \in A\right\} $. Then $w \in A^n$ is a \emph{word} of the subshift $X$ if there is an $x \in X$ and $k \in \Zb$ such that $w = x_{k}x_{k+1}\ldots x_{k+n-1}$. We say $n$ is the \emph{length} of $w$. The \emph{language} of a subshift $X$, denoted by $\cL(X)$, is the set of all words of the subshift $X$. Denote $\cL_n(X) := \cL(X) \cap A^n$, all words of $X$ with length $n$.

Similarly, one can define one-sided shifts and one-sided subshifts. $A^{\Nb} = \left\{(x_0\ldots) : x_i \in A\right\} $ is called the \emph{one-sided full symbolic space}. The shift action on one-sided full symbolic space is defined by
\[ \sigma:A^{\Nb} \to A^{\Nb}, \quad (x_0x_1\ldots) \mapsto (x_1x_2\ldots). \]
$\left(A^{\Nb}, \sigma\right)$ forms a dynamical system under the discrete product topology and is called a one-sided shift. We equip $A^{\Nb}$ with the following metric. For $x = \left(x_0x_1\ldots\right) $ and $y = \left(y_0y_1\ldots\right)$ with $x \ne y$, let
\[ d(x,y) =  q^{-\min\left\{i \in \Nb: x_i \ne y_i \right\} }.\]
A closed subset $X \subseteq A^{\Nb}$ is called \emph{one-sided subshift} if it is invariant under the shift action $\sigma$. We endow the space $\M(X)$ with the same metric as defined in Section \ref{subsec:M(X)}. For $w=w_0\ldots w_{n-1}\in\mathcal L_n$, let
\[
    [w]:=\{x\in X:x_i=w_i\text{ for }0\leq i\leq n-1\}.
\]

For two-sided shifts, the following lemma shows the relationship of one system and its iteration from the aspects of expansiveness, transitivity and the shadowing property.

\begin{lemma}\label{lemma-f^n-f}\cite[Proposition 3.3]{dong2024abundance}
    Let $X$ be a  two-sided shift. Consider a compact set $\Delta \subseteq X$ which satisfies $\sigma^n(\Delta) = \Delta$ for some $n \in \Nb^+$ and let $\Lambda =\cup_{i=0}^{n-1}\sigma^i(\Delta)$. If $\sigma^i(\Delta) \cap \sigma^j(\Delta) = \emptyset$ for any $0 \leq i < j \leq n-1$, then the following statements hold.
    \begin{enumerate}
        \item If $(\Delta, \sigma^n)$ is expansive, then $(\Lambda, \sigma)$ is expansive.
        \item If $(\Delta, \sigma^n)$ is topologically transitive, then $(\Lambda, \sigma)$ is topologically transitive.
        \item If $(\Delta, \sigma^n)$ has the shadowing property, then $(\Lambda, \sigma)$ also has the shadowing property.
    \end{enumerate}
\end{lemma}

\subsection{Topological entropy and metric entropy}
Let $X$ be a one-sided or two-sided subshift on a finite alphabet with language $\mathcal{L}$. Given a collection $\mathcal{D} \subset \mathcal{L}$, the topological entropy of $\mathcal{D}$ is
$$\htop(\sigma, \mathcal{D}):= \limsup_{n\rightarrow\infty}\frac{1}{n}\log\left\lvert \mathcal{D}_n\right\rvert.$$
The entropy of an invariant measure $\mu \in \mathcal{M}(\sigma, X)$ is
$$h_{\mu}(\sigma) := \lim_{n\rightarrow\infty}\frac{1}{n}\sum_{w\in\mathcal{L}_n}-\mu([w])\log\mu([w]).$$
For a fixed function $\varphi \in C(X)$, we define a function $\varphi_n : \mathcal{L}_n \rightarrow \mathbb{R}$  by $\varphi_n(w)= \sup_ {x\in [w]} S_n \varphi(x)$, where $S_n\varphi(x)=\varphi(x)+\varphi(\sigma x)+\cdots+ \varphi(\sigma^{n-1}x)$.  We consider the quantities $\Lambda_n(\mathcal{D},\varphi)= \sum_{ w\in \mathcal{D}_n} e^{\varphi_n(w)}$. The upper capacity pressure of $\varphi$ on $\mathcal{D}$ is given by $$P(\mathcal{D},\varphi)= \limsup_{ n\rightarrow \infty} \frac{1}{n} \log \Lambda_n(\mathcal{D},\varphi).$$
We write $P(\varphi)=P(\mathcal{L},\varphi)$.

For non-compact sets,
Bowen developed a satisfying definition via dimension language \cite{bowen1973topological}. Readers may refer to Definition \ref{def:htop}.

\subsection{Symbolic systems with non-uniform structure}
Climenhaga, Thompson and Yamamoto \cite{climenhaga2017large} introduced a class of subshifts, called \emph{systems with non-uniform structure}. The details of the definitions will be given in the following.

Let $X$ be a one-sided or two-sided subshift on a finite alphabet with language $\mathcal{L}$.

\begin{definition}[{{\cite[Definition 2.5]{climenhaga2017large}}}]
    Define an \emph{edit} of a word $w \in \cL$ to be a transformation of $w$ by one of the following actions, where $a, a' \in A$ are arbitrary letters and $u_1,u_2$ are possibly empty words:
    \begin{enumerate}
        \item Substitution:
              $w=u_1au_2\mapsto w'=u_1a'u_2$.
        \item Insertion:
              $w=u_1u_2\mapsto w'=u_1a'u_2$.
        \item Deletion:
              $w=u_1au_2\mapsto w'=u_1u_2$,
    \end{enumerate}
    Given $v, w \in \mathcal{L}$, define the \emph{edit distance} between $v$ and $w$ to be the minimum number of edits required to transform the word $v$ into the word $w$, which is denoted by $\widehat{d}(v, w)$.
\end{definition}

\begin{definition}[{{\cite[Definition 2.7]{climenhaga2017large}}}]\label{def:edit_approachable}
    A non-decreasing function $g: \mathbb{N}^+ \rightarrow \mathbb{N}^+$ is called a \emph{mistake function} if $\lim_{n \rightarrow \infty} g(n)/n = 0$. For any $\mathcal{G} \subset \mathcal{L}$, we say that $\mathcal{L}$ is \emph{edit approachable} by $\mathcal{G}$ if there exists a mistake function $g$ such that for any $w \in \mathcal{L}$, there exists $v \in \mathcal{G}$ with $\widehat{d}(w, v) \leq g(|w|)$.
\end{definition}

The following lemmas describe the size of the ball in the metric of edit distance.

\begin{lemma}[{{\cite[Lemma 2.6]{climenhaga2017large}}}]\label{Lem-the-size-of-balls}
    There is $C > 0$ such that for any $n \in \Nb^+$, any $w \in \mathcal{L}_n$ and any $0 < \delta \leq 1$, we have
    $$ \left\lvert \left\{v \in \mathcal{L} : \widehat{d}(v, w) \leq \delta n\right\}\right\rvert   \leq Cn^C(e^{C\delta}e^{-\delta \log \delta})^n.$$
\end{lemma}

Since \(X\) is compact, the metric used in
\cite[Lemma 2.2]{zhao2018topological} and the Wasserstein metric
\(D\) are uniformly equivalent on \(\M(X)\); hence the stated
form follows.

\begin{lemma}[{{\cite[Lemma 2.2]{zhao2018topological}}}]\label{Lem-closedistance}
    For any mistake function $g$, there is a sequence of positive numbers $\left\{\delta_n\right\}_{n=1}^{\infty}$ with $\lim_{n\to\infty} \delta_n = 0$ such that for any $x, y \in X$ and any $m, n \in \Nb^+$ with $\widehat{d}(x_0\dots x_{n-1}, y_0 \dots y_{m-1}) \leq g(n)$, one has
    $$D\left(\cE_n(x), \cE_m(y)\right) \leq \delta_n.$$
\end{lemma}

\subsection{Specification property for symbolic systems}
In this subsection, we recall the specification property for symbolic systems and some related lemmas.

Let $X$ be a one-sided or two-sided subshift on a finite alphabet with language $\mathcal{L}$. Given words $u, v$, we use juxtaposition $u v$ to denote the word obtained by concatenation. Let $e$ denote the empty word. In the following definition, the gluing word is allowed to be $e$.

\begin{definition}[{{\cite[Definition 2.2]{climenhaga2017large}}}]
    Consider a subset $\mathcal{G} \subset \mathcal{L}$. We say $\mathcal{G} $ has \emph{(W)-specification} with gap length $\tau \in \mathbb{N}$ if for every $u, v \in \mathcal{G}$, there exists $w \in \mathcal{L} \cup \left\{e\right\} $ such that $u w v \in \mathcal{G}$ and $|w| \leq \tau$. If the word $w$ can always be taken to have length exactly $\tau$, we say that $\mathcal{G}$ has \emph{(S)-specification}.
\end{definition}

An important special case, which corresponds to specification with gap length $0$, is the following.

\begin{definition}[{{\cite[Definition 2.3]{climenhaga2017large}}}]
    We say the collection $\mathcal{G} \subset \mathcal{L}$ has the \emph{free concatenation property} if for all $u, w \in \mathcal{G}$, we have $uw \in \mathcal{G}$.
\end{definition}

By applying \cite[Proposition 4.2, Lemma 4.3]{climenhaga2017large}, we get the following lemma which is the key point to construct horseshoes.

\begin{lemma}\label{Lem-FLapproachable}
    Let $X$ be a one-sided or two-sided subshift on a finite alphabet with language $\mathcal{L}$ and $\mathcal{G} \subset \mathcal{L}$. If $\mathcal{G}$ has $(W)$-specification and $\mathcal{L}$ is edit approachable by $\mathcal{G}$, then there exists $\mathcal{F} \subset \mathcal{L}$ such that $\mathcal{F}$ has free concatenation property and $\mathcal{L}$ is edit approachable by $\mathcal{F}$.
\end{lemma}

A natural question is whether the symbolic systems with non-uniform structure are entropy-dense. The following result proved by Climenhaga, Thompson and Yamamoto shows the entropy-dense property of those systems. It is analogous to the Katok's approximation theorem in smooth dynamical systems, see \cite[Theorem 3.3]{avila2021c1}.

\begin{theorem}[{{\cite[Proposition 3.6, Remark 4.5]{climenhaga2017large}}}]\label{Thm-horseshoe}
    Let $X$ be a one-sided or two-sided subshift on a finite alphabet with language $\mathcal{L}$. Suppose that $\mathcal{G} \subset \mathcal{L}$ has the $(W)$-specification property and $\mathcal{L}$ is edit approachable by $\mathcal{G}$. Then there exists an increasing sequence $\{X_l\}_{l\in\mathbb{N}^{+}}$ of compact $\sigma$-invariant subsets of $X$ with the following properties.
    \begin{enumerate}
        \item \label{item:thm 8.10(1)} Each $X_l$ is a topologically transitive sofic shift and has $(W)$-specification.
        \item \label{item:thm 8.10(2)} There is $T \in \Nb$ such that for every $n$ and every $w \in \cL(X_n)$, there are $u,v \in \cL$ with $\left\lvert u\right\rvert, \left\lvert v\right\rvert \le n + T $ such that $uwv \in \mathcal{G}$.
        \item \label{item:thm 8.10(3)} Every invariant measure on $X$ is entropy approachable by ergodic measures on $\{X_l\}_{l\in\mathbb{N}^{+}}$, i.e., for any $\eta> 0$, any $\mu \in \M(\sigma,X)$, and any neighborhood $U$ of $\mu \in \M(\sigma,X)$, there exist $l \geq 1$ and $\nu \in \Merg(\sigma, X_l)\cap U$ such that $h_{\nu}(\sigma) >h_{\mu}(\sigma)-\eta$.
    \end{enumerate}

    If further every word in $\cL$ can be extended to a word in $\mathcal{G}$, then $X=\overline{\cup_{l\in \Nb^+}X_l}$.
\end{theorem}

\subsection{Some useful facts and lemmas}
In this subsection, we recall some useful facts and lemmas. Let $X$ be a one-sided or two-sided subshift on a finite alphabet with language $\mathcal{L}$.

Given $\mu \in  \mathcal{M}(\sigma, X)$ and $\ve> 0$, let
\begin{equation}\label{def:L}
    \mathcal{L}_n^{\mu,\ve}:=\{w \in \mathcal{L}_n : D(\cE_n(x), \mu) < \ve \textit{ for any } x \in [w] \}.
\end{equation}
By \cite[Proposition 2.1, Proposition 4.1]{pfister2005large}, we have the following lemma. It shows the relation between entropy and the cardinality of $\mathcal{L}_n^{\mu,\ve}$.

\begin{lemma}\label{Lemma-number-approximate}
    For any ergodic measure $\mu \in \Merg(\sigma, X)$, any $\ve > 0$ and any $\delta > 0$, there exists $N(\mu, \ve, \delta) \in \Nb^+$ such that for any $n \geq N(\mu, \ve, \delta)$, we have
    $\left\lvert \mathcal{L}_n^{\mu,\ve}\right\rvert \geq e^{n\left(h_{\mu}(\sigma)-\delta\right)}.$
\end{lemma}

Recall we call $\mu \in \M(\sigma, X)$ a periodic measure if $\supp(\mu)$ is a periodic orbit.

\begin{lemma}\label{Lem-dense-periodic-measure}
    Suppose that $\mathcal{G} \subset \mathcal{L}$ has $(W)$-specification and $\mathcal{L}$ is edit approachable by $\mathcal{G}$. Then periodic measures are dense in $\M(\sigma, X)$.
\end{lemma}
\begin{proof}
    We only prove the case of two-sided subshifts since the case of one-sided subshifts is similar. Fix $\ve >0$, and $\nu \in \mathcal{M}(\sigma, X)$. By Theorem \ref{Thm-horseshoe}(3), there exists $\mu \in \Merg(\sigma, X)$ such that $D(\nu,\mu)\leq \ve/2$. By Lemma \ref{Lem-FLapproachable}, there exists $\mathcal{F} \subset \mathcal{L}$  which has free concatenation property and $\mathcal{L}$ is edit approachable by  $\mathcal{F}$. Then there exists a mistake function $g$ and a map $\phi_{\mathcal{F}}: \mathcal{L} \rightarrow \mathcal{F}$ such that $\widehat{d}(w, \phi_{\mathcal{F}}(w)) \leq g(|w|)$ for every $w \in \mathcal{L}$. By Lemma \ref{Lem-closedistance}, there is a sequence of positive numbers $\left\{\delta_n\right\}_n^\infty$ with $\lim_{n \to \infty}\delta_n = 0$ such that
    \[D(\cE_n(x), \cE_{|\phi_{\mathcal{F}}(w)|}(y)) \leq \delta_n\]
    for any $w\in \mathcal{L}_n$, any $x \in [w]$ and any $ y \in [\phi_{\mathcal{F}}(w)]$.
    Take $N\in \Nb^+$ such that
    \begin{equation}\label{eq-lem-n(2)}
        \delta_n \leq \ve/4
    \end{equation}
    for any $n \geq N$. By Lemma \ref{Lemma-number-approximate}, there exists \(n\geq N\) such that $\mathcal L_n^{\mu,\ve/4}\neq\emptyset$. Choose \(w\in\mathcal L_n^{\mu,\ve/4}\) and \(x\in[w]\). Let $y = \dots \phi_{\mathcal{F}}(w) \phi_{\mathcal{F}}(w) \dots $ generated by $\phi_{\mathcal{F}}(w)$. It has a period $|\phi_{\mathcal{F}}(w)|$. By the definition of $\mathcal{L}^{\mu ,\frac{\ve}{4}}_n$, we have $D(\cE_n(x), \mu)<\ve / 4$. By (\ref{eq-lem-n(2)}),
    \begin{align*}
        D\left(\cE_{\left\lvert \phi_{\mathcal{F}}(w)\right\rvert }(y),\nu\right) & \leq  D\left(\cE_{\left\lvert \phi_{\mathcal{F}}(w)\right\rvert }(y),\cE_n(x)\right)  + D\left(\cE_n(x),\mu\right) + D(\mu,\nu) \\
                                                                                  & < \frac{\ve}{4} + \frac{\ve}{4} + \frac{\ve}{2} = \ve,
    \end{align*}
    which completes the proof.
\end{proof}

Note every periodic measure has zero metric entropy. Then invariant measures with zero metric entropy are dense in the space of invariant measures by the above lemma.

\section{Conditional entropy realization for symbolic systems with non-uniform structure}\label{sec:minimal_symbolic}
In this section, we extend the preceding multi-horseshoe result to symbolic systems with non-uniform structure and use it to obtain exact entropy-and-average realization and weak* approximation by uniquely ergodic measures with compact support.

\begin{theoremx}\label{prop-nonuniform-multi-horseshoe}
    Let $X$ be a one-sided or two-sided subshift on a finite alphabet with language $\mathcal{L}$ and positive topological entropy. Suppose that $\mathcal{G} \subset \mathcal{L}$ has $(W)$-specification and $\mathcal{L}$ is edit approachable by $\mathcal{G}$. Then for any $W = \cov\{\alpha_i\}^m_{i=1} \subseteq \mathcal{M}(\sigma, X)$ and any $\eta,\xi> 0$, there exist compact invariant subsets $\Lambda_i\subseteq \Lambda \varsubsetneq X$ such that for each $1 \leq i \leq m$, the following hold.
    \begin{enumerate}
        \item If $(X,\sigma)$ is a two-sided shift, then $(\Lambda_i, \sigma|_{\Lambda_i})$ and $(\Lambda, \sigma|_{\Lambda})$ are transitive subshifts of finite type; if $(X,\sigma)$ is a one-sided shift, then $(\Lambda_i, \sigma|_{\Lambda_i})$ and $(\Lambda, \sigma|_{\Lambda})$ satisfy $(W)$-specification.
        \item $\htop(\sigma,\mathcal{L}(\Lambda_i)) \geq h_{\alpha_i}(\sigma)-\eta $ and consequently, $\htop(\sigma,\mathcal{L}( \Lambda)) \ge \sup\{h_{\kappa}(\sigma) : \kappa \in W\}-\eta$.
        \item $D_{\mathrm{H}}\left(W, \mathcal{M}(\sigma, \Lambda)\right)  < \xi, D_{\mathrm{H}}\left(\left\{\alpha_i\right\} , \mathcal{M}\left(\sigma, \Lambda_i\right) \right)  < \xi$.
    \end{enumerate}
\end{theoremx}

\begin{proof}
    We first reduce the proof to the case where $h_{\alpha_i}(\sigma) > 0$ for any $1 \leq i \leq m$. To distinguish the original data from the
    auxiliary data used below, write
    \[
        \overline{W}=\cov\{\overline{\alpha}_i\}_{i=1}^m,
        \qquad \overline{\eta}>0,
        \qquad \overline{\xi}>0
    \]
    for the set and the two error tolerances appearing in the statement.
    After deleting repetitions among the measures
    $\overline{\alpha}_i$, and assigning the same subsystem to repeated
    measures at the end, we may assume that they are pairwise distinct.

    Put $H=\htop(\sigma,\mathcal L)>0$. By the variational principle, there is a measure
    $\nu\in\M(\sigma,X)$ with $h_\nu(\sigma)>0$.  Choose
    $\lambda\in(0,1)$ sufficiently small that
    \begin{equation}\label{eq:positive-entropy-perturbation}
        \lambda <\frac{\overline{\xi}}4,
        \qquad
        \lambda H<\frac{\overline{\eta}}4.
    \end{equation}
    Define
    \[
        \beta_i=(1-\lambda)\overline{\alpha}_i+\lambda\nu,
        \qquad
        W_\lambda=\cov\{\beta_i\}_{i=1}^m.
    \]
    Note
    \begin{equation}\label{eq:positive-entropy-beta}
        b:=\min_{1\leq i\leq m}h_{\beta_i}(\sigma) > 0.
    \end{equation}
    Fix
    \begin{equation}\label{eq:auxiliary-errors}
        0 < \eta_0<\min\left\{\frac b2,\frac H4,
        \overline{\eta}-\lambda H\right\},
        \qquad
        0 < \xi_0<\overline{\xi}-\lambda.
    \end{equation}
    We now deal with the construction for
    $W_\lambda=\cov\{\beta_i\}_{i=1}^m$ with
    $\eta_0$ and $\xi_0$. To avoid introducing new notation throughout
    the construction, for the remainder of the main argument we relabel
    \[
        \beta_i,\quad W_\lambda,\quad \eta_0,\quad \xi_0
        \quad\text{as}\quad
        \alpha_i,\quad W,\quad \eta,\quad \xi,
    \]
    respectively. Thus
    \begin{equation}\label{eq:positive-entropy-relabelled}
        0<\eta<\frac12\min_{1\leq i\leq m}h_{\alpha_i}(\sigma),
        \qquad \eta<\frac H4.
    \end{equation}

    By the classical variational principle, there exists $\alpha_0\in \M(\sigma,X)$ such that
    \begin{equation}\label{eq:alpha_0}
        h_{\alpha_0}(\sigma) > \htop(\sigma,\mathcal{L}) - \eta / 4.
    \end{equation}
    Let
    \[
        \rho_0=
        \begin{cases}
            \min\{D(\alpha_i,\alpha_j):1 \leq i<j\leq m\}/2, & m\geq2, \\
            +\infty,                                         & m=1.
        \end{cases}
    \]
    Note that $\rho_0>0$. There are infinitely many ergodic measures on $X$  by Theorem \ref{Thm-horseshoe}(\ref{item:thm 8.10(3)}). Since an ergodic measure is an extreme point of $\M(\sigma,X)$, not all of these ergodic measures can belong to the convex hull of the finitely many measures $\alpha_i$. Hence $W\varsubsetneq\M(\sigma,X)$, and compactness gives $D_{\mathrm{H}}(W, \M(\sigma,X)) > 0$. Take
    \[0<\ve < \min\{D_{\mathrm{H}}(W, \M(\sigma,X)),\eta/4,\rho_0,\xi\}\]
    small enough such that
    \begin{equation}\label{eq:Cve}
        C\ve -\ve \log \ve < \eta / 20
    \end{equation}
    where $C>0$ is decided by Lemma \ref{Lem-the-size-of-balls}. By Theorem \ref{Thm-horseshoe}(\ref{item:thm 8.10(3)}), for any $0 \leq j \leq m$, there exists $X_{l_j}\subseteq X$ and $\mu_j \in \Merg(\sigma, X_{l_j})\subset \Merg(\sigma, X)$ such that
    \begin{equation}\label{eq-Thm1.1-1}
        \mu_j \in B\left(\alpha_j,\ve/4\right)  \textit{ and } h_{\mu_j}(\sigma)>h_{\alpha_j}(\sigma)-\eta/16.
    \end{equation}
    It follows from Lemma \ref{Lemma-number-approximate} and the definition of topological entropy that there exists $\widehat{N} \in \Nb^+$ such that for any $n >\widehat{N}$ and any $0 \leq j \leq m$, one has
    \begin{equation}\label{eq:9.2}
        \left\lvert \mathcal{L}_n^{\mu_j,\frac{\ve}{4}}\right\rvert  \geq e^{n\left(h_{\mu_j}(\sigma)-\eta/16\right)}
    \end{equation}
    and
    \begin{equation}\label{eq-entropy-n}
        \left\lvert \mathcal{L}_n\right\rvert  < e^{n \left(\htop(\sigma, \mathcal{L})+\eta/8\right)}
    \end{equation}
    By the definition of $\mathcal{L}_n^{\alpha_j,\frac{\ve}{2}}$ (see (\ref{def:L})), (\ref{eq-Thm1.1-1}) and (\ref{eq:9.2}) we have
    \begin{equation}\label{eq-number-of-L}
        \left\lvert \mathcal{L}_n^{\alpha_j,\frac{\ve}{2}}\right\rvert \geq \left\lvert \mathcal{L}_n^{\mu_j,\frac{\ve}{4}} \right\rvert  \geq e^{n\left(h_{\alpha_j}(\sigma)-\eta/8\right)}.
    \end{equation}

    By Lemma \ref{Lem-FLapproachable}, there exists a subset $\mathcal{F} \subset \mathcal{L}$ with the free concatenation property such that $\mathcal{L}$ is edit approachable by $\mathcal{F}$. Then there is a mistake function $g$ and a map $\phi_{\mathcal{F}} : \mathcal{L}\rightarrow \mathcal{F}$ such that
    \begin{equation}\label{eq:phi_cF}
        \widehat{d}(w, \phi_{\mathcal{F}}(w)) \leq g(|w|)
    \end{equation}
    for any $w \in \mathcal{L}$. By Lemma \ref{Lem-closedistance}, there is a sequence of positive numbers $\left\{\delta_n\right\}_{n = 1}^\infty$ with $\lim_{n \to \infty} \delta_n = 0$ such that
    \begin{equation}\label{eq:D}
        D\left(\cE_n(x), \cE_{|\phi_{\mathcal{F}}(w_j)|}(y)\right)  \leq \delta_n
    \end{equation}
    for any $0 \leq j \leq m$, any $w_j \in \mathcal{L}_n^{\alpha_j,\frac{\ve}{2}}$, any $x \in [w_j]$ and any $y \in [\phi_{\mathcal{F}}(w_j)]$. We claim that there exists $n>\widehat{N}$ such that
    \begin{equation}\label{eq-n0}
        \left\lfloor \frac{n-g(n)}{2}\right\rfloor  > \widehat{N},
    \end{equation}
    \begin{equation}\label{eq-n(1)}
        g(n)/n \leq \min\left\{1/9,\ve/4\right\}  ,\  2g(n)+1\leq e^{(n-g(n))\eta/4},
    \end{equation}
    \begin{equation}\label{eq-n(1.1)}
        g(n)\left(\max_{0 \leq j \leq m}\{h_{\alpha_j}(\sigma)\}-\eta/2\right)   \leq \eta n/4,
    \end{equation}
    \begin{equation}\label{eq-n(2)}
        \delta_n \leq\ve/4,
    \end{equation}
    \begin{equation}\label{eq-n(3)} C(n+g(n))^Ce^{\eta\left(n+g(n)\right)/20 }\leq e^{\eta (n-g(n))/16},
    \end{equation}
    \begin{equation}\label{eq-n(4)}
        e^{{(n-g(n))}^{m+1}(h_{\alpha_j}(\sigma)-3\eta/4)}-(n-g(n))^{m+1}\geq  e^{{(n+g(n))}^{m+1} \left(h_{\alpha_j}(\sigma)-9\eta/10\right)}
    \end{equation}
    for any $0 \leq j \leq m$ and
    \begin{equation}\label{eq-n(5)}
        \left\lfloor \frac{(n+g(n))^{m+1}}{2}\right\rfloor  e^{ \left\lfloor \frac{(n+g(n))^{m+1}}{2}\right\rfloor  \left(\htop(\sigma,  \mathcal{L})+\eta/2\right)  } +\sum_{k=1}^{\widehat{N}}  \left\lvert \mathcal{L}_k\right\rvert
        \leq e^{(n-g(n))^{m+1}(\htop(\sigma,\mathcal{L})-\eta)}.
    \end{equation}
    Here, (\ref{eq-n0}), (\ref{eq-n(1)}), (\ref{eq-n(1.1)}),
    (\ref{eq-n(3)}) and (\ref{eq-n(5)}) follow from
    $\lim_{n \to \infty}g(n)/n=0$, together with
    (\ref{eq:positive-entropy-relabelled}).  Let us justify
    (\ref{eq-n(4)}) explicitly, since this is where positivity of the
    entropies is needed.  For $0\leq j\leq m$, put
    \[
        A_{j,n}=(n-g(n))^{m+1}
        \left(h_{\alpha_j}(\sigma)-\frac{3\eta}{4}\right),
        \qquad
        B_{j,n}=(n+g(n))^{m+1}
        \left(h_{\alpha_j}(\sigma)-\frac{9\eta}{10}\right).
    \]
    For $1\leq j\leq m$, (\ref{eq:positive-entropy-relabelled}) gives
    $h_{\alpha_j}(\sigma)>2\eta$.  For $j=0$, equations
    (\ref{eq:alpha_0}) and (\ref{eq:positive-entropy-relabelled}) give
    $h_{\alpha_0}(\sigma)>H-\eta/4>3\eta/4$.  Consequently,
    $A_{j,n}\to+\infty$ for every $0\leq j\leq m$.  Moreover, writing
    $r_n=g(n)/n$, we have
    \begin{align*}
        \frac{A_{j,n}-B_{j,n}}{n^{m+1}}
        ={} & (1-r_n)^{m+1}
        \left(h_{\alpha_j}(\sigma)-\frac{3\eta}{4}\right) \\
            & -(1+r_n)^{m+1}
        \left(h_{\alpha_j}(\sigma)-\frac{9\eta}{10}\right)
        \longrightarrow \frac{3\eta}{20}>0.
    \end{align*}
    Thus $A_{j,n}-B_{j,n}\to+\infty$, uniformly over the finite set
    $0\leq j\leq m$.  It follows that, for all sufficiently large $n$,
    \[
        e^{A_{j,n}}-(n-g(n))^{m+1}\geq e^{B_{j,n}}
        \qquad(0\leq j\leq m),
    \]
    which is precisely (\ref{eq-n(4)}).  Finally,
    (\ref{eq-n(2)}) follows from
    $\lim_{n \to \infty}\delta_n=0$.
    By the definition of $\phi_{\cF}$ (see (\ref{eq:phi_cF})), one has
    \begin{equation}\label{eq:phi_cF_control}
        n-g(n) \leq |\phi_{\mathcal{F}}(w)| \leq n + g(n)
    \end{equation}
    for any $w \in \mathcal{L}_n^{\alpha_j,\frac{\ve}{2}}$. For any $v \in \mathcal{F}$, we have
    \begin{equation}\label{eq-numberofpreimage}
        \begin{split}
                                                                      & \phantom{{}\leq{}} \left\lvert \left\{w\in \mathcal{L}_n^{\alpha_j,\frac{\ve}{2}} : \phi_{\mathcal{F}}(w)=v\right\}\right\rvert \\
            (\text{by (\ref{eq:phi_cF})})                             & \leq \left\lvert \left\{w\in \mathcal{L}_n^{\alpha_j,\frac{\ve}{2}} : \widehat{d}(w,v)\leq g(n)\right\} \right\rvert \quad      \\
            (\text{by (\ref{eq-n(1)}) and (\ref{eq:phi_cF_control})}) & \leq   \left\lvert \left\{w\in \mathcal{L}_n^{\alpha_j,\frac{\ve}{2}} : \widehat{d}(w,v)\leq \ve |v|\right\} \right\rvert       \\
            (\text{by Lemma \ref{Lem-the-size-of-balls}})             & \leq C{|v|}^C(e^{C\ve}e^{-\ve \log \ve})^{|v|}                                                                                  \\
            (\text{by (\ref{eq:Cve}) and (\ref{eq:phi_cF_control})})  & \leq C{(n+g(n))}^C e^{\eta(n+g(n))/20}                                                                                          \\
            (\text{by (\ref{eq-n(3)})})                               & \leq  e^{\eta {(n-g(n))}/16}                                                                                                    \\
            (\text{by (\ref{eq:phi_cF_control})})                     & \leq e^{\eta |v|/16} \leq e^{\eta n/8}.
        \end{split}
    \end{equation}
    By (\ref{eq-number-of-L}) and (\ref{eq-numberofpreimage}), we have
    \begin{equation}\label{eq-number-of-Ln}
        \left\lvert \mathcal{L}_n^{\alpha_j,\frac{\ve}{2}}\right\rvert
        \geq  \left\lvert \phi_{\mathcal{F}}(\mathcal{L}_n^{\alpha_j,\frac{\ve}{2}})\right\rvert
        \geq \frac{e^{n\left(h_{\alpha_j}(\sigma)-\eta/8\right)}}{e^{\eta n/8}}
        = e^{n\left(h_{\alpha_j}(\sigma)-\eta/4\right)}.
    \end{equation}
    We denote by $\Gamma_t^j \subseteq \phi_{\mathcal{F}}(\mathcal{L}_n^{\alpha_j,\frac{\ve}{2}}) \subseteq \cF$ the collection of words in $\phi_{\mathcal{F}}(\mathcal{L}_n^{\alpha_j,\frac{\ve}{2}})$ with length $t$.
    By the pigeonhole principle, there exists
    \begin{equation}\label{eq:t_j}
        n-g(n) \leq t_j \leq n + g(n)
    \end{equation}
    such that
    \begin{equation}\label{eq-number-of-fixlength}
        \left\lvert \Gamma_{t_j}^j\right\rvert  \geq\frac{\left\lvert \left\{\phi_{\mathcal{F}}(w):w \in \mathcal{L}_n^{\alpha_j,\frac{\ve}{2}}\right\} \right\rvert }{2g(n)+1}.
    \end{equation}
    Then
    \begin{equation}\label{eq:9.20}
        \begin{split}
                                                                                                   & \phantom{{}\geq{}} \left\lvert \Gamma_{t_j}^j\right\rvert                                   \\
            (\text{by (\ref{eq-n(1)}) (\ref{eq-number-of-Ln}) and (\ref{eq-number-of-fixlength})}) & \geq \frac{e^{n \left(h_{\alpha_j}(\sigma)-\eta/4\right)  }}{e^{\left(n-g(n)\right)\eta/4}} \\
            (\text{by (\ref{eq-n(1.1)}) and (\ref{eq:t_j})})                                       & \geq \frac{e^{(n+g(n))\left(h_{\alpha_j}(\sigma)-\eta/2\right)}}{e^{{t_j} \eta/4}}          \\
            (\text{by (\ref{eq:t_j})})                                                             & \geq \frac{e ^{{t_j}\left(h_{\alpha_j}(\sigma)-\eta/2\right)}}{e^{{t_j}\eta/4}}             \\
                                                                                                   & \geq e^{{t_j}\left(h_{\alpha_j}(\sigma)-3\eta/4\right)}.
        \end{split}
    \end{equation}

    Let $T=t_0 t_1\dots t_m$. By (\ref{eq:t_j}), we have
    \begin{equation}\label{eq:T}
        (n-g(n))^{m+1} \leq T \leq (n+g(n))^{m+1}
    \end{equation}
    For any $1 \leq k \leq T$, let
    \[P^*_k(\sigma)=\left\{v \in  \prod_{i=1}^{T/t_0}\Gamma_{t_0}^0 : y=\dots v v \dots \mathrm{\ has\ minimal\ positive\ period\ }k   \right\}\]
    where
    \[\prod_{i=1}^{T/t_0}\Gamma_{t_0}^0 = \left\{v = v_1\dots v_{T/t_0}: v_i \in \Gamma_{t_0}^0 \mathrm{\ for\ any\ } 1 \leq i \leq T/t_0\right\}.\]
    Note $\prod_{i=1}^{T/t_0}\Gamma_{t_0}^0 \subseteq \cF$ by the free concatenation property of $\cF$. By (\ref{eq-n0}), (\ref{eq-entropy-n}) and (\ref{eq-n(5)}),
    \begin{align*}
                                                     & \phantom{{}\geq{}} \sum_{k=1}^{\left\lfloor T/2\right\rfloor}   \left\lvert P^*_k(\sigma)\right\rvert                                                                              \\
                                                     & = \sum_{k=\widehat{N}+1}^{\left\lfloor T/2\right\rfloor} \left\lvert P^*_k(\sigma)\right\rvert  +\sum_{k=1}^{\widehat{N}}\left\lvert P^*_k(\sigma)\right\rvert                     \\
        (\text{by (\ref{eq-entropy-n})})             & < \sum_{k=\widehat{N}+1}^{\left\lfloor T/2\right\rfloor}e^{k\left(\htop(\sigma,\mathcal{L})+\eta/8\right)} +\sum_{k=1}^{\widehat{N}} \left\lvert \mathcal{L}_k\right\rvert         \\
                                                     & \leq \left\lfloor T/2\right\rfloor e^{\left\lfloor T/2\right\rfloor \left(\htop(\sigma,\mathcal{L})+\eta/2\right)} +\sum_{k=1}^{\widehat{N}} \left\lvert \mathcal{L}_k\right\rvert \\
        (\text{by (\ref{eq-n(5)}) and (\ref{eq:T})}) & \leq e^{T(\htop(\sigma,\mathcal{L})-\eta)}                                                                                                                                         \\
        (\text{by (\ref{eq:alpha_0})})               & < e^{T\left(h_{\alpha_0}(\sigma)-3\eta/4\right)}                                                                                                                                   \\
        (\text{by (\ref{eq:9.20})})                  & \leq  \left\lvert \prod_{i=1}^{T/t_0}\Gamma_{t_0}^0\right\rvert = \left\lvert \Gamma_{t_0}^0\right\rvert^{T/t_0}.
    \end{align*}
    Then $P_T^*(\sigma)\neq\emptyset$. For one-sided subshifts, we can prove the set
    $$P^*_T(\sigma)=\{v \in \prod_{i=1}^{T/t_0}\Gamma_{t_0}^0 : y= v v \dots \mathrm{\ has\ minimal\ period\ } T   \}$$
    is nonempty by a similar method.

    Take $\pi \in P_T^*(\sigma)$. Denote $\tilde{\pi}=\dots\pi\pi\dots$ and
    $$\Delta=\left\{w \in \mathcal{L}_{T}:w \text{ occurs in }\pi\pi \right\}.$$
    Then $ \left\lvert \Delta\right\rvert \leq T$. Let $ \Theta_{T}^j=\left(\prod_{i=1}^{T/t_j}\Gamma_{t_j}^j\right) \setminus \Delta \subseteq \cF$. Then
    \begin{equation}\label{eq-numberofGamma}
        \left\lvert \Theta_{T}^j\right\rvert
        \geq \left\lvert \prod_{i=1}^{T/t_j}\Gamma_{t_j}^j\right\rvert -T
        \overset{\text{by (\ref{eq:9.20})}}{\geq} e^{T\left(h_{\alpha_j}(\sigma)-3\eta/4\right)} - T
        \overset{\text{by (\ref{eq-n(4)}) and (\ref{eq:T})}}{\geq} e^{T\left(h_{\alpha_j}(\sigma)-9\eta/10\right)}.
    \end{equation}
    Fix $M \in \Nb^+$ such that
    \begin{equation}\label{eq-M}
        M \geq \max\left\{24/\ve , 20\max_{1 \leq j\leq m}\{h_{\alpha_j}(\sigma)\}/\eta\right\}.
    \end{equation}
    Denote
    $$ H^j_M:= \{ \pi\pi y^1 \dots y^{M-2} \in \cL: y^i \in \Theta_{T}^j \mathrm{\ for\ any\ } 1\leq i\leq M-2\}$$
    and
    $$ H_M:= \{\pi \pi y^1 \dots y^{M-2} \in \cL: y^i \in \bigcup_{j=1}^m \Theta_{T}^j \mathrm{\ for\ any\ } 1\leq i\leq M-2 \}. $$
    Here $H^j_M \subseteq H_M \subseteq \cF$ by the free concatenation property of $\cF$. Let $Y_j$ and $Y$ be the sets of points generated by words in $H^j_M$ and $H_M$ respectively, i.e.,
    \[Y_j =\left\{\dots w^{-1}w^0w^1\dots \in X:w^i \in H^j_M \mathrm{\ for\ any\ } i \in \Zb \right\} \]
    and
    \[Y =\left\{\dots w^{-1}w^0w^1\dots\in X:w^i \in H_M \mathrm{\ for\ any\ } i \in \Zb \right\} \]
    Note $(Y_j,\sigma^{MT})$ and $(Y,\sigma^{MT})$ are full shifts. By Lemma \ref{lemFullShiftsProperties}, they are topologically transitive and satisfy the shadowing property. Define
    $$\Lambda_j=\bigcup_{i=0}^{MT-1}\sigma^i(Y_j),\ \Lambda=\bigcup_{i=0}^{MT-1}\sigma^i(Y).$$
    $\Lambda_j$ and $\Lambda$ are compact and $\sigma$-invariant. Next, we show that $\Lambda_j$ and $\Lambda$ satisfy (1)-(3).

    (1)  We have the following claim.
    \begin{claim}
        $\sigma^k(Y_j)\cap\sigma^{k'}(Y_j)=\emptyset$ for any $0 \leq k < k' \leq MT-1$.
    \end{claim}
    \begin{proof}[Proof of the claim]
        Otherwise, there exists $\theta \in Y_j\cap\sigma^{k-k'}(Y_j)$.
        \begin{itemize}
            \item  If $1 \leq k'-k < T $,  let $\theta =\dots \theta_{-1}\theta_0\theta_1 \dots \in Y_j$ where $\theta_0 \dots \theta_{2T-1}=\pi\pi$.  Since $\sigma^{k'-k} \theta\in Y_j$, then $ \theta_{k'-k} \dots \theta_{k'-k+T-1}=\pi $. This implies $\sigma^{k'-k}\tilde{\pi}=\tilde{\pi}$, which contradicts $\pi \in P^*_T(\sigma)$.
            \item  If $T \leq k'-k < (M-1)T$, assume $k'-k=lT+p$ where $1 \leq l\leq M-2$ and $0 \leq p \leq T-1$. Since $\sigma^{k'-k}\theta\in Y_j$, we have $\theta_{lT+p}\dots \theta_{(l+2)T+p-1}=\pi\pi$. Since $\theta \in Y_j$, we have
                  $\theta_{(l+1)T}\dots \theta _{(l+2)T-1} \in \Theta_T^j$. This contradicts  $\Delta\cap \Theta_T^j=\emptyset$.
            \item If $k'-k=(M-1)T$, we have $\theta_{(M-1)T}\dots \theta_{MT-1}=\pi$ by $\sigma^{k'-k}\theta\in Y_j$. Since $\theta\in Y_j$, we have $\theta_{(M-1)T}\dots \theta_{MT-1} \in \Theta_T^j$. This contradicts $\Theta_{T}^j\cap \Delta=\emptyset$.
            \item If $(M-1)T<k'-k <MT$, we have $MT < k'-k+T<(M+1)T$. By $\theta\in Y_j$, we have $\theta_{MT} \dots \theta_{(M+2)T-1}=\pi\pi$. By $\sigma^{k'-k} \theta\in Y_j$, we have $ \theta_{k'-k+T} \dots \theta_{k'-k+2T-1}=\pi $. This  implies $\sigma^{k'-k-(M-1)T}\tilde{\pi}=\tilde{\pi}$, which contradicts $\pi \in P^*_T(\sigma)$.
        \end{itemize}
        Therefore, the claim is proved.
    \end{proof}

    For two-sided shifts, $\sigma$ is a homeomorphism. By Lemma \ref{lemma-f^n-f}, $(\Lambda_j, \sigma)$ is topologically transitive and satisfies the shadowing property. Then it is a topologically transitive subshift of finite type. By a similar method, $(\Lambda, \sigma)$ is also topologically transitive and satisfies the shadowing property. Then $(\Lambda, \sigma)$ is a topologically transitive subshift of finite type.

    For one-sided shifts, let $Y_j$ and $Y$ be the sets of points generated by words in $H^j_M$ and $H_M$ respectively, i.e.,
    \[Y_j =\left\{w^0w^1\dots\in X:w^i \in H^j_M \mathrm{\ for\ any\ } i \in \Nb \right\},\]
    \[Y =\left\{w^0w^1\dots\in X:w^i \in H_M \mathrm{\ for\ any\ } i \in \Nb \right\} \]
    and
    $$\Lambda_j=\bigcup_{i=0}^{MT-1}\sigma^i(Y_j),\ \Lambda=\bigcup_{i=0}^{MT-1}\sigma^i(Y).$$
    For any $u,v\in \mathcal{L}(\Lambda_j)$, there exists $x,y \in \Lambda_j$ such that $x_0\dots x_{|u|-1}=u$ and $y_0 \dots y_{|v|-1}=v$. Then there are $\tilde{x}, \tilde{y}\in Y_j$ and  $0 \leq i_1, i_2 \leq MT-1$ such that $x=\sigma^{i_1}\tilde{x}$ and $y=\sigma^{i_2}\tilde{y}$, which implies $\tilde{x}_{i_1}\dots \tilde{x}_{i_1+|u|-1}=u$ and $\tilde{y}_{i_2}\dots \tilde{y}_{i_2+|v|-1}=v$. There exist $0 \leq l \leq MT-1$ and $k\in\Nb^+$ such that $i_1+|u|+l=kMT$. Let
    $ w= \tilde{x}_{i_1+|u|} \dots \tilde{x}_{kMT-1} \tilde{y}_{0}\dots\tilde{y}_{i_2-1}$. Note $w$ is a subword of $w^{1}w^2 \in \mathcal{L}(Y_j)$ for some $w^{1}, w^2 \in \mathcal{L}(Y_j)$. We have $w \in \mathcal{L}(\Lambda_j)$, $|w|\leq 2MT$ and
    $$uwv=\tilde{x}_{i_1}\dots \tilde{x}_{i_1+|u|-1} \tilde{x}_{i_1+|u|}\dots \tilde{x}_{kMT-1} \tilde{y}_{0}\dots\tilde{y}_{i_2-1} \tilde{y}_{i_2}\dots \tilde{y}_{i_2+|v|-1}\in \mathcal{L}(\Lambda_j).$$
    Therefore, $(\Lambda_j,\sigma)$ satisfies (W)-specification property with gap length $2MT$. By a similar argument,
    $(\Lambda, \sigma)$ also satisfies (W)-specification property with gap length $2MT$.

    (2)
    \begin{align*}
        \htop\left(\sigma, \mathcal{L}(\Lambda_j)\right) & \geq \frac{1}{MT}\htop\left(\sigma^{MT},Y_j\right)                    \\
                                                         & = \frac{\log|\Theta_T^j|^{(M-2)}}{MT}                                 \\
        (\text{by (\ref{eq-numberofGamma})})             & \geq \frac{M-2 }{M}\left(h_{\alpha_j}(\sigma)-\frac{9\eta}{10}\right) \\
        (\text{by (\ref{eq-M})})                         & \geq h_{\alpha_j}(\sigma)-\eta.
    \end{align*}

    (3)
    Define
    \[\mathcal{F}^{*} := \bigcup_{m=1}^\infty\prod_{i=1}^m \cF = \bigcup_{m=1}^\infty\left\{\left(w^1 \dots w^m\right) : w^i \in \mathcal{F} \mathrm{\ for\ any\ } 1 \leq i \leq m\right\} \]
    and $\mathcal{L}^{*}$ similarly. Let $\phi :\mathcal{F}^{*} \rightarrow \cL$ be the concatenation map $ \left(w^1, \dots , w^m\right)  \rightarrow w^1\dots w^m$. Since $\cF$ has the free concatenation property, we have $\phi\left(\cF^{*}\right) = \cF$. We extend the map $\phi: \mathcal{F}^*\rightarrow \mathcal{F}$ to a map $\phi: \mathcal{L}^*\rightarrow \mathcal{F}$ in the following way. Given $\left(w^1, \dots, w^m\right) \in \mathcal{L}^*$, let $\phi\left(\left(w^1, \dots , w^m\right)\right)  = \phi_{\mathcal{F}}(w^1) \dots \phi_{\mathcal{F}}(w^m)$. Since $\cL$ is edit approachable by $\cF$ and $\cF$ has the free concatenation property, such $\phi$ is well-defined.

    We only deal with the case of two-sided subshifts since the case of one-sided shifts is similar. For any
    $\theta=\dots\theta^{-1}\theta^0\theta^1\ldots \in Y_j$ where $\theta^k=y^k_0 y^k_0 y^k_1 \dots y^k_{M-2} \in H^j_M$ where $y^k_0 = \pi$ for any $k \in \Zb$. Note $y^k_i \in \Theta_{T}^j$ and $\Theta_{T}^j=\left(\prod_{i=1}^{T/t_j}\Gamma_{t_j}^j\right) \setminus \Delta$. Since $\Gamma_{t_j}^j \subseteq \phi_{\mathcal{F}}(\mathcal{L}_n^{\alpha_j,\frac{\ve}{2}})$, there exists $\left\{w_{l}^{i,j,k}\right\}_{l=1}^{T/t_j}\subseteq \mathcal{L}^{\alpha_j,\frac{\ve}{2}}_n$ such that
    $$\phi\left( \left(w_1^{i,j,k},w_2^{i,j,k},\dots,w_{T/t_j}^{i,j,k}\right) \right)=\phi_{\mathcal{F}}\left(w_1^{i,j,k}\right) \phi_{\mathcal{F}}\left(w_2^{i,j,k}\right)\dots \phi_{\mathcal{F}}\left(w_{T/t_j}^{i,j,k }\right)=y_{i}^k$$
    for any $1\leq i\leq M-2$. By the definition of $\Gamma_t^j$, we have the length of $\phi_{\mathcal{F}}\left(w_{l}^{i,j,k}\right)$ is $t_j$ for any $1\leq i\leq M-2$, any $1 \leq l \leq T/t_j$ and any $k \in \Zb$.

    For any $1\leq i\leq M-2$ and any $1\leq l \leq T/{t_j}$, take $x^{i,j,k}_{l}\in [w^{i,j,k}_{l}]$. Let $b=kMT+2T$ and $d_{l,i}=T(i-1)+t_j( l-1)$. Then $\sigma^{b+d_{l,i}}(\theta) \in \left[\phi_{\mathcal{F}}\left(w_{l}^{i,j,k}\right)\right] $. We have
    \begin{align*}
                                                   & \phantom{{}\leq{}} D\left(\cE_{t_j}\left(\sigma^{b+d_{l,i}}(\theta)\right),\alpha_j\right)                                                                     \\
                                                   & \leq  D\left(\cE_{t_j}\left(\sigma^{b+d_{l,i}}(\theta )\right),\cE_n\left(x^{i,j,k}_{l}\right)\right) + D\left(\cE_n\left(x^{i,j,k}_{l}\right),\alpha_j\right) \\
        (\text{by (\ref{eq:D}) and (\ref{def:L})}) & \leq \delta_{n} + \ve / 2                                                                                                                                      \\
        (\text{by (\ref{eq-n(2)})})                & <  \ve / 4+ \ve / 2 = 3\ve/4.
    \end{align*}
    Then for any $1 \leq i \leq M-2$,
    \begin{equation}\label{eq-estimate-1}
        D\left(\cE_{T}\left(\sigma^{b+(i-1)T}\theta\right),\alpha_j\right)
        = D\left(\frac{t_j}{T}\sum_{l=1}^{T/{t_j}}\cE_{t_j}\left(\sigma^{b+d_{l,i}}(\theta)\right), \alpha_j\right) < 3\ve / 4.
    \end{equation}
    Therefore,
    \begin{equation}\label{eq-estimateofmeasure-i}
        \begin{split}
            D\left(\cE_{MT}\left(\sigma^{kMT}(\theta)\right) ,\alpha_j\right)
                                                                                  & \leq D\left(\cE_{MT}\left(\sigma^{kMT}(\theta)\right),\cE_{(M-2)T}\left(\sigma^{b}(\theta)\right)\right)                               \\
                                                                                  & + D\left(\cE_{(M-2)T}\left(\sigma^{b}(\theta)\right),\frac{1}{M-2}\sum_{l=1}^{M-2}\cE_{T}\left(\sigma^{b+(l-1)T}(\theta)\right)\right) \\
                                                                                  & + D\left(\frac{1}{M-2}\sum_{l=1}^{M-2}\cE_{T}\left(\sigma^{b+(l-1)T}(\theta)\right),\alpha_j\right)                                    \\
            (\text{by Lemma \ref{lemMeasuresDistance} and (\ref{eq-estimate-1})}) & < 4/M + 3\ve/{4}.
        \end{split}
    \end{equation}
    For any \(\nu_j\in\M(\sigma,\Lambda_j)\), by
    Lemma \ref{lemTau} and Remark \ref{rmk:tau-barycenter}, there exists
    \(\lambda_j\in\M(\sigma^{MT},Y_j)\) such that
    \[
        \nu_j=\tau_{MT}\lambda_j
        =\int_{Y_j}\mathcal E_{MT}(\theta)
        \,\dif\lambda_j(\theta).
    \]
    It follows from (\ref{eq-estimateofmeasure-i}) with \(k=0\)
    and the convexity of \(D\) that
    \[
        \begin{split}
            D(\nu_j,\alpha_j)
             & \leq \int_{Y_j}
            D\bigl(\mathcal E_{MT}(\theta),\alpha_j\bigr)
            \,\dif\lambda_j(\theta)              \\
             & < \frac{4}{M}+\frac{3\ve}{4}<\ve.
        \end{split}
    \]

    By the ergodic decomposition theorem, we have $D_{\mathrm{H}}\left(\left\{\alpha_j\right\} ,\M(\sigma, \Lambda_j)\right) < \ve<\xi$. Since $W$ is convex and $\Lambda_j \subset \Lambda$, we have $W \subseteq B\left(\M(\sigma, \Lambda), \ve\right)$.

    On the other hand, for any ergodic measure $\nu \in \M(\sigma, \Lambda)$, pick $\theta=\dots \theta^{-1}\theta^0\theta^1\dots \in G_\nu \cap Y$ where $\theta^k=y_0^ky_0^ky^k_1\dots y^k_{M-2} \in H_M$ and $y_0^k=\pi$ for any $k\in \mathbb{Z}$. Denote $Q_j^k = \left\{1 \leq i \leq M-2  :  y_i^k \in \Theta_T^j \right\}$ and $q_j^k = \left\lvert Q_j^k\right\rvert $. By the choice of $\rho_0$, we have $\Theta_T^{j_1} \cap \Theta_T^{j_2} = \emptyset$ for any $1 \leq j_1 \ne j_2 \leq m$. Then $\sum_{j=1}^{m}q_j^k=M-2$. We have
    \begin{equation}\label{eq-estimateofmeasure2-i}
        \begin{split}
            D\left(\cE_{MT}\left(\sigma^{kMT}\theta\right), \frac{\sum_{j=1}^{m}q_j^k\alpha_j}{\sum_{j=1}^{m}q_j^k}\right)
                                                        & \leq D\left(\cE_{MT}\left(\sigma^{kMT}\theta\right), \frac{1}{M-2}\sum_{i=1}^{M-2}\cE_T\left(\sigma^{b+(i-1)T}\theta\right)\right)              \\
                                                        & + D\left(\frac{1}{M-2}\sum_{i=1}^{M-2}\cE_T\left(\sigma^{b+(i-1)T}\theta\right), \frac{\sum_{j=1}^{m}q_j^k\alpha_j}{\sum_{j=1}^{m}q_j^k}\right) \\
            (\text{by Lemma \ref{lemMeasuresDistance}}) & \leq \frac{4}{M}+ \frac{1}{\sum_{j=1}^{m}q_j^k}\sum_{j=1}^{m}\sum_{i\in Q_j^k}D\left(\cE_{T}(\sigma^{b+(i-1)T}\theta),\alpha_j\right)           \\
            (\text{by (\ref{eq-estimate-1})})           & < \frac{4}{M}+\frac{3\ve}{4}.
        \end{split}
    \end{equation}
    For any $k \in \mathbb{Z}$, there exists a sequence of non-negative integers $\left\{q_i^k\right\}_{i=1}^m$ such that
    $$D\left(\cE_{MT}\left(\sigma^{kMT}\theta\right), \frac{\sum_{i=1}^{m}q_i^k\alpha_i}{\sum_{i=1}^{m}q_i^k}\right)\le 4/M + 3\ve/{4} < \ve$$
    by (\ref{eq-estimateofmeasure2-i}). Therefore, $\nu \in B(W, \ve)$. By the ergodic decomposition theorem, we have $\M(\sigma, \Lambda) \subseteq B(W, \ve)$. Thus, we obtain that $D_{\mathrm{H}}(W,\M(\sigma, \Lambda)) < \ve<\xi $. Moreover, we have $\Lambda \varsubsetneq X$. Otherwise, $D_{\mathrm{H}}(W,\M(\sigma, \Lambda)) = D_{\mathrm{H}}(W, \M(\sigma,X))\geq \ve$, which leads to a contradiction.

    It remains to return from the measures $\beta_i$ to the
    original measures $\overline{\alpha}_i$. Convexity of the
    Wasserstein distance gives
    \begin{equation}\label{eq:return-measures}
        D(\overline{\alpha}_i,\beta_i)\leq\lambda,
        \qquad
        D_{\mathrm H}(\overline W,W_\lambda)\leq\lambda.
    \end{equation}
    For any
    $\overline\kappa=\sum_{i=1}^m a_i\overline\alpha_i\in\overline W$,
    the corresponding measure
    \[
        \kappa_\lambda=\sum_{i=1}^m a_i\beta_i
        =(1-\lambda)\overline\kappa+\lambda\nu
    \]
    belongs to $W_\lambda$ and satisfies
    $D(\overline\kappa,\kappa_\lambda)\leq\lambda $. The same
    correspondence works in the opposite direction. Hence, the estimates
    obtained above and (\ref{eq:auxiliary-errors}) imply
    \begin{align*}
        D_{\mathrm H}\left(\left\{\overline\alpha_i\right\} ,
        \M(\sigma,\Lambda_i)\right)
         & \leq D(\overline\alpha_i,\beta_i)
        +D_{\mathrm H}\left(\left\{\beta_i\right\} ,
        \M(\sigma,\Lambda_i)\right)          \\
         & <\lambda +\xi_0<\overline\xi,
    \end{align*}
    and
    \begin{align*}
        D_{\mathrm H}\left(\overline W,
        \M(\sigma,\Lambda)\right)
         & \leq D_{\mathrm H}(\overline W,W_\lambda)
        +D_{\mathrm H}\left(W_\lambda,
        \M(\sigma,\Lambda)\right)                    \\
         & <\lambda +\xi_0<\overline\xi.
    \end{align*}
    Since
    $0\leq h_\omega(\sigma)\leq H$ for any
    $\omega\in\M(\sigma,X)$, we obtain
    \[
        h_{\beta_i}(\sigma)
        \geq h_{\overline\alpha_i}(\sigma)-\lambda H.
    \]
    By (\ref{eq:auxiliary-errors}), we have
    \begin{align*}
        \htop(\sigma,\mathcal L(\Lambda_i))
         & \geq h_{\beta_i}(\sigma)-\eta_0              \\
         & >h_{\overline\alpha_i}(\sigma)-\overline\eta
    \end{align*}
    and hence
    \[
        \htop(\sigma,\mathcal L(\Lambda))
        >\sup_{\overline\kappa\in\overline W}
        h_{\overline\kappa}(\sigma)-\overline\eta.
    \]
    The dynamical properties proved in item (1), as well as
    $\Lambda\varsubsetneq X$, are unaffected by the passage back to the
    original data.  If repeated measures were deleted at the beginning,
    assign to each repeated index the subsystem already constructed for
    the corresponding distinct measure.  This completes the proof.
\end{proof}

\begin{remark}
    As can be seen from the proof of the theorem, there exists $n \in \Nb^+$, $\Delta \subseteq \Lambda$ and $\Delta_i \subseteq \Lambda_i$ such that
    $\Lambda = \bigcup_{j=0}^{n-1} \sigma^j(\Delta)$, $\Lambda_i = \bigcup_{j=0}^{n-1} \sigma^j\left(\Delta_i\right) $ for any $1 \leq i \leq m$, and $\left(\Delta_i, \sigma^{n}\right)$ and $\left(\Delta, \sigma^{n}\right)$ are topologically conjugate to full shifts.
\end{remark}

Theorem \ref{prop-nonuniform-multi-horseshoe} provides the
non-uniform symbolic input needed to run the same nested construction.
We therefore obtain the following realization theorem.

\begin{theorem}\label{thm:cmed_nonuniform}
    Let $X$ be a one-sided or two-sided subshift on a finite alphabet with language $\mathcal{L}$. Suppose that $\mathcal{G} \subset \mathcal{L}$ has $(W)$-specification and $\mathcal{L}$ is edit approachable by $\mathcal{G}$. Let $\vphi \in C(X)$ such that $\mathrm{Int} L_\vphi \ne \emptyset$. Then for any $\ve > 0$, any $a \in \mathrm{Int} L_\vphi$, any $\mu \in \M(\sigma,X)$ satisfying $\int\vphi\dif\mu = a$ and $h_\mu(\sigma) < \sup\left\{h_\omega(\sigma): \omega \in \M(\sigma,X), \int \varphi \dif \omega = a\right\} $, and any $0 \leq h \leq h_\mu(\sigma)$, there exists a compact $\sigma$-invariant set $Y \subseteq X$ such that
    \begin{enumerate}
        \item $\left(Y,\sigma\right) $ is minimal and uniquely ergodic;
        \item $h_{\nu}(\sigma) = \htop(\sigma,Y) = h$ where $\nu$ is the unique $\sigma$-invariant measure on $Y$;
        \item $\int \vphi \dif \nu = a$;
        \item $D(\nu,\mu) < \ve$.
    \end{enumerate}
\end{theorem}
\begin{proof}[Sketch of proof]
    Fix $\ve>0$, $a\in\Int L_\varphi$, $\mu\in\M(\sigma,X)$ and
    $0\leq h\leq h_\mu(\sigma)$ as in the statement. Since
    \[
        h_\mu(\sigma)
        <
        \sup\left\{
        h_\omega(\sigma):
        \omega\in\M(\sigma,X),\
        \int\varphi\dif\omega=a
        \right\},
    \]
    we have $\htop(\sigma,X)>0$, and hence Theorem
    \ref{prop-nonuniform-multi-horseshoe} is applicable.

    We first choose the measures used in Step (1) of the proof of
    Theorem \ref{thmA}. By Lemma \ref{Lem-dense-periodic-measure},
    periodic measures are dense in $\M(\sigma,X)$. Thus, using convex
    combinations as in that proof, we can find two invariant measures
    $\mu_1^1,\mu_2^1$, sufficiently close to $\mu$, such that
    \[
        \int\varphi\dif\mu_1^1<a<
        \int\varphi\dif\mu_2^1
    \]
    and their entropies lie in the required interval above $h$.

    Apply Theorem \ref{prop-nonuniform-multi-horseshoe} to
    $\cov\{\mu_1^1,\mu_2^1\}$. By the remark following that theorem,
    there exist $N_1\in\Nb^+$ and compact sets
    $\Delta_1^1,\Delta_2^1\subseteq\Delta^1$ such that
    \[
        \Lambda_i^1=\bigcup_{j=0}^{N_1-1}\sigma^j(\Delta_i^1),
        \qquad
        \Lambda^1=\bigcup_{j=0}^{N_1-1}\sigma^j(\Delta^1),
    \]
    and $(\Delta_i^1,\sigma^{N_1})$ and
    $(\Delta^1,\sigma^{N_1})$ are conjugate to finite full shifts.
    Moreover,
    \[
        \htop(\sigma,\Lambda_i^1)
        =
        \frac{1}{N_1}\htop(\sigma^{N_1},\Delta_i^1).
    \]
    The Hausdorff estimates in Theorem
    \ref{prop-nonuniform-multi-horseshoe} ensure that the invariant
    measures supported on $\Lambda_1^1$ and $\Lambda_2^1$ have
    $\varphi$-integrals on the two opposite sides of $a$.

    The remainder of the proof is the same as Steps (2)--$(k)$ and the
    argument after the countable steps in the proof of Theorem
    \ref{thmA}. Indeed, after Step (1), the construction takes place
    inside compact finite full shifts. Lemma \ref{lemTau} transfers
    invariant measures, entropy and the $\varphi$-integral between
    $(\Delta^k,\sigma^{N_k})$ and
    $\bigcup_{j=0}^{N_k-1}\sigma^j(\Delta^k)$. We therefore obtain a
    decreasing sequence of compact invariant sets $\{\Lambda^k\}$ such
    that
    \[
        \diam\M(\sigma,\Lambda^k)\longrightarrow0,
        \qquad
        \htop(\sigma,\Lambda^k)\longrightarrow h,
    \]
    while the corresponding integrals converge to $a$. Consequently,
    $\bigcap_{k\geq1}\Lambda^k$ supports a unique invariant measure
    $\nu$ satisfying
    \[
        h_\nu(\sigma)=h,\qquad
        \int\varphi\dif\nu=a,\qquad
        D(\nu,\mu)<\ve.
    \]
    Taking $Y=\supp(\nu)$, we obtain that $(Y,\sigma)$ is minimal and
    uniquely ergodic and that
    $\htop(\sigma,Y)=h_\nu(\sigma)=h$.
\end{proof}

\bigskip
\section*{Acknowledgements}
X. Hou is supported by the National Natural Science Foundation of China No. 12401231 and the Fundamental Research Funds for the Central Universities No. DUT25RC(3)106. W. Lin is supported by the National Natural Science Foundation of China (No. 124B2010). X. Tian is supported by the National Natural Science Foundation of China (No. 12471182) and Natural Science Foundation of Shanghai (No. 23ZR1405800). X. Zhao is supported by the National Key R\&D Program of China (No. 2021YFA1001900).

\bibliographystyle{plain}
\bibliography{ref}

\end{document}